\documentclass[reqno,12pt]{amsart}
\usepackage[margin=1in,footskip=0.5in]{geometry}

\usepackage{enumitem}

\usepackage[latin1]{inputenc}
\usepackage[english]{babel}
\usepackage{amsaddr}
\usepackage{subfigure}
\usepackage{amsmath,amssymb,amsthm}
\usepackage{ae}
\usepackage[T1]{fontenc}
\usepackage{graphicx}
\usepackage{epstopdf}
\usepackage{color}

\newcommand{\Wert}{{\vert\kern-0.25ex\vert\kern-0.25ex\vert}}
\renewcommand{\i}{{\mathrm i}}
\newcommand{\e}{{\mathrm e}}

\newcommand{\be}{\begin{eqnarray}}
\newcommand{\ee}{\end{eqnarray}}
\newcommand{\ba}{\begin{align}}
\newcommand{\ea}{\end{align}}
\newcommand{\bi}{\begin{itemize}}
\newcommand{\ei}{\end{itemize}}
\usepackage{fullpage}

\newcommand{\secref}[1]{Section~\ref{sec:#1}}

\newcommand{\seclab}[1]{\label{sec:#1}}
\newcommand{\eqlab}[1]{\label{eq:#1}}
\renewcommand{\eqref}[1]{(\ref{eq:#1})}

\newcommand{\figref}[1]{Fig.~\ref{fig:#1}}
\newcommand{\figlab}[1]{\label{fig:#1}}
\newcommand{\propref}[1]{Proposition~\ref{proposition:#1}}
\newcommand{\proplab}[1]{\label{proposition:#1}}
\newcommand{\lemmaref}[1]{Lemma~\ref{lemma:#1}}
\newcommand{\lemmalab}[1]{\label{lemma:#1}}
\newcommand{\remref}[1]{Remark~\ref{remark:#1}}
\newcommand{\remlab}[1]{\label{remark:#1}}
\newcommand{\thmref}[1]{Theorem~\ref{theorem:#1}}
\newcommand{\thmlab}[1]{\label{theorem:#1}}

\newcommand{\appref}[1]{Appendix~\ref{app:#1}}

\newcommand{\applab}[1]{\label{app:#1}}

\usepackage{graphicx}

\newtheorem{theorem}{Theorem}[section]
\newtheorem{proposition}[theorem]{Proposition}

\newtheorem{lemma}[theorem]{Lemma}

\newtheorem{remark}[theorem]{Remark}

\numberwithin{equation}{section}

\title{A geometric approach to inner problems associated with exponentially small splitting phenomena in local bifurcations}
\author{K. Uldall Kristiansen}
\address{Department of Applied Mathematics and Computer Science,
Technical University of Denmark,
2800 Kgs. Lyngby,
Denmark }
\begin{document}
 \begin{abstract}

 In this paper, we study analytic nonlinear partial differential equations for $\mathbf y=\mathbf y(x,\theta)\in \mathbb C^n$ of the form $x^2 \mathbf y'_x(1+\mathcal O(x)) + \mathbf y'_\theta+x \mathbf A \mathbf y = \mathcal O(x^2)$, $()'_z=\frac{\partial}{\partial z}$, $z=x,\theta$, with $x\in \mathbb C$ and $\theta\in \mathbb C /(2\pi \mathbb Z)$ denoting the independent variables. We show that solutions are $1$-sums (with respect to $x$) of Fourier series with coefficients of type Gevrey-$1$. The motivation for studying these equations is that so-called inner problems, associated with two-dimensional formal connections in unfoldings of local bifurcations, can be brought into this form (by looking for invariant manifolds in blowup coordinates). In the present paper, we give two examples: The zero-Hopf bifurcation and the resonant Hopf-Hopf bifurcation, both in the reversible settings. Importantly, these inner problems are given by the unperturbed problem (i.e. at the bifurcation) and the invariant manifolds are expressed directly in phase space (through blowup coordinates associated with the local bifurcation). This contrasts the Lazutkin-based formulation of inner problems which is based upon a blowup of singularities (with respect to complex time) of approximate solutions; for local bifurcations this requires (artificial) coordinate transformations. We solve the PDE by extending the Banach-convolution-algebra-approach to Borel-Laplace by Bonckaert and De Maesschalck (2008) to account for analytic Fourier series. In further details, we apply the Borel transform with respect to $x$ (keeping $\theta$ fixed) and solve the resulting equation in an appropriate Banach space of Fourier series with coefficients that have exponential growth in the Borel plane. The solutions of the PDE are then obtained through application of the Laplace transformation. %We believe that this approach is interesting in itself.
 We see our results as prerequisites for the (forthcoming) phase space-analysis  of exponentially small phenomena associated with local bifurcations, without following Lazutkin's approach. %We believe that inner problems date back to the work of Lazutkin on exponentially small splitting in the standard map. Nowadays, inner problems play a central role in exponentially small phenomena in analytic ODEs and certain PDEs. Following Lazutkin they are derived by zooming in on singularities of approximative solutions and the general hypothesis is that these problems describe the invariant manifolds near these singularities. In this paper, we show that the inner problems in our local bifurcations are equivalent to the unperturbed problems (i.e. at the bifurcations). We obtain invariant manifolds of these unperturbed system in phase space by writing the system in our normal form PDE. Blowup? Phase space methods?% to pursue phase space methods for exponentially small phenomena, using blowup as the main technical tool. %Although this is a relatively simple fact, we believe that it has important consequences.  %We believe that this identification of the inner problem as the unperturbed problem



 \end{abstract}
 \bigskip
\smallskip

% \maketitle
\maketitle
\tableofcontents
% \note{use $[]$ for maps (like $\mathcal L[\widehat W]$}
\section{Introduction}
In this paper, we consider the following partial differential equation for $\mathbf y=\mathbf y(x,\theta)$:
% \begin{align}\eqlab{main}
%  s
% \end{align}
\begin{align}\eqlab{main}
 x^2 \mathbf y'_x (1+xF_0+x^2 F_1(x,\mathbf y,\theta))+\mathbf y'_\theta+ x \mathbf A \mathbf y = x^2 \mathbf G_1(x,\mathbf y,\theta),\quad x\in B_\kappa,\,\theta\in \mathbb T_\zeta,
\end{align}
with $()'_z=\frac{\partial}{\partial z}$, $z=x,\theta$.
% $\mathbf y = \mathbf y(x,\theta)$ of
Here $F_0\in \mathbb C$, and
\begin{align}\eqlab{Adiag}
 \mathbf A=  \operatorname{diag}(\lambda^1,\ldots,\lambda^n)\in \mathbb R^{n\times n},\quad \forall\,j\in \{1,\ldots,n\}\,:\,\lambda^j>0.
\end{align}
Importantly, $$F_1:B_\kappa\times B_{\kappa}^n\times \mathbb T_\zeta\rightarrow \mathbb C,\quad \mbox{and}\quad \mathbf G_1:B_\kappa\times B_\kappa^n\times \mathbb T_\zeta\to \mathbb C^n,$$
 with $B_\kappa^n\subset \mathbb C^n$, $n\in \mathbb N$, denoting the open ball of radius $\kappa>0$ centered at the origin, are assumed to be analytic functions. For simplicity, we write $B_\kappa:=B_\kappa^1$ throughout. Finally, we note that $$\mathbb T_\zeta:=\left\{\theta\in \mathbb C/(2\pi \mathbb Z)\,:\,0\le \operatorname{Im}(\theta)\vert <\zeta\right\},$$ with $\zeta>0$.
% denotes $\mathbb T=\mathbb R/(2\pi\mathbb Z)$ extended into the complex plane such that $\theta\in \mathbb T_\xi$ if and only if $\operatorname{Re}(\theta)\in \mathbb R/(2\pi \mathbb Z)$ and $0\le \vert \operatorname{Im}\theta \vert<\xi$.
%  The invariant manifolds satisfy the PDE:

We emphasize that solutions of \eqref{main} define invariant manifold solutions of
\begin{equation}\eqlab{mainvf}
\begin{aligned}
 \dot x &=   x^2 (1+xF_0+ x^2 F_1(x,\mathbf y,\theta)),\\
 \dot{\mathbf y} & =  x\left(- \mathbf A {\mathbf y} +x \mathbf G_1(x,{\mathbf y},\theta)\right),\\
 \dot \theta &=1.
\end{aligned}
\end{equation}
The main motivation for looking at \eqref{mainvf} is that such systems appear as so-called \textit{inner problems} associated with exponentially small phenomena related to formal connections in formal normal forms of local bifurcations. In this paper, we will give two main examples: The two-dimensional formal connections in the zero-Hopf bifurcation and the resonant Hopf-Hopf bifurcation. We believe that the reversible setting is the interesting case and will therefore restrict attention to reversible systems.

% \note{read Inma's inner problem paper}
Our first main result on \eqref{main} is the following:
\begin{theorem}\thmlab{main}
Let $S^{\pm}(\delta,\chi)\subset \mathbb C$ denote the local open sectors, centered along the positive ($+$) respectively negative $(-$) real axis, of radius $\delta>0$ and opening $\pi+\chi\in (\pi,2\pi)$:
\begin{align}\eqlab{Spm}
 S^{\pm}(\delta,\chi)=\left\{x\in \mathbb C\,:\,0< \vert x\vert< \delta\, \mbox{ and }\, \vert \operatorname{Arg}(\pm x)\vert<  \frac{\pi+\chi}{2}\right\},
%  S_{-,\theta,\rho}&=\left\{x\in \mathbb C\,:\,0< \vert x\vert< \rho\,\, \mbox{and}\,\, \vert \operatorname{Arg}(-x)-\pi\vert<  \frac{\theta}{2}\right\}
\end{align}
see \figref{Spm}.
% Let $S^{\pm}=S^\pm(\nu,\pi+\eta) \subset \mathbb C$ denote the sectors centered along the positive (negative, respectively) real axis and with openings that are greater than $\pi$
Fix any $\chi\in (0,\pi)$. Then for $\delta>0$, $\xi>0$, both small enough, there
exists two analytic solutions:
\begin{align*}
 \mathbf y^\pm\,:\,S^\pm \times \mathbb T_\xi \to \mathbb C^n,
\end{align*}
respectively,
 of \eqref{main}. %Moreover, $\mathbf y^\pm$ are unique.
 In particular, upon writing $\mathbf y^\pm$ as Fourier series:
\begin{align*}
 \mathbf y^\pm(x,\theta) =: \sum_{\alpha\in \mathbb Z} \mathbf y_\alpha^\pm(x) \e^{i\alpha\theta},
\end{align*}
respectively,
we have that the Fourier coefficients
\begin{align*}
\mathbf y_\alpha^\pm(x) =\mathcal O(x\e^{-\vert \alpha\vert\xi}) ,\quad \alpha\in \mathbb Z,\,x\in S^\pm,
\end{align*}
(the estimate being uniform)
are $1$-sums of a Gevrey-$1$ series:
\begin{align}\eqlab{gevrey1series}
%  \mathbf y^\pm(x,\theta) \sim_1 \sum_{\alpha \in \mathbb Z}\mathbf y_\alpha(x)   \e^{i\alpha \theta},\quad \mathbf y_\alpha(x)
\mathbf y_\alpha(x) \sim_1 \sum_{\beta=1}^\infty \mathbf y_{\alpha,\beta} x^\beta\quad \forall\,\alpha\in \mathbb Z,
\end{align}
(the series being independent of $\pm$)
in the directions defined by $S^\pm$, respectively.
% with
% \begin{align*}
%  \vert \mathbf y_\alpha^\pm(x)\vert \le c_1 \e^{-
% \end{align*}

\end{theorem}

\begin{figure}[h!]
\begin{center}
% \subfigure[$\sigma<0$]{\includegraphics[width=.45\textwidth]{x2rho2_neg.pdf}}
{\includegraphics[width=.65\textwidth]{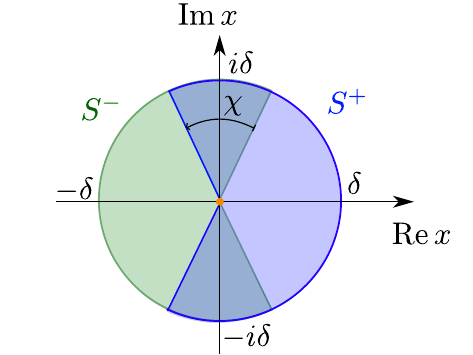}}
% \subfigure[$\sigma>0$]{\includegraphics[width=.45\textwidth]{x2rho2_pos.pdf}}
\end{center}
\caption{Illustration of the local sectorial domains $S^\pm$ (with radius $\delta>0$ and opening $\pi+\chi\in (\pi,2\pi)$.}
\figlab{Spm}
% Remark: If $U$ smooth for $|u|<\nu$, the compact manifolds lie inside $|u|<\nu$ for large $n$
\end{figure}

Recall that a Gevrey-$1$ series $\widetilde W\in \mathbb C[[x]]$ is
 said to be $1$-summable in the direction $\varphi\in \mathbb T:=\mathbb R/(2\pi \mathbb Z)$ if (a) the Borel transform $w\mapsto \widehat W(w)\in \mathbb C,\,w\in B_\kappa\subset \mathbb C$, of the series (which is a convergent series, see \cite{balser2000a} and \remref{borellaplace} below) can be endlessly continued to an analytic function $\widehat W^\varphi$, along the ray defined by $w=r\e^{i\varphi}$, $r\ge 0$, and (b) it is of at most exponential growth of order $1$ along this ray, i.e.
\begin{align*}
 \e^{-\eta \vert w\vert} \vert \widehat W^\varphi(w)\vert =\mathcal O(1) \quad \mbox{for}\quad w\to \infty\e^{i\varphi},
\end{align*}
for some $\eta>0$ large enough. Then the Laplace-transform $W^\varphi:=\mathcal L^{\varphi}[\widehat W^\varphi]$ is well-defined on a local sector $S^\varphi$ centered along $\varphi$ and with opening greater than $\pi$. $W^\varphi$ is called the $1$-sum of $\widetilde W$ in the direction $\varphi$ and it is unique (due to the opening being greater than $\pi$, see Watson's lemma \cite[Proposition 11]{balser2000a}). In particular, $W^\varphi$ is Gevrey-$1$ asymptotic to $\widetilde W$ (which we write as $W^\varphi\sim_1\widetilde W$) on $S^\varphi$. For further details, we refer to \cite{balser2000a}, see also \secref{borellaplace}. %said to be $1$-summable if its Borel transform (which is a convergent series) can be endlessly continued along an infinite sector with direction $\varphi$ with at most exponential growth. The resulting function can then be Laplace transformed and the resulting function, which is well-defined on a local sectorial region (like $S^\pm$) centered along $\varphi$ with opening greather than $\pi$, is called the $1$-sum.

Now, the formal Gevrey-$1$ series in \eqref{gevrey1series} is in general divergent (since the manifolds will be center-like manifolds) and the difference
\begin{align}
\Delta \mathbf y (x,\theta):= \mathbf y^+(x,\theta)-\mathbf y^-(x,\theta)\sim_1 0,\quad x\in S^+\cap S^-,
\end{align}
uniformly with respect to $\theta\in \mathbb T_\xi$,
is therefore generically nonzero, see e.g. \cite{balser2000a,sauzin2015}. (The divergence can be understood in different ways but one possible characterization is through singularities of the Borel transform, see \cite{costin2009,kristiansen2025a}.)
Notice that $S^+\cap S^-$ is the union of two sectors of the complex plane centered along the directions defined by $\pm \frac{\pi}{2}$, each with opening $\chi\in (0,\pi)$.

In the following, we define
\begin{align*}
q^{-\mathbf A}:=\operatorname{diag}(q^{-\lambda^1},\cdots,q^{-\lambda^n})\in \mathbb C^{n\times n},\quad q\ge 0,
\end{align*}
recall \eqref{Adiag}.
% In our next result, we therefore suppose that the series \eqref{gevrey1series} with $\alpha=-1$ and $\alpha=1$:
% \begin{align*}
% \sum_{\beta=1}^\infty \mathbf y_{-1,\beta} x^\beta,\quad \sum_{\beta=1}^\infty \mathbf y_{1,\beta} x^\beta
% % \sum_{\beta=1}^\infty \mathbf y_{\alpha,\beta} x^\beta \quad \mbox{for} \quad \alpha\in \{-1,1\},
% \end{align*}
% are divergent series. Then we have the following:
\begin{theorem}\thmlab{main2}
Consider $x=ip$ with $p\in [0,p_0]$ with $p_0>0$ small enough. Then there exists a constant vector
\begin{align*}
 \mathbf C_{-1}\in \mathbb C^n,
\end{align*}
such that
\begin{align}
\Delta \mathbf y(ip,\theta) = \e^{-\frac{1}{p}-iF_0\log p}p^{-\mathbf A} \left(\mathbf C_{-1}\e^{-i\theta}+ p \mathbf R(p,\theta)\right)\quad \forall\,p\in (0,p_0],\,\theta\in \mathbb T_\xi.\eqlab{Deltay}
\end{align}
Here $\mathbf R\,:\,[0,p_0]\times \mathbb T_\xi\to \mathbb C^n$ is $C^\infty$-smooth with respect to $p$ and analytic with respect to $\theta\in \mathbb T_\xi$. In particular,
\begin{align*}
 \mathbf R(p,\theta)  = \sum_{\alpha\in \mathbb Z} \mathbf R_\alpha (p)\e^{i\alpha\theta},
\end{align*}
with $\vert \mathbf  R_\alpha(p)\vert = \mathcal O(\e^{-\xi \vert \alpha\vert})$ uniformly with respect to $p$.
%  There exists two invariant manifold solutions:
\end{theorem}
There is obviously a similar expression for the difference for $p<0$. It takes the same form as \eqref{Deltay}  after replacing $(p,F_0,\theta,\mathbf C_{-1})$ by $(-p,-F_0,-\theta,\mathbf C_{1})$. In particular, if the equations are real-analytic (which is often the interesting case), then the expression for $p<0$ is simply obtained from \eqref{Deltay} by conjugation (so that $\mathbf C_{1}=\overline{\mathbf C}_{-1}$). Moreover, if \eqref{gevrey1series} with $\alpha=-1$ is convergent then $\Delta \mathbf y_{-1} = 0$. Hence $\mathbf C_{-1} \ne 0$ implies that \eqref{gevrey1series} with $\alpha=-1$ is divergent.  It seems likely that the other direction does not hold true in general, but we are not able to address this problem at the moment. %so that ``if \eqref{gevrey1series} with $\alpha=-1$ is divergent then $\mathbf C_{-1} \ne 0$''.

We will study \eqref{main} using Borel-Laplace, extending the Banach-convolution-algebra-approach of \cite{bonckaert2008a}. We believe that Borel-Laplace is the natural framework in problems with exponentially small slitting because the beyond all order phenomena are intrinsically related to the divergence of formal series solutions. Borel-Laplace deals with the resummation of such series.

\subsection{A new perspective on inner problems}
% Another motivation is also to give a new perspective on inner problems.
In this paper, we also use \eqref{main} as a platform to promote a new perspective (initiated in \cite{bkt}) on inner problems and more broadly on exponentially small splitting phenomena in local bifurcations.

Inner problems are known to be central in the rigorous description of exponentially small phenomena in analytic singular perturbation problems. Following the work of Lazutkin \cite{gelfreich1999a,lazutkin2005a}, these problems are ``suitable approximations'' of (invariant manifold) solutions near singularities of unperturbed (or approximate) solutions, see also \cite{baldoma2012a}.
They are typically derived through a blowup of poles and appropriate scalings of the variables.

For the purpose of this introduction, we will explain such derivation within the context of the real-analytic unfolding of the zero-Hopf bifurcation with the normal form:
\begin{equation}\eqlab{zerohopf0}
% \begin{equation}\eqlab{model0}
\begin{aligned}
 x' &=-x^2+\epsilon-a (y^2 +z^2) + F(x,y,z,\epsilon),\\
 y'&= b x y -z+G(x,y,z,\epsilon),\\
 z'&= b x z+ y+H(x,y,z,\epsilon),
\end{aligned}
% \end{equation}
\end{equation}
see \cite[Eq. (6)]{baldom2013a} and \cite[Eq. (1.1)]{bkt}. We will then subsequently explain a different perspective on the inner problem.

Notice in comparison with \cite[Eq. (1.1)]{bkt} that we have replaced their $(x,y,z,\mu)$ by $(-x,z,y,\epsilon)$  in \eqref{zerohopf0}. Moreover, we restrict attention to the reversible case (setting $\sigma=0$ in \cite{bkt}), assuming that
\begin{align*}
\begin{cases}
 F(-x,-z,-y,\epsilon)=F(x,y,z,\epsilon)\\
 G(-x,-z,-y,\epsilon)=H(x,y,z,\epsilon)
 \end{cases}\quad \forall\,(x,y,z)\in B_\tau^3, \,\epsilon\in (-\epsilon_0,\epsilon_0),
\end{align*}
with $\tau>0$, $\epsilon_0>0$, both small enough. It then follows that $(x(t),y(t),z(t))$ is a solution if and only if $(-x(-t),-z(-t),-y(-t))$ is a solution. We let
\begin{align}
 \mathcal S\,:\,(x,y,z)\mapsto (-x,-z,-y),\eqlab{symmetry0}
\end{align}
denote the associated symmetry.
 The functions $F,G$ and $H$ are assumed to be real-analytic, defined in a neighborhood of $(x,y,z,\epsilon)=(0,0,0,0)$ in $\mathbb C^4$, and each function is third order:
\begin{align}\eqlab{Wthird}
 W(x,y,z,\epsilon)  = \mathcal O(\vert (x,y,z,\epsilon)\vert^3),\quad W=F,G,H,
\end{align}
with respect to $(x,y,z,\epsilon)\to (0,0,0,0)$. We are mainly interested in $$a>0,\quad b>0.$$
The parameter $\epsilon\sim 0$ is the unfolding parameters. (Recall that the general non-reversible case has two unfolding parameters, see \cite{baldom2013a}). For $\epsilon=0$, we have a locally unique singularity at the origin with eigenvalues $0,\pm i$ of the linearization.
% Following \cite{}, we consider a region of parameter space defined by
% \begin{align}\eqlab{parameterbu}
%  \begin{cases}
%   \mu = \epsilon^2,\\
%   \nu = \epsilon b \sigma,
%  \end{cases}\quad 0\le \epsilon\ll 1,\,\sigma\in (-1,1).
% \end{align}

Consider now the change of coordinates $(x_2,y_2,z_2)\mapsto (x,y,z)$ defined by
%
% First, the authors define $(x_2,y_2,z_2)$ by
\begin{align}\eqlab{x2y2z20}
\begin{cases} x=r_2 x_2, \\ y=r_2 y_2, \\ z=r_2 z_2,\\
\epsilon = r_2^2.
\end{cases}
 \end{align}
%  where
%  \begin{align*}
%   r_2 = \sqrt{\epsilon}.
%  \end{align*}
%
% \begin{align}
%  \begin{cases}
%   s
%  \end{cases}
% \end{align}
This brings the system into the following slow-fast form
% \begin{align}
  \begin{equation}\eqlab{zerohopf20}
\begin{aligned}
 \dot x_2 &=-x_2^2+1+ a (y_2^2+z_2^2)+r_2 F_2(x_2, y_2, z_2,r_2),\\
 r_2 \dot y_2&=r_2 b x_2y_2- z_2+r_2^2 G_2(x_2, y_2, z_2,r_2),\\
 r_2 \dot z_2&= r_2 b x_2 z_2+y_2+r_2^2 H_2(x_2,y_2,z_2,r_2),
\end{aligned}
\end{equation}
where $\frac{d}{dt}=\dot{(\cdot)}=r_2^{-1} (\cdot )'$
and
\begin{align*}
 W_2(x_2, y_2, z_2,r_2^2):=r_2^{-3} W(r_2 x_2,r_2  y_2,r_2  z_2,r_2^2),\quad W=F,G,H,
\end{align*}
which are well-defined by \eqref{Wthird}.
For $r_2\to 0$, we obtain the reduced problem:
\begin{align}\eqlab{reduced}
 \dot x_2 &= -x_2^2+1,
\end{align}
on the normally elliptic critical manifold defined by $(y_2,z_2)=(0,0)$. The reduced problem \eqref{reduced} has a homoclinic orbit $\gamma_0$ with the following parametrization:
\begin{align}\eqlab{unperturbed_sol}
x_2 = \tanh (t),\quad t\in \mathbb R.
\end{align}
For $b>0$, a simple calculation (based upon the implicit function theorem) shows that there are two saddle-focus equilibria $E_2^\pm(r_2)$ near $(\pm 1,0,0)$ for all $0<r_2\ll 1$. However, for $r_2>0$ but small, the unperturbed homoclinic $\gamma_0$ will in general break up, see \cite{baldom2013a}. The splitting is beyond all orders and exponentially small with respect to $r_2\to 0$ (i.e. $\epsilon\to 0$ cf. \eqref{x2y2z20}).

The central observation of Lazutkin is that the splitting of the invariant manifolds is not beyond all orders near the  singularities $t_2=\pm \frac{i\pi}{2}$ (closest to the real axis) of the unperturbed connection. Although Lazutkin's work centered around the standard map, this framework has subsequently proven to be very powerful for the description of exponentially small phenomena in a range of different analytic differential equations, see e.g. \cite{MR4455359,MR4621957,MR4940205,MR4743478,gaivao2011a,MR4892796} and references therein.
%pursued this characterization within the context of maps, see \cite{}.)

To describe the splitting of the one-dimensional invariant manifolds in the zero-Hopf bifurcation, the authors of \cite{baldom2013a} derive an inner problem in the following way: %The unperturbed solution \eqref{unperturbed_sol} has a pole at $t=\frac{i\pi}{2}$.
Let $s$ be defined as $$t=t(s):=\frac{i\pi}{2}+r_2 s.$$ Notice that this is a blowup of $t=\frac{i\pi}{2}$ for $r_2=0$. Through \eqref{unperturbed_sol} and \eqref{x2y2z20} we obtain a change of coordinates $(s,y,z)\mapsto (x_2,y_2,z_2)$ defined by
\begin{align}\eqlab{x2ts}
 x_2 = \tanh (t(s)) = {\operatorname{coth}(r_2 s)}=\frac{1}{r_2 s} \left(1+\mathcal O(r_2)\right),\quad s\ne 0,
\end{align}
 which brings the equations \eqref{zerohopf0} into the following form:
\begin{equation}\eqlab{inner_temp}
\begin{aligned}
 \left(-r_2^2 \operatorname{coth}^2(r_2 s)-a(y^2+z^2)+ F\right)\frac{dy}{ds}&=-r_2^2 \operatorname{csch}^2(r_2 s)\left(-z+b{r_2 \operatorname{coth}(r_2 s)}y + G\right),\\
 \left(-r_2^2 \operatorname{coth}^2(r_2 s)-a(y^2+z^2)+ F\right)\frac{dz}{ds}&=-r_2^2 \operatorname{csch}^2(r_2 s)\left(y+b{r_2 \operatorname{coth}(r_2 s)}z + H\right).
\end{aligned}
\end{equation}
The inner problem of \cite[Eq. (33)]{baldom2013a} is then the $r_2 \to 0$ limit of \eqref{inner_temp}:
\begin{equation}\eqlab{inner_fucked}
\begin{aligned}
 \left(-s^{-2} -a(y^2+z^2)+ F\right)\frac{dy}{ds}&=-s^{-2} \left(-z+bs^{-1}y + G\right),\\
 \left(-s^{-2} -a(y^2+z^2)+ F\right)\frac{dz}{ds}&=-s^{-2} \left(y+bs^{-1}z + H\right),
\end{aligned}
\end{equation}
with $s\ne 0$ and
where $W=W(s^{-1},y,z,0)$, $W=F,G,H$. Here we have used that
\begin{align}\eqlab{sech_limit}
 r_2 \operatorname{coth}(r_2 s) \to s^{-1},\quad r_2 \operatorname{csch}(r_2 s) \to s^{-1},
\end{align}
as $r_2\to 0$, $s\ne 0$.
In \cite[Theorem 4]{baldom2013a} the authors are concerned with invariant manifolds of \eqref{inner_fucked} for $s\to\pm \infty$ in certain sectors of the complex plane.

We now notice the following simple fact:  \textit{The system \eqref{inner_fucked} is equivalent with \eqref{zerohopf0} for $\epsilon = 0$:
% \begin{align}
\begin{equation}\eqlab{inner}
% \begin{equation}\eqlab{model0}
\begin{aligned}
 x' &=-x^2-a (y^2 +z^2) + F(x,y,z,0),\\
 y'&=b x y -  z+G(x,y,z,0),\\
 z'&= b x z +  y+H(x,y,z,0),
\end{aligned}
% \end{equation}
\end{equation}
% \end{align}
through the change of coordinates defined by $x=s^{-1}$.} Notice in particular, that by \eqref{unperturbed_sol}
$x_2(t(s)) = \frac{1}{r_2 s} +\mathcal O(1)$
so that $$x=s^{-1} \quad \mbox{as}\quad r_2=\sqrt{\epsilon} \to 0, \,s\ne 0,$$ recall \eqref{x2y2z20}.

In conclusion, the inner problem associated with \eqref{zerohopf0} is just the unperturbed system. This will also be the case in our subsequent examples, and we believe that this a general phenomena in local bifurcation problems. (At present, we will not attempt to make this into a rigorous statement. We also remark that this connection between the inner problem and the unperturbed system has also been noted (but not exploited) by experts in the field, see e.g. \cite[Remark 2.9]{baldoma2018a} and \cite[p. 518]{gelfreich2001a}.) Therefore it seems unnatural to follow the approach of Lazutkin for these kind of local bifurcation problems. %Another issue with using $t\in \mathbb C$ as a cooordinate is that $x_2(t)\in (-1,1)$ for $t\in \mathbb R$; therefore if the equilibria $E_2^\pm$ move outside this interval for $\epsilon>0$, then adjustments have to be made (see \cite[]{}).
% \note{solution vs phase space}

Notice that for \eqref{inner}, we look for invariant manifolds in the $(x,y,z)$-phase space as graphs over $x\in S^{\pm}$. The existence of such manifolds follows from the theory of generalized saddle-nodes, see e.g. \cite{bonckaert2008a}. Indeed, the linearization of \eqref{inner} has eigenvalues $0,\pm i$ and the invariant manifolds are center-like manifolds, being graphs over the zero eigenspace along the sectors $x\in S^\pm$.

In \cite{bonckaert2008a}, generalized saddle-nodes are described through Borel-Laplace. Since their invariant manifolds are intrinsically related to summability, it is  (in the opinion of the author) the most natural approach to study such manifolds.  %This approach is natural because it relates formal series etc.
In this paper, we will extend the Borel-Laplace approach of \cite{bonckaert2008a} to study the PDE \eqref{main}. We believe that this extension is interesting in itself. We are confident that the approach can be extended to inner problems associated with maps (as in \cite{gelfreich2001a}).

\subsection{Exponentially small splitting}

Let us briefly explain how the inner problem in \cite{baldom2013a} relates to the exponentially small splitting of the one-dimensional connection \eqref{unperturbed_sol} for $\epsilon\to 0$: In line with Lazutkin's approach, the authors of \cite{baldom2013a} first parameterize the one-dimensional stable and unstable manifolds of $E_2^\pm$ in terms of $t\in \mathbb C$ (using \eqref{x2ts} as a change of coordinates $t\mapsto x_2 = \tanh (t)$) up close to the poles $t=\pm \frac{i\pi}{2}$ for all $0<\epsilon\ll 1$. (Notice that $\tanh(t)\in (-1,1)$ for $t\in \mathbb R$, and consequently additional adjustments have to made to ensure that $E_2^\pm \in (-1,1)$, see \cite[p.342]{baldom2013a}.) Near the poles, the authors show that the invariant manifolds are given by the invariant manifolds of \eqref{inner} to leading order. The authors refer to this as \textit{matching} (on an $\epsilon$-dependent domain) of the ``outer expansion'' (away from the poles) and the ``inner expansion'' (on the blowup of the pole). %This leads to sub-optimal estimates of the form $\mathcal O(\epsilon^{1-\gamma})$ with $\gamma\in (0,1)$ (see \cite[Theorem 5]{}) which are not part of the problem but a consequence of the approach. In any case,
Since the invariant manifolds for $s\to \pm \infty$ are different in general on overlapping doamins  (which is quantified in terms of a Stokes constant in \cite{baldom2013a}), this essentially leads to a splitting of the invariant manifolds for $\epsilon>0$ small enough.
To determine an asymptotic formula for the real splitting for $t=0$ (corresponding to $x_2=0$), a linear boundary value problem is derived. In this way, the separation close to the pole is carried down to the real axis; in this step one also uses the fact that the equations are real analytic. %information at the pole at $t_2=-\frac{i\pi}{2}$.

% CUT: In the context of \eqref{zerohopf0}, one obvious disadvantages of using \eqref{unperturbed} as a change of coordinates is that for $\epsilon>0$ the equilibria $E_2^\pm$ will in general move away from $\pm 1$ and for this reason the equations have to be modified slightly. This creates some technicalities in \cite{} which are not essential but a result of the method.

In \cite{bkt}, the present author provided an alternative more geometric approach to the splitting problem of the one-dimensional invariant manifolds of $E_2^\pm$. This approach builds upon the fact that the inner problem is just the unperturbed system and therefore works exclusively in the (complex) phase space. Importantly, the author treats the scaling \eqref{x2y2z20} as a $\breve \epsilon=1$-chart of an associated blowup transformation:
\begin{align}\eqlab{blowup0}
r\ge 0,\,(\breve x,\breve y,\breve z,\breve \epsilon) \in \mathbb S^4\,\mapsto \begin{cases}
        x = r\breve x,\\
                        y = r\breve y,\\
                        z=r\breve z,\\
                        \epsilon = r^2 \breve \epsilon.\end{cases}
\end{align}
(Strictly, speaking the authors use a slightly different scaling, but this is a technicality and not really important here.) First, the author describes the stable and unstable manifolds of $E_2^\pm$ in the coordinates $(x_2,y_2,z_2)$ as graphs over $x_2$ in compacts subsets for $0<r_2\ll 1$. The invariant manifolds are then extended by the flow to $x=\mathcal O(1)$ by working in the separate coordinate chart defined by $\breve x=1$:
\begin{align*}
\begin{cases}
        x = r_1,\\
                        y = r_1 y_1,\\
                        z=r_1 z_1,\\
                        \epsilon = r_1^2 \epsilon_1.\end{cases}
\end{align*}
The coordinates $(r_1,y_1,z_1)$ are natural coordinates for the parametrization of the invariant center-like manifolds of \eqref{inner} within $\epsilon=0$ (corresponding to $\epsilon_1=0$). Indeed, in these coordinates it is elementary to bring \eqref{inner} into the form
\begin{align*}
 r_1^2 \frac{d\zeta_1}{dr_1} =\begin{pmatrix} 0 & -1\\
                                    1 & 0
                                   \end{pmatrix} \zeta_1+\mathcal O(\vert(r_1,\zeta_1)\vert^2),
\end{align*}
with $\zeta_1=(y_1,z_1)$. Moreover, it follows from the change of coordinates between $\breve x=1$ and $\breve \epsilon=1$:
\begin{align*}
\begin{cases}
 r_1 =r_2x_2,\\
 y_1 = y_2x_2^{-1},\\
 z_1 = z_2 x_2^{-1},\\
 \epsilon_1 = x_2^{-2},
\end{cases}
\end{align*}
that the invariant manifolds of the $\breve \epsilon=1$-chart are also (partially) visible in the $\breve x=1$-chart (for $r_1=\mathcal O(\sqrt{\epsilon})$ and with $x_2$ bounded uniformly away from zero).
%                                    The invariant manifolds therefore take the graph form
%                                    \begin{align}\eqlab{zeta1man}
%                                     (y,z) = x^3 m^\pm (x),\quad r_1 \in S^\pm.
%                                    \end{align}
The results of \cite{bkt} show that the extensions of the local stable and unstable manifolds of $E_2^\pm$ are $\mathcal O(r_2)$-close to the invariant manifolds of the unperturbed system \eqref{inner} within appropriate small but compact sets of $x\in \mathbb C$ that are uniformly bounded away from $x=0$. The paper \cite{bkt} also presents an alternative geometric approach for the difference by deriving a Fenichel normal form, see \cite{jones_1995}. In this way, the results of \cite{bkt} directly \textit{relate the exponentially small splitting with the lack of analyticity of the center-like manifolds of \eqref{inner}}.

We emphasize that \cite{bkt} works directly in phase space, with the outer (inner) problem of the Lazutkin-based approach in \cite{baldom2013a} corresponding to $x=\mathcal O(\sqrt \epsilon)$ ($x=\mathcal O(1)$, respectively). (From the perspective of blowup, it feels most natural to refer to the unperturbed system \eqref{zerohopf0} as the outer system and \eqref{zerohopf20} as inner system; after all, in order to obtain \eqref{zerohopf20} from \eqref{zerohopf0} we zoom in on $(x,y,z,\epsilon)=(0,0,0,0)$ through the scaling \eqref{x2y2z20}. Unfortunately, within the Lazutkin-framework it is the other way around. To avoid (further) confusion, we stick to the terminology from Lazutkin in this paper.) %It is well-documented that blowup is a suitable tool for matching ``blowup for matching''
It is by now well-established that blowup is a powerful systematic technique for connecting different scaling regimes in dynamical systems, see e.g. \cite{Gucwa2009783,kristiansen2024a,krupa_extending_2001,kuehn2015a}. We believe that \cite{bkt} supports this further. In this reference, we also see that blowup allows for a very detailed information about the $\epsilon$-dependency (in line with previous work on blowup, see e.g. \cite{DMS2016} and \cite{de2021a} more broadly) of the splitting of the one-dimensional invariant manifolds of the zero-Hopf.
\subsection{Discussion}
The papers \cite{MR4855745,MR4445442} demonstrate a different connection between splitting phenomena and lack of analyticity of center manifolds. In particular, in \cite{MR4855745} the authors consider the analytic unfolding of the planar saddle-node bifurcation and study the properties of the analytic weak-stable manifold as the system approaches the saddle-node (coming from the side of the bifurcation where a saddle and a node co-exist). The problem is singular in the sense that the analytic weak-stable manifold is only well-defined when the node is nonresonant and these resonances accumulate as the system approaches the bifurcation. However, the main result of \cite{MR4855745} shows (under some hypothesis which by the subsequent work \cite{kristiansen2025a} can be relaxed) that if a certain Stokes-like constant is nonzero then the center manifold at the bifurcation is nonanalytic. Moreover, it is demonstrated that this quantity (qualitatively) determines the position of the analytic weak-stable manifold. In particular, it is shown (under the assumption of a nonzero Stokes constant) that the analytic weak-stable manifold never coincides with the invariant manifolds of the saddle. Although this problem is clearly different from the zero-Hopf bifurcation in many aspects, it does bear some important similarities. For example, \cite{MR4855745} treats the unperturbed problem as an inner-like system having a center manifold (in the usual sense of a partially hyperbolic singularity). As for the zero-Hopf in \cite{bkt}, this manifold is then connected to an invariant manifold (the analytic weak-stable manifold of the nonresonant node) of the system in the scaled (blowup) coordinates.  %which is naturally described in scaled coordinates (which uniformly separates the node from the saddle).

In \cite{new}, the present author applies the geometric method from \cite{bkt} (on the one-dimensional splitting associated with the zero-Hopf bifurcation) to a general class of co-dim $k$, $k\ge 3$, (non-reversible) zero-Hopf bifurcations. Although these bifurcations are more esoteric, the paper \cite{new} perhaps illustrates the method in \cite{bkt} more clearly, by also making connections to Stokes and anti-Stokes curves (drawing inspiration from \cite{hayes2016a,neishtadt1987a,neishtadt1988a}).

Looking ahead, the geometric viewpoint in \cite{bkt,new} opens up for potential applications in general problems where the explicit time dependency of unperturbed solutions are unknown.
% also from the perspective of blow-up, the wording ``inner'' and ``outer'' used by \cite{} (within the Lazutkin framework) is slightly confusing. Indeed, the inner problem \eqref{
However, a central component in the development of phase space space methods for the splitting of more complicated connection problems, is to describe invariant manifolds of the unperturbed systems directly in phase space.
We address this issue in the present paper by studying existence of two-dimensional invariant manifolds of \eqref{mainvf}.

 \subsection{The zero-Hopf and Hopf-Hopf bifurcations}
In the present paper, we will demonstrate that the inner problems associated with the two-dimensional formal connections in the following local bifurcations, can be brought into the form \eqref{mainvf}:
\begin{enumerate}
 \item  The zero-Hopf in $\mathbb R^3$ (see $\Gamma_0$ in \figref{x2rho2} and \secref{zero} for further details).
 \item The reversible and resonant Hopf-Hopf bifurcation in $\mathbb R^4$ (see $\Gamma_0$ in \figref{hopfhopf} and \secref{hopfhopf} for further details).
 \end{enumerate}
%  can both be brought into the form \eqref{mainvf}. %In contrast to the approaches of \cite{}, the particular expression for the difference in \thmref{main2} is not important in the geometric approach of \cite{} for $\epsilon>0$. However, for completeness we include it here. Our approach to proving \thmref{main2} is also geometric and mimics the approach used for $\epsilon>0$ small enough.

% We will demonstrate our perspective on inner problems further in the present paper by considering (a) the two-dimensional formal heteroclinic connection between $E_2^\pm$ in the zero-Hopf bifurcation and (b) the two-dimensonal formal homoclinic connection in
For the reversible zero-Hopf singularity, we will work with the normal form \eqref{zerohopf0}, whereas for the reversible and resonant Hopf-Hopf bifurcation in $\mathbb R^4$, we will consider the following normal form:
\begin{equation}\eqlab{nfhopfhopf0}
\begin{aligned}
\dot x &=-(1+\Omega)y+z+F,\\
\dot y &=(1+\Omega)x+w+G,\\
\dot z &= -(1+\Omega)w+\Gamma x+H, \\
\dot w &=(1+\Omega )z+\Gamma y+J,
\end{aligned}
\end{equation}
with
\begin{align*}
W=W(x,y,z,w,\epsilon)=\mathcal O(\vert (x,y,z,w,\epsilon)\vert^5),\quad W=F,G,H,J,
\end{align*}
and where
\begin{align}\nonumber
\begin{cases}
\Gamma(\rho^2,L,\epsilon) =  \epsilon - b \rho^2 +c L,\\
 \Omega(\rho^2 ,L,\epsilon)=\alpha \epsilon + \beta \rho^2 +\gamma L,
 \end{cases}
\end{align}
for
\begin{align*}
  \rho^2:=x^2+y^2,\,L:=wx-yz,\end{align*}
see \cite[Lemma 3.17, p.  215]{haragus2011a}. (Notice that in comparison with \cite{haragus2011a}, we write their $(A,B)$ as $(x+iy,z+iw)$ to obtain a real normal form.) Here $\epsilon\sim 0$ is the unfolding parameter.

Notice that for $\epsilon=0$, the linearization of the origin has eigenvalues $\pm i$ each with algebraic multiplicity $2$ and geometric multiplicity $1$. This bifurcation is also known as the $1:1$ resonance bifurcation or $(i\omega)^2$ in symbols, see \cite[p. 214]{haragus2011a}.
Importantly, the system \eqref{nfhopfhopf0} is assumed to be reversible with respect to the symmetry:
\begin{align*}
\mathcal S\,:\,(x,y,z,w)\mapsto (x,-y,-z,w).
\end{align*}

This type of bifurcation occurs in the stationary
generalized Swift-Hohenberg equation:
\begin{align*}
   \kappa u^2 -u^3 -(1+\partial_x^2)^2 u =\epsilon u,
\end{align*}
see \cite{gaivao2011a}, and in travelling waves of the fifth order KdV-equation, see \cite{yang1997a}. Both cases have a Hamiltonian structure, see also \cite{glebsky1995a} for a Hamiltonian analysis of the Swift-Hohenberg equation (and generalizations hereof). We do not cover Hamiltonian systems but are confident that our approach carries over completely analagously.

In the context of \eqref{nfhopfhopf0}, we again define an inner problem as the unperturbed system, i.e. \eqref{nfhopfhopf0} with $\epsilon=0$.

For further details, we refer to \secref{zero} and \secref{hopfhopf} where we also compare with the Lazutkin-based approach for the inner problem.

We summarize our findings on the two examples as follows:
\begin{proposition}\proplab{prop13}
 There exist changes of coordinates (of blowup-type) that bring the inner problems (i.e. the unperturbed problems with $\epsilon=0$) associated with \eqref{zerohopf0} (for any $b\in \mathbb R\setminus [-1,0]$, $a\in \mathbb R \setminus\{0\}$) and \eqref{nfhopfhopf0} (for any $b\in \mathbb R\setminus\{0\}$) into the normal form \eqref{mainvf}.
\end{proposition}\proplab{zerohopf}

In turn, we have the following corollaries of \thmref{main}:
\begin{proposition}
 Consider \eqref{zerohopf0} with $\epsilon=0$ for any $b\in \mathbb R\setminus [-1,0]$, $a\in \mathbb R \setminus\{0\}$, in the cylindrical coordinates $(x,\rho,\theta)$ defined by
\begin{align*}
\begin{cases}
 y = \rho \cos \theta,\\
 z = \rho \sin \theta.
\end{cases}\end{align*}
 Then there are invariant manifolds of the graph form
\begin{align*}
 \rho = x\Phi^\pm (x,\theta),\quad x\in S^\pm,\quad \theta\in \mathbb T_\xi,
\end{align*}
where $\Phi^\pm\,:\,S^\pm \times \mathbb T_\xi\to \mathbb C$ with $$\Phi^\pm(0,\theta)  \equiv i \sqrt{\frac{1+b}{a}},$$
are $1$-sums (with respect to $x$, uniformly with respect to $\theta\in \mathbb T_\xi$) of a Gevrey-$1$ formal series along the directions defined by $x\in S^\pm$, respectively.
% \note{think about symmetry}
% The invariant manifolds are related by the symmetry $\mathcal S$.
\end{proposition}
In the following, we let $\upsilon S^\pm$, $\upsilon\in \mathbb C$, denote the sets
\begin{align*}
 \upsilon S^\pm : =\{x\in \mathbb C\,:\,\upsilon^{-1}x\in S^\pm\},
\end{align*}
recall \eqref{Spm}.
% \begin{proposition}
%  Suppose that $b\in \mathbb R\setminus\{0\}$ and
%
\begin{proposition}\proplab{hopfhopf}
 Consider \eqref{nfhopfhopf0} with $\epsilon=0$ for any $b\in \mathbb R\setminus\{0\}$ in the coordinates $(\rho,\theta,\psi,L)$ defined by
% \begin{align*}
% \begin{cases}
 \begin{align}\nonumber
\begin{cases}
 x = \rho \cos \theta,\\
 y =\rho \sin \theta,\\
 z =\psi \cos \theta - L\rho^{-1} \sin \theta,\\
 w = \psi \sin \theta+ L\rho^{-1} \cos \theta.
 \end{cases}
\end{align}
 Let
 \begin{align}
  \upsilon:=i^{-1}\sqrt{\frac{2}{b}}.
 \end{align}
 Then there are invariant manifolds of the graph form
\begin{align*}
\begin{cases}
 \psi = \rho^2 \Phi^{\psi,\pm} (\rho,\theta),\\
 L = \rho^3\Phi^{L,\pm}(\rho,\theta),
\end{cases}\quad \quad \rho \in \upsilon S^\pm,\quad \theta\in \mathbb T_\xi,
\end{align*}
where $\Phi^{q,\pm}\,:\,\upsilon S^\pm \times \mathbb T_\xi\to \mathbb C$ with $$\begin{cases}\Phi^{\psi,\pm}(0,\theta)  \equiv i \sqrt{\frac{b}{2}},\\
\Phi^{L,\pm}(0,\theta)\equiv 0,\end{cases}$$
% \end{cases}\end{align*}
are $1$-sums (with respect to $x$, uniformly with respect to $\theta\in \mathbb T_\xi$) of a Gevrey-$1$ formal series along the directions defined by $x\in \upsilon S^\pm$, respectively.
\end{proposition}

\propref{prop13} is a consequence of \lemmaref{lemma24} and \lemmaref{lemma33} below. These results are each obtained in the same way, by working in charts associated with the problem-dependent blowups. In further details, we first show the existence of formal series solutions. We then truncate this series and subsequently apply a (separate) blowup (following the preparation of the generalized saddle-nodes in \cite[Proposition 1]{bonckaert2008a}).

\propref{zerohopf} and \propref{hopfhopf} follow from more detailed statements below (in charts), see \propref{prop25} respectively \propref{prop34}. For further details, see \secref{zero} and \secref{hopfhopf}.

We view these results as prerequisites for our upcoming treatment of the associated exponentially small splitting phenomena in these problems as $\epsilon\to 0$.

\subsection{Outline}
The paper is organized as follows. In \secref{zero}, we first consider the two-dimensional formal connection of \eqref{zerohopf0}. We derive the associated inner problem by following Lazutkin's approach and show that the system is equivalent with the unperturbed system. Moreover, we show that the unperturbed system (in polar coordinates) can be written in the form \eqref{mainvf}. Next in \secref{hopfhopf}, we perform a similar analysis on the reversible and resonant Hopf-Hopf bifurcation with normal form \eqref{nfhopfhopf0}. In particular, we again identify two-dimensional formal connections and bring the unperturbed system (which agrees with Lazutkin's version of the inner problem) into the form \eqref{mainvf}. In both cases, blowup plays a crucial role.

In \secref{borellaplace}, we then review the Banach-convolution-algebra-approach to Borel-Laplace by Bonckaert and De Maesschalck, see \cite{bonckaert2008a}, and extend it so that it can be applied to \eqref{main} (with the $\theta$-dependency). In \secref{equations}, we then recast \eqref{main} into an equation in the Borel-plane for the Borel transform $\widehat{\mathbf y}$ of $\mathbf y$. We solve this equation in the appropriate Banach space of exponentially growing solutions along an infinite sector, using the theory from \secref{borellaplace} and a fixed-point argument. The Laplace transformation $\mathcal L^\varphi[\widehat{\mathbf y}]$ of $\widehat{\mathbf y}$ then solves \eqref{main} along sectors (defined by $\varphi$). In this way, we prove \thmref{main} (taking $\varphi=0$ and $\varphi=\pi$). Finally, in \secref{diff} we study the difference $\Delta \mathbf y(x,\theta)$ for $x\in S^+\cap S^-$ and prove \thmref{main2}. Our approach for the difference is also geometric (using invariant manifolds). %We conclude the paper in \secref{conclusion}.

\section{The reversible zero-Hopf bifurcation}\seclab{zero}
In this section, we describe the inner problem associated with the two-dimensional formal connection between the saddle-focus singularities $E_2^\pm$ of the reversible zero-Hopf bifurcation \eqref{zerohopf20} and bring it into the general form \eqref{mainvf}.% Lazutkin's approach and then relate this to the unperturbed system of \eqref{zerohopf0}.

In order to introduce the two-dimensional formal  connection for \eqref{zerohopf20}, we start by ignoring the higher order terms of \eqref{zerohopf0}:
% \begin{align*}
\begin{equation}\eqlab{zerohopf0trunc}
% \begin{equation}\eqlab{model0}
\begin{aligned}
 x' &=-x^2+\epsilon-a (y^2 +z^2),\\
 y'&= bx y -z,\\
 z'&= bx z+ y,
\end{aligned}
% \end{equation}
\end{equation}
and write the resulting system in polar coordinates $(x,\rho,\theta)$ defined by
\begin{align}\eqlab{polar}
 \begin{cases}
  y = \rho \cos \theta,\\
  z =\rho\sin \theta.
 \end{cases}
\end{align}
This gives the following system:
\begin{equation}\eqlab{xrho}
\begin{aligned}
 x' &=-x^2+\epsilon -a \rho^2,\\
 \rho' &=bx\rho,\\
 \epsilon' &=0,
\end{aligned}
\end{equation}
and $\dot \theta = 1$, which decouples. %Here we have also used the parameter blowup  \eqref{parameterbu}.
In anticipation of the blowup:
\begin{align}\eqlab{blowup0polar}
 r\ge 0,\,(\breve x,\breve \rho,\breve \epsilon)\in \mathbb S^3\mapsto \begin{cases}
                                                  x =r \breve x,\\
                                                  \rho = r\breve \rho,\\
                                                  \epsilon = r^2\breve \epsilon
                                                 \end{cases}
\end{align}
recall \eqref{blowup0},
we have also augmented an equation for $\epsilon$.
We are primarily interested in the case
$$a>0,\quad b>0.$$

In the $\breve\epsilon=1$-chart associated with \eqref{blowup0polar}, having the chart-specific coordinates $(x_2,\rho_2,r_2)$ defined by
\begin{align*}
 \begin{cases}
  x =r_2 x_2,\\
  \rho =r_2 \rho_2,\\
  \epsilon =r_2^2,
 \end{cases}
\end{align*}
we find that
\begin{equation}\eqlab{x2rho2}
\begin{aligned}
 \dot x_2 &=-x_2^2+1 -a \rho_2^2,\\
 \dot \rho_2 &=bx_2\rho_2,
\end{aligned}
\end{equation}
and $\dot r_2=0$ after dividing the right hand side by the common factor $r_2$.
%
%
% \note{maybe start with outer system in $(x,\rho,\theta)$-coordinates and apply blowup: we first look in scaling chart and infer that we have invariant manifolds of the unperturbed problem that are best described in the entry chart}
% \begin{equation}\eqlab{zerohopf2}
% \begin{aligned}
%  \dot x_2 &= -x_2^2+1-a (y_2^2+z_2^2),\\
%  \epsilon \dot y_2&= -z_2+\epsilon b ( x_2+\sigma)y_2,\\
%  \epsilon \dot z_2 &=y_2+\epsilon b(x_2+\sigma)z_2,
% \end{aligned}
% \end{equation}
% and introduce polar coordinates in the $(y_2,z_2)$-plane:
% \begin{align*}
%  \begin{cases}
%   y_2 = \rho_2 \cos \theta,\\
%   z_2 =\rho_2\sin \theta.
%  \end{cases}
% \end{align*}
% This gives
% \begin{equation}\eqlab{x2rho2}
% \begin{aligned}
%  \dot x_2 &=-x_2^2+1 -a \rho_2^2,\\
%  \dot \rho_2 &=b(x_2+\sigma)\rho_2,
% \end{aligned}
% \end{equation}
% and $\dot \theta = \epsilon^{-1}$ which decouples.
We then notice that $(x_2,\rho_2)=(\pm 1,0)$ (corresponding to $E_2^\pm$ from the introduction) are hyperbolic saddles of \eqref{x2rho2} for any $b>0$, with the linearization having eigenvalues
\begin{align*}
 \mp 2,\quad \pm b,
\end{align*}
respectively.
 The set  $\gamma_0$ defined by $\rho_2=0,\,x_2\in (-1,1)$, is therefore a heteroclinic connection (corresponding to \eqref{unperturbed_sol}) for any $b>0$, see \figref{x2rho2}. On the other hand for $a>0$ and $b>0$, we also have local stable and unstable manifolds of $(x_2,\rho_2)=(\mp 1,0)$, respectively, transverse to $\rho_2=0$ that also coincide (due to the symmetry $\mathcal S$) to form a separate heteroclinic connection $\Gamma_0$ contained within $\rho_2>0$. We illustrate this situation
in \figref{x2rho2}.
% for $a>0$, $b>0$ and $\sigma=0$. Notice that $\gamma_0\,:\,\rho_2=0,x_2\in (-1,1)$ defines heteroclinic connection (corresponding to \eqref{unperturbed_sol}). But there is also another connection within $\rho_2>0$ for these values of the parameters given by:
An easy computation shows that the connection $\Gamma_0$ takes the following graph form:
\begin{align}\eqlab{Gamma0eqn}
\Gamma_0\,:\,\rho_2 = \sqrt{\frac{1+b}{a}(1-x_2^2)},\quad x_2\in (-1,1).
\end{align}
% $\Gamma_0$ breaks up with nonzero speed for $\sigma\ne 0$, see \figref{x2rho2}(a) and (b) for $\sigma<0$ and $\sigma>0$, respectively. This follows from a simple Melnikov computation.
Finally, we note that on $\Gamma_0$ we have
\begin{align*}
 \dot x_2 &=b (x_2^2-1),
\end{align*}
and from this we deduce that the heteroclinic connection $\Gamma_0$ has the following time parametrization:
\begin{align}\eqlab{unperturbed_sol2}
 \begin{cases}
  x_2(t) = -\tanh(bt),\\
  \rho_2(t) =\sqrt{\frac{1+b}{a}} \operatorname{sech}(bt),
 \end{cases}
\end{align}
with poles closest to the real axis given by $t=\pm \frac{i\pi}{2b}$.

\begin{remark}\remlab{homoclinic}
It is easy to see that the higher order terms in \eqref{zerohopf0}, ignored in \eqref{zerohopf0trunc}, lead to regular perturbations of \eqref{rho2psi2L2eqns} of order $\mathcal O(r_2)$ in the $(x_2,\rho_2,\theta)$-coordinates (in compact domains with $\rho_2$ bounded uniformly away from zero). The symmetry $\mathcal S$ takes the following form $(x_2,\rho_2,\theta)\mapsto (-x_2,\rho_2,-\theta-\frac{\pi}{2})$ in the $(x_2,\rho_2,\theta)$-coordinates. From this one can deduce the existence of two symmetric homoclinic orbits for all $0<r_2\ll 1$ (due to the intersection of stable and unstable manifolds with $x_2=0$, $\theta = -\frac{\pi}{4}+n\pi$, $n\in \mathbb Z$, being the fixed-point set of the symmetry). The interesting problem (from the perspective of exponentially small phenomena) is then an asymptotic formula for the splitting of tangent spaces and whether there are other non-symmetric homoclinic orbits. This question is addressed in \cite{baldoma2018a} (also within the non-reversible case). In further details, \cite[Theorem 1.1]{baldoma2018a} gives an asymptotic formula for the difference between the two-dimensional stable and unstable invariant manifolds within $\{x_2=0\}$. We aim to study this problem using the geometric approach of the present paper and \cite{bkt,new} in future work.

\end{remark}

\begin{figure}[h!]
\begin{center}
% \subfigure[$\sigma<0$]{\includegraphics[width=.45\textwidth]{x2rho2_neg.pdf}}
\subfigure{\includegraphics[width=.65\textwidth]{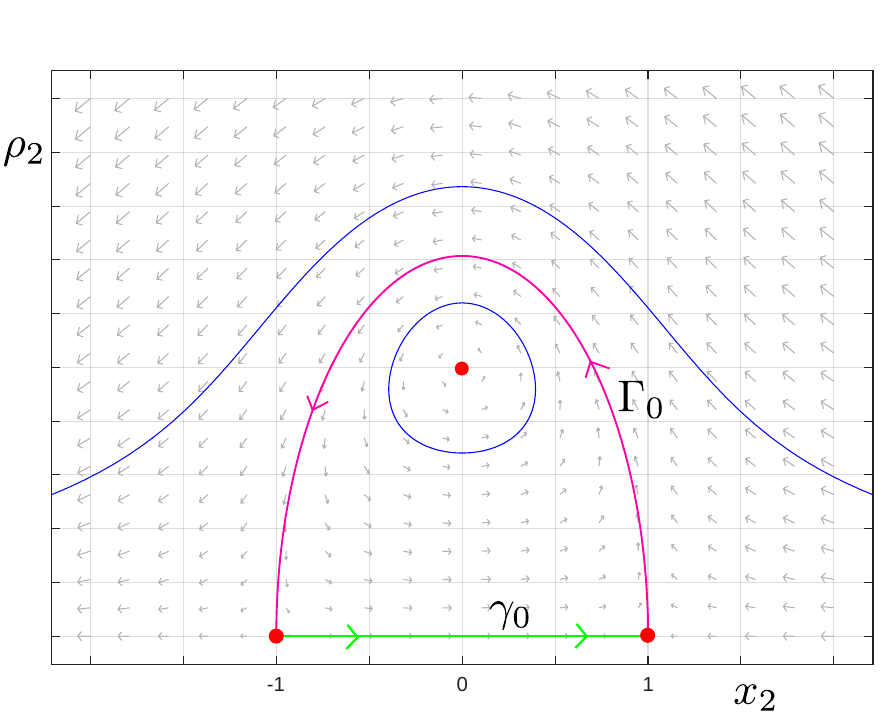}}
% \subfigure[$\sigma>0$]{\includegraphics[width=.45\textwidth]{x2rho2_pos.pdf}}
\end{center}
\caption{Phase portrait of \eqref{x2rho2} for $a>0$, $b>0$. There are two heteroclinic connections $\gamma_0\subset \{\rho_2=0\}$ and $\Gamma_0\subset \{\rho_2>0\}$.}
\figlab{x2rho2}
% Remark: If $U$ smooth for $|u|<\nu$, the compact manifolds lie inside $|u|<\nu$ for large $n$
\end{figure}

We first follow Lazutkin's approach for the derivation of an inner problem associated with $\Gamma_0$ (see also \cite{baldoma2018a}):
Let $s$ be such that
\begin{align}\eqlab{ts}t=t(s):=\frac{i\pi}{2b}+r_2 \frac{s}{b}.\end{align}
Then \eqref{unperturbed_sol2} becomes
\begin{align}\eqlab{unperturbed_sol2s}
 \begin{cases}
  x_2(t(s)) = - \coth(r_2 s),\\
  \rho_2(t(s)) =-i \sqrt{\frac{1+b}{a}}\operatorname{csch}(r_2 s),
 \end{cases}
\end{align}
We write this set in the $(x,\rho,\theta)$ coordinates defined by $(y,z)=\rho (\cos(\theta),\sin(\theta))$, recall \eqref{x2y2z20}:
\begin{align*}
%  \begin{align}\eqlab{unperturbed_sol2s}
 \begin{cases}
  x(t(s)) := r_2 x_2(t(s))\to  -s^{-1},\\
  \rho(t(s)) :=r_2 \rho_2(t(s))) \to  -i\sqrt{\frac{1+b}{a}} s^{-1},
 \end{cases}
% \end{align}
\end{align*}
as $r_2 \to 0$, $s\ne 0$. It follows that
\begin{align}
 \rho = i\sqrt{\frac{1+b}{a}} x,\quad x\in \mathbb C,\eqlab{unperturbed_sol20}
\end{align}
is an invariant manifold solution of \eqref{xrho} within $r_2=0$. This can also be directly verified. There is obviously also an invariant manifold of the form $\rho = -i\sqrt{\frac{1+b}{a}} x$. Notice moreover that the equation for $\Gamma_0$ in \eqref{Gamma0eqn} becomes
\begin{align*}
 \rho =  \sqrt{\frac{1+b}{a}(r_2^2-x^2)}\to \pm i\sqrt{\frac{1+b}{a}}x\quad \mbox{for}\quad r_2 \to 0,
\end{align*}
in the $(x,\rho)$-coordinates. In conclusion, \eqref{unperturbed_sol20} is an approximation (nonuniform) of the invariant manifolds of \eqref{zerohopf0trunc} in the phase space variables $(x,\rho,\theta)$. (In the language of \cite{baldoma2018a} (inspired by Lazutkin) it is an approximation of the invariant manifolds near the poles of the unperturbed solution \eqref{unperturbed_sol2}).

This motivates the definition of the inner problem associated with $\Gamma_0$ as the unperturbed problem \eqref{inner} (with $F, G, H$ added) written in polar coordinates:
\begin{equation}\eqlab{inner2}
\begin{aligned}
x' &=-x^2 - a \rho^2 + F,\\
\rho' &=b x \rho + G\cos \theta+H\sin \theta,\\
\theta' &=1+\rho^{-1}\left(-G\sin \theta+H\cos \theta\right),
\end{aligned}
\end{equation}
where $W=W(x,\rho \cos \theta,\rho \sin \theta,0)$, $W=F,G,H$.
% \begin{remark}\remlab{rho1}
% The set given by \eqref{unperturbed_sol20} defines an invariant manifold solution of \eqref{zerohopf0} within $\epsilon=0$. Therefore when written in the chart $\breve x = 1$ associated with the blowup \eqref{blowup0} with the chart-specific coordinats $(r_1,\rho_1,\theta,\epsilon_1)$ defined by
% \begin{align*}
%   \begin{cases}
%    x = r_1,\\
%    \rho = r_1 \rho_1,\\
%    \epsilon = r_1 \epsilon_1,
%   \end{cases}
% \end{align*}
% takes the form
% \begin{align*}
%  \rho_1= i\sqrt{\frac{1+b}{a}},\,r_1\in \mathbb C,\,\epsilon_1=0.
% \end{align*}
% \end{remark}
%
% We now argue that the inner problem associated to the connection $\Gamma_0$ is the unperturbed system \eqref{inner}, repeated here in polar coordinates $(x,\rho,\theta)$ defined by
% $(y,z)=\rho(\cos \theta,\sin \theta)$:
% We therefore define the system \eqref{inner2} as the inner system associated with $\Gamma_0$.
In contrast to the inner problem associated with $\gamma_0$ (treated in the introduction), we are now interested in manifolds that are graphs over $(x,\theta)$ of the form:
\begin{align}
 \rho =\rho(x,\theta)=\pm i\sqrt{\frac{1+b}{a}}x (1+\mathcal O(x))\quad \mbox{for}\quad x\to 0.\eqlab{rho_ansatz_zerohopf}
\end{align}

It is natural to study \eqref{inner2} in the coordinates $(r_1,\rho_1)$ of the $\breve x=1$-chart defined by
\begin{align}\eqlab{entry1}
  \begin{cases}
   x = r_1,\\
   \rho = r_1 \rho_1,\\
   \epsilon = r_1^2 \epsilon_1.
  \end{cases}
  \end{align}
  Indeed, \eqref{unperturbed_sol20} takes the following form
%
% \end{align}
% Indeed, here it takes the following form
\begin{align*}
 \rho_1= i\sqrt{\frac{1+b}{a}},\,r_1\in \mathbb C,\,\epsilon_1=0,
\end{align*}
in these coordinates. The change of coordinates defined by \eqref{entry1} brings
\eqref{inner2} into the following form:%This gives
\begin{equation}\eqlab{entry1eqns}
\begin{aligned}
 \dot r_1  &=r_1^2 \left(-1-a \rho_1^2 + r_1 F_1\right),\\
 \dot \rho_1 &=r_1\rho_1\left(1+b+a\rho_1^2-r_1 F_1\right)+r_1^2 G_1 \cos \theta+r_1^2 H_1\sin \theta,\\
 \dot \theta &=1+\rho_1^{-1} r_1^2 \left(-G_1\sin \theta+H_1\cos \theta\right),
\end{aligned}
\end{equation}
within $\epsilon_1=0$. Here the functions $W_1=W_1(r_1,\rho_1,\theta)$, $W=F,G,H$, are defined as
\begin{align}\eqlab{W1defn}
W_1(r_1,\rho_1,\theta):=r_1^{-3}W(r_1,r_1\rho_1\cos \theta,r_1\rho_1\sin \theta,0),\quad r_1\in B_\tau,\,\rho_1\in D,\,\theta\in \mathbb T_\xi,
\end{align}
with $D\subset \mathbb C$ a fixed compact set bounded away from zero and $\tau>0$, $\xi>0$, both small enough. Notice in particular that each $W_1$ extends analytically to $r_1=0$ by \eqref{Wthird}. For \eqref{entry1eqns} we are interested in invariant manifold solutions $\rho_1=\rho_1(r_1,\theta)$ with
\begin{align*}
 \rho_1(0,\theta) =\pm i\sqrt{\frac{1+b}{a}}.
\end{align*}
We will only focus on ``$+$''. % (the other case, is obtained by conjugation).
% From such solutions, we can generate new ones by applying the symmetry
% %
% % another one (with $\rho_1(0,\theta) = -i\sqrt{\frac{1+b}{a}}$) by applying the symmetry:
% \begin{align}
% \eqlab{symmetryzerohopf}
% \mathcal S_0\,:\,(r_1,\rho_1,\theta)\mapsto (r_1,-\rho_1,\theta+\pi),
% \end{align}

The system \eqref{entry1eqns} is reversible with respect to the symmetry
\begin{align}
 \mathcal S_1\,:\, (r_1,\rho_1,\theta)\mapsto (-r_1,\rho_1,-\theta+\pi/2).\eqlab{symmetryzerohopf}
\end{align}
This is derived from \eqref{symmetry0} using \eqref{polar} and \eqref{entry1}, see also \remref{homoclinic}.

% and the reversible symmetry
% \note{rewrite everything below in terms of $\rho_1$ instead}
% To motivate \eqref{rho_ansatz_zerohopf} further recall that $x=\epsilon x_2$ and $\rho=\epsilon \rho_2$.

We note that the inner problem associated with the connection $\Gamma_0$ is cast in a different form in \cite{baldoma2018a} (since the authors follow Lazutkin's approach). Here the authors again view the $x_2$-solution of the truncated system: $x_2 = -\tanh(bt)$ as a change of coordinates $t\mapsto x_2$ and zoom in on the pole at $t=\frac{i\pi}{2b}$ ($t=-\frac{i\pi}{2b}$ is related by conjugation) through \eqref{ts}.
% In further details, let $s$ be such that $$t=t(s):=-\frac{i\pi}{2b}+\epsilon \frac{s}{b}.$$ Notice by \eqref{unperturbed_sol2} that
% \begin{align}\eqlab{x2rho2s}
%  \begin{cases}
%   x_2(t(s)) = -\frac{1}{\epsilon s} + \mathcal O(\epsilon s),\\
%   \rho_2(t(s)) =i \sqrt{\frac{1+b}{a}} \frac{1}{\epsilon s} + \mathcal O(\epsilon s),
%  \end{cases}
% \end{align}
% for $s\ne 0$ and $\epsilon\to 0$.  This motivates the change of coordinates $(s,\rho,\theta)\mapsto (x_2,y_2,z_2)$
%  with $$x_2 = -\tanh(t(s))= -\operatorname{coth}(\epsilon s),\quad\rho=\epsilon \rho_2,$$ which
This brings \eqref{zerohopf0} into the following form
\begin{align*}
 \left(-r_2^2 \operatorname{coth}^2(r_2 s) +r_2^2 -a \rho^2 + F\right)  \frac{d\rho}{ds} &=r_2^2 \operatorname{csch}^2(r_2 s) \left( b\rho\left(- r_2 \operatorname{coth}(r_2 s) +r_2 \sigma\right) + G\cos \theta+H\sin \theta\right),\\
 \left(-r_2^2 \operatorname{coth}^2(r_2 s) +r_2^2 -a \rho^2 + F\right)  \frac{d\theta}{ds} &=r_2^2 \operatorname{csch}^2(r_2 s) \left( 1 +\rho^{-1}\left(- G\sin \theta+H\cos \theta\right)\right),
\end{align*}
with $W=W(-r_2 \coth(r_2 s),\rho \cos \theta,\rho \sin \theta,r_2^2)$, $W=F,G,H$.
The inner problem (based upon Lazutkin's approach) is then defined as the $r_2\to 0$ limit:
\begin{equation}\eqlab{inner_fucked2}
\begin{aligned}
% \begin{align*}
 \left(-s^{-2} -a \rho^2 + F\right)  \frac{d\rho}{ds} &=s^{-2} \left( b\rho s^{-1} + G\cos \theta+H\sin \theta\right),\\
 \left(-s^{-2}  -a \rho^2 + F\right)  \frac{d\theta}{ds} &=s^{-2} \left( 1 +\rho^{-1}\left(- G\sin \theta+H\cos \theta\right)\right),
% \end{align*}
\end{aligned}
\end{equation}
with $s\ne 0$ and where we redefine $W=W(-s^{-1},\rho \cos \theta,\rho \sin \theta,0)$, $W=F,G,H$.
Here we have used \eqref{sech_limit}.
But then as $x(t(s))\to -\frac{1}{s}$ for $r_2\to 0$, $s\ne 0$, we conclude that: \textit{\eqref{inner_fucked2} is equivalent with \eqref{inner2} upon the change of coordinates defined by $x=-s^{-1}$.}
% Notice that the form \eqref{rho_ansatz_zerohopf} also follows from \eqref{x2rho2s} (since $\rho=\epsilon \rho_2 \to i \sqrt{\frac{1+b}{a}} \frac{1}{ s}$ as $\epsilon \to 0$ for $s\ne 0$).
\begin{remark}
The inner equation in \cite[Eq. (25)]{baldoma2018a} is not exactly \eqref{inner_fucked2}. However, we can obtain \cite[Eq. (25)]{baldoma2018a} from \eqref{inner_fucked2} upon applying the change of coordinates $(s,\psi,\theta)\mapsto (s,\rho,\theta)$ defined by
\begin{align*}
 \rho = \sqrt{2\left(- \frac{1+b}{2a}\frac{1}{s^{2}}+\psi\right)},
\end{align*}
and writing the invariance equation for $\psi=\psi (s,\theta)$.
This follows from a simple calculation. %Notice that we here use $\widetilde \rho$ instead of the $r_1$ used in \cite[Eq. (25)]{} (since $r_1$ has a different meaning in the present manuscript).

\end{remark}

%
% Let $\rho=\epsilon \rho_2$. Then the inner problem of \cite{} then look for solutions
% % $b>0$ and $\sigma \in (-1,1)$ fixed
%
% A central component in the development of phase space methods for the splitting of the two-dimensional heteroclinic connections, is to describe two-dimensional invariant manifolds for the unperturbed system:
% \begin{equation}\eqlab{model0_unperturbed}
% \begin{aligned}
%  x' &=x^2-a (y^2 +z^2) + F(x,y,z,0,0),\\
%  y'&= -bx y +  z+G(x,y,z,0,0),\\
%  z'&= -bx z -  y+H(x,y,z,0,0),
% \end{aligned}
% \end{equation}
% corresponding to \eqref{model0} with $\mu=0$ and $\epsilon=0$. These manifolds correspond to invariant manifolds of the ``inner equation'' in \cite{}. Here we paramatrize these manifolds in phase space.

\subsection{Normal form}\seclab{normalformzerohopf}

In this section, we bring \eqref{entry1eqns} into the general form \eqref{mainvf}. For this we first recall from normal form theory, see e.g. \cite[Lemma 1.12, p. 104]{haragus2011a}, that for any $N\in \mathbb N$ there is a near-identity diffeomorphism $(x,y,z)\mapsto (\widetilde x,\widetilde y,\widetilde z)$ of the form
\begin{align}\eqlab{near_identity}
(\widetilde x,\widetilde y,\widetilde z) = (x,y,z)+T_N(x,y,z),\quad T_N(x,y,z)=\mathcal O(\vert (x,y,z)\vert^3),
%  \begin{pmatrix}
%   \widetilde x\\
%   \widetilde y\\
%   \widetilde z
%  \end{pmatrix} &=\begin{pmatrix}
%   x\\
%   y\\
%   z
%  \end{pmatrix}+\widetilde T(,
\end{align}
which conjugates \eqref{zerohopf0} with
\begin{equation}\eqlab{model0_nfN}
\begin{aligned}
 \widetilde x' &=-\widetilde x^2-a (\widetilde y^2 +\widetilde z^2) + \widetilde X_N(\widetilde x^2,\widetilde y^2+\widetilde z^2)+\widetilde F_N(\widetilde x,\widetilde y,\widetilde z),\\
 \widetilde y'&= b\widetilde x \widetilde y -  \widetilde z+\widetilde x\widetilde \Lambda_N(\widetilde x^2,\widetilde y^2+\widetilde z^2) \widetilde y - \widetilde \Omega_N(\widetilde x^2,\widetilde y^2+\widetilde z^2) \widetilde z+\widetilde G_N(\widetilde x,\widetilde y,\widetilde z),\\
 \widetilde z'&= b\widetilde x \widetilde z +  \widetilde y+\widetilde x \widetilde \Lambda_N(\widetilde x^2,\widetilde y^2+\widetilde z^2) \widetilde z + \widetilde \Omega_N(\widetilde x,\widetilde y^2+\widetilde z^2) \widetilde z+\widetilde H_N(\widetilde x,\widetilde y,\widetilde z),
\end{aligned}
\end{equation}
where $\widetilde W_N(\widetilde x,\widetilde y,\widetilde z) = \mathcal O(\vert (\widetilde x,\widetilde y,\widetilde z)\vert^{N+1})$ for $W=F,G,H$,
% \begin{align*}
%  \widetilde Z_N(\widetilde x^2,\widetilde y^2+\widetilde z^2) = \mathcal O(\vert (\widetilde x,\widetilde y,\widetilde z)\vert^2),\quad Z=\Lambda,\Omega,
% \end{align*}
$\widetilde Z_N(0,0)=0$ for $Z=\Lambda,\Omega$,
and finally $\widetilde X_N(0,0)=0$, $D\widetilde X_N(0,0)=(0,0)$. The transformation \eqref{near_identity} is equivariant with respect to the symmetry $\mathcal S$.
Notice also that we have used the reversible symmetry to simplify the normal form.

For $N\to \infty$, we obtain the ``formal'' normal form
\begin{equation}\eqlab{model0_nfinfty}
\begin{aligned}
 \widetilde x' &=-\widetilde x^2-a (\widetilde y^2 +\widetilde z^2) + \widetilde X_\infty(\widetilde x^2,\widetilde y^2+\widetilde z^2),\\
 \widetilde y'&= b\widetilde x \widetilde y -  \widetilde z+\widetilde x \widetilde \Lambda_\infty(\widetilde x^2,\widetilde y^2+\widetilde z^2) \widetilde y - \widetilde \Omega_\infty(\widetilde x^2,\widetilde y^2+\widetilde z^2) \widetilde z,\\
 \widetilde z'&= b\widetilde x \widetilde z +  \widetilde y+\widetilde x\widetilde \Lambda_\infty(\widetilde x^2,\widetilde y^2+\widetilde z^2) \widetilde z+ \widetilde \Omega_\infty(\widetilde x^2,\widetilde y^2+\widetilde z^2) \widetilde y,
\end{aligned}
\end{equation}
with $\widetilde W_\infty\in \mathbb R [[x,y,z]]$, $W=X,\Lambda,\Omega$ and $T_{\infty}\in \mathbb R^3[[x,y,z]]$.
% \begin{align*}
%  T_\infty(x,y,z) \in \mathbb C[[x= \sum_{ \alpha+\beta+\gamma \ge 3} T_{\alpha,\beta,\gamma} x^\alpha y^\beta z^\gamma,\quad T_{\alpha,\beta,\gamma}\in \mathbb R^3.
% \end{align*}
(The divergence of \eqref{near_identity} is intrinsically related to the exponentially small phenomena.) We henceforth drop the $\infty$-subscripts. Consider now the formal system \eqref{model0_nfinfty} in the cylindrical coordinates $(\widetilde x,\widetilde \rho,\widetilde \theta)$ defined by
\begin{align*}
 \begin{cases} \widetilde y = \widetilde \rho \cos \widetilde \theta,\\
\widetilde z = \widetilde \rho\sin \widetilde \theta.
 \end{cases}
\end{align*}
Then
\begin{equation}\eqlab{ninfty_xr}
\begin{aligned}
 \widetilde x' &=-\widetilde x^2-a \widetilde \rho^2 + \widetilde X(\widetilde x^2,\widetilde \rho^2),\\
 \widetilde \rho' &=b \widetilde x \widetilde \rho +\widetilde x \widetilde \Lambda(\widetilde x^2,\widetilde \rho^2)\widetilde \rho,
\end{aligned}
\end{equation}
and $\widetilde \theta' =1+\widetilde \Omega(\widetilde x,\widetilde \rho^2)$, which decouples. The linear part of \eqref{ninfty_xr} vanishes and the origin is therefore fully degenerate.  We then introduce the (directional) blowup of $(\widetilde x,\widetilde \rho)=(0,0)$
by
\begin{align*}
 \begin{cases}
  \widetilde x = \widetilde r_1,\\
\widetilde \rho = \widetilde r_1 \widetilde \rho_1,
 \end{cases}
\end{align*}
in line with \eqref{entry1},
which brings \eqref{ninfty_xr} into
\begin{equation}\eqlab{formel}
\begin{aligned}
 (\widetilde r_1^2)' &= 2\widetilde r_1^2 \left(-1-a \widetilde \rho_1^2 +\widetilde r_1^2 \widetilde X_1(\widetilde r_1^2,\widetilde \rho_1^2)\right),\\
  (\widetilde \rho_1^2)' &=  2\widetilde \rho_1^2\left(1+b+a \widetilde \rho_1^2 +\widetilde r_1^2 \widetilde \Lambda_1(\widetilde r_1^2,\widetilde \rho_1^2)\right),
\end{aligned}
\end{equation}
written as a formal system in terms of $(\widetilde r_1^2,\widetilde \rho_1^2)$,
after division of the right hand side by the common factor $\widetilde r_1$ (desingularization). Here $$\widetilde W_1\in \mathbb R\{\widetilde \rho_1^2\}[[\widetilde r_1^2]], \quad W=X,\Lambda,$$ are new formal series with respect to $\widetilde r_1^2$ having real-analytic $\widetilde \rho_1^2$-dependent coefficients. (We will use a similar notation for power series with analytic coefficients henceforth.) Given that $a\in \mathbb R \setminus\{0\}$, we see that $(\widetilde r_1^2,\widetilde \rho_1^2)=(0,-\frac{1+b}{a})$ is a hyperbolic saddle for any $b \in \mathbb R\setminus \{-1,0\}$, with the linearization having eigenvalues
\begin{align*}
 2b,\quad -2(1+b).
\end{align*}
The associated eigenvectors are given by $(1,0)$ and $(0,1)$, respectively. For any $b>0$, we therefore have a formal stable manifold given as a formal series:
\begin{align}\eqlab{formal}
 \widetilde \rho_1^2 = \widetilde \Psi(\widetilde r_1^2),\quad \widetilde \Psi(0)=-\frac{1+b}{a},
\end{align}
or $\widetilde \rho^2 =\widetilde x^2 \widetilde \Psi(\widetilde x^2)$ in terms of $(\widetilde x,\widetilde \rho)$. Here $\widetilde \Psi\in \mathbb R[[\widetilde r_1^2]]$. %For convenience we illustrate the phaseportrait of \eqref{formel} (supposing that it is convergent) for $b>0$ in \figref{formel}(b). The stable manifold is in magenta.
For further details on formal stable and unstable manifolds, we refer to \appref{formalinvman} and \lemmaref{formalinvman}.  For $b<-1$, we similarly have an unstable manifold of the form \eqref{formal}. %, see \figref{formel}(a) (magenta curve).
In this way, through the formal change of coordinates, see \eqref{near_identity} with $N=\infty$, we obtain the following:
\begin{lemma}
Suppose that $b\in \mathbb R\setminus [-1,0]$ and $a\in \mathbb R \setminus\{0\}$. Then there is a formal invariant manifold of \eqref{entry1eqns} of the form
\begin{align}\eqlab{psi}
 \rho_1 =  \Phi(r_1,\theta),\quad \Phi(r_1,\theta) = \sum_{\alpha=0}^\infty \Phi_{\alpha}(\theta) r_1^\alpha\in \mathbb C\{\theta\}[[r_1]],\, \Phi_0(\theta)\equiv i \sqrt{\frac{1+b}{a}}.
\end{align}
The manifold \eqref{psi} is symmetric with respect to $\mathcal S_1$ in the following sense:
\begin{align*}
 \Phi(-r_1,-\theta+\pi/2)=\Phi(r_1,\theta),
\end{align*}
with equality understood in $\mathbb C\{\theta\}[[r_1]]$.
% for any $b\in \mathbb R\setminus [-1,0]$.
\end{lemma}
\begin{proof}
 We obtain the statement by writing the set \eqref{formal} in the $(r_1,\rho_1,\theta)$-coordinates. For this we use the formal conjugacy \eqref{near_identity} with $N=\infty$ to obtain that
 \begin{align*}
  \widetilde r_1^2 &=  r_1^2 (1+ \mathcal O(r_1^2)),\\
  \widetilde \rho_1^2 &= \rho_1^2 + \mathcal O(r_1^2),
 \end{align*}
where both $\mathcal O(r_1^2)\in r_1^2\mathbb R\{\rho_1,\theta\}[[r_1]]$. We then obtain the resulting set by solving the equation
\begin{align*}
 F(r_1,\rho_1,\theta):=\rho_1^2 +\mathcal O(r_1^2) - \widetilde \Psi(r_1^2(1+\mathcal O(r_1^2))) =0,
\end{align*}
with $F\in \mathbb C\{\rho_1,\theta\}[[r_1]]$
for $\rho_1$ as a function of $(r_1,\theta)$. We have
\begin{align*}
 F\left(0,i\sqrt{\tfrac{1+b}{a}},\theta \right)= 0,\quad F'_{\rho_1}\left(0,i\sqrt{\tfrac{1+b}{a}},\theta \right)= -\frac{2(1+b)}{a}\ne 0.
\end{align*}
% .
We conclude that
\begin{align}
 \rho_1 = \Phi(r_1,\theta),\quad \Phi(0,\theta) = i \sqrt{\frac{1+b}{a}},\eqlab{sol}
\end{align}
with $\Phi\in \mathbb R\{\theta\}[[r_1]]$,
by the formal version of the implicit function theorem. Finally, we notice that $F$ is invariant with respect to the symmetry $\mathcal S_1$: $F(-r_1,\rho_1,-\theta+\pi/2)=F(r_1,\rho_1,\theta)$. The solution \eqref{sol} is the unique solution with $$\rho_1= i \sqrt{\frac{1+b}{a}},$$ for $(r_1,\theta)=(0,\frac{\pi}{4}+n\pi)$, $n\in \mathbb Z$ (which defines the fixed-point set of the symmetry). From this we conclude that the solution is symmetric.
% To complete the proof of the statement, we use that $\rho=x\rho_1$. % and the fact that the system is invariant with respect to $(x,\rho,\theta)\mapsto (x,-\rho,\theta+\pi)$.
\end{proof}

We now finally turn to the question of bringing \eqref{inner2} into the general form \eqref{mainvf}. For this we proceed as in \cite[Proposition 1]{bonckaert2008a} by first truncating the series \eqref{psi} and then applying a blowup.
Fix any $N\in \mathbb N$. We then define the partial sum
% \begin{align*}
% \Phi(x,\theta) = i \sqrt{-\Psi(x,\theta)}:=i \sum_{\alpha\in \mathbb N_0} \Phi_\alpha(\theta) x^\alpha\in \mathbb C\{\theta\}[[x]],
% \end{align*}
\begin{align*}
 \Phi^{N}(r_1,\theta) := \sum_{\alpha=0}^N \Phi_\alpha(\theta) r_1^\alpha,
\end{align*}
% for any $N\in \mathbb N$,
which is analytic on $B_\tau\times \mathbb T_\xi$ for $\tau>0,\xi>0$, both small enough.

\begin{lemma}\lemmalab{lemma24}
Suppose that $b\in \mathbb R\setminus [-1,0]$ and $a\in \mathbb R \setminus\{0\}$. Then for $M\in \mathbb N$ sufficiently large, $M\gg 1$, the blowup transformation $(r_1,\widetilde \rho_1,\theta)\mapsto (r_1,\rho_1,\theta)$ defined by
\begin{align}
 \rho_1 = \Phi^{2M}(r_1,\theta)+r_1^M \widetilde \rho_1,\eqlab{blowuprho1}
\end{align}
brings \eqref{entry1eqns} into the following \textnormal{prepared analytic form}:
\begin{equation}\eqlab{prepared_form}
 \begin{aligned}
  r_1' &= b r_1^2+r_1^4 Q(r_1, \widetilde\rho_1,\theta),\\%x^2\left(-b + x X( x)\right)+x^{3} F(x,r_1,\theta),\\%( x,r_1 \cos \theta, r_1 \sin \theta),\\
  \widetilde \rho_1'&= b \lambda r_1 \widetilde \rho_1 + r_1^2 R(r_1, \widetilde \rho_1,\theta),\\%r_1 r_1 \left(b\lambda+ r_1 \Lambda(r_1)\right)+r_1^{3} R(r_1,r_1,\theta),\\
  \theta' &=1,
%   r_1'&= -b r_1  y +   z+ \Lambda_N( r_1, y^2+ z^2)  y +  \Omega_N( r_1, y^2+ z^2)  z+ G_N( r_1, y, z),\\
%   \theta'&= -1-  x^2 \Omega_{N,1}( x, r_1^2)  z+ H_{N,1}( x, y, z).
\end{aligned}
\end{equation}
where
% and
\begin{align*}
 \lambda=-2b^{-1}(1+b)-M<0,
\end{align*}
after division of the right hand side by a nonzero quantity for $(r_1,\widetilde \rho_1,\theta)\in B_\tau\times B_\tau \times \mathbb T_\xi$ with $\tau,\xi>0$, both small enough.  The system is reversible  with respect to ${\mathcal S}_1\,:\,(r_1,\widetilde \rho_1,\theta)\mapsto (-r_1,\widetilde \rho_1,-\theta+\pi/2)$:
\begin{align*}
 \begin{cases} Q(-r_1,\widetilde \rho_1,-\theta+\pi/2) = Q(r_1,\widetilde \rho_1,\theta),\\
  R(-r_1,\widetilde \rho_1,-\theta+\pi/2) = -R(r_1,\widetilde \rho_1,\theta),
 \end{cases}
 \end{align*}
for all $(r_1,\widetilde \rho_1,\theta)\in B_\tau\times B_\tau \times \mathbb T_\xi$.
% Moreover, $X,F$ and $R$ are new analytic functions.
\end{lemma}
\begin{proof}
Since $\rho_1=\Phi^{N}(r_1,\theta)$ defines a formal invariant manifold for $N\to \infty$, it follows by construction that
\begin{align*}
 P^N(r_1,\theta):=&r_1^{-N-1}\bigg( r_1 \Phi^N \left(1+b+a(\Phi^N)^2+r_1 F_1\right)+ r_1^2G_1 \cos \theta+r_1^2 H_1\sin \theta \\
 &-\frac{\partial \Phi^N}{\partial r_1} r_1^2 (-1 - a (r_1\Phi^N)^2 - r_1F_1)- \frac{\partial \Phi^N}{\partial \theta} \left(1+(\Phi^N)^{-1}r_1^2\left(-G_1\sin \theta+H_1\cos \theta\right)\right)\bigg),
\end{align*}
is well-defined and analytic on $B_\tau \times \mathbb T_\xi$ for any $N\in \mathbb N$ for $\tau>0$, $\xi>0$ small enough. Here $W_1=W_1(r_1,\Phi^N, \theta)$, recall \eqref{W1defn}. ( Notice, in particular that $P^N=0$ if and only if $\rho_1=\Phi^{N}(r_1,\theta)$ is invariant.)
Then by applying \eqref{blowuprho1} (corresponding to $N=2M$) to \eqref{inner2}, we find the following equations
\begin{align*}
r_1' &= r_1^2 (b+F^M(r_1,r_1^M \rho_1,\theta)),\\
\rho_1 ' &= r_1 A^M(r_1,r_1^M \rho_1,\theta) \rho_1 + r_1^{M+1} P^{2M}(r_1,\theta),\\
\theta' &= 1+ r_1^2 \Theta^M(r_1,r_1^M\rho_1,\theta),
\end{align*}
with $F^M(0,0,\theta)=0$ and $A^M(0,0,\theta) = -2(1+b)-Mb$. This follows from a simple calculation using the mean-value theorem. We now divide the right hand side by $1+r_1^2\Theta^M(r_1,r_1^M\rho_1,\theta)$ which is nonzero on $B_\tau^2 \times \mathbb T_\xi$ for $\tau>0$, $\xi>0$, both small enough.
The result then follows from a simple expansion of the right hand side.
In particular, we use the reversible symmetry \eqref{symmetryzerohopf} to deduce that the $r_1^3$-term in the $r_1$-equations is absent.% First, we consider \eqref{model0_nfN} (with tildes dropped) for fixed $N$ (large enough) in the cylindrical coordinates $(x,r\cos \theta,r\sin \theta)$ and then blowup $(x,r)=(0,0)$ by setting
\end{proof}
Notice that \eqref{prepared_form} takes the form \eqref{mainvf} (with $F_0\equiv 0$) after setting $$x:=br_1, \quad \mathbf y: = \widetilde \rho_1.$$ Now, recall the definition of $S^\pm=S^\pm(\delta,\chi)$ in \eqref{Spm}. We fix any $\chi \in (0,\pi)$.
% We prove our result by working on this system. In particular, we look for invariant manifold solutions of \eqref{prepared_form} of the graph form
% \begin{align*}
%  r_1 = \Psi(x,\theta),
% \end{align*}
% which are solutions of the associated invariance equation:
% \begin{align*}
% -b x^2 \Psi'_x -\Psi'_\theta  -b\lambda x \Psi + F_0 x^3 \Psi'_x =-x^2 \Psi'_x (x^2 F)+x^2 R,
% \end{align*}
% which we solve by Borel-Laplace.
% In the following, we define
% \begin{align*}
%  b^{-1} S^\pm := \{x\in \mathbb C\,:\, bx \in S^\pm\},
% \end{align*}
% recall \eqref{Spm}.
\begin{proposition} \proplab{prop25}
Suppose that $b\in \mathbb R\setminus [-1,0]$, $a\in \mathbb R \setminus\{0\}$. Then there exists two invariant manifold solutions of \eqref{entry1eqns} of the form
 \begin{align}\eqlab{zerohopfmanifold}
  \rho_1 = \Phi^\pm (r_1,\theta),\quad r_1\in S^\pm,\,\theta\in \mathbb T_\xi,
 \end{align}
with $\delta>0$, $\xi>0$, both sufficiently small. Here $\Phi^\pm\,:\,S^\pm \times \mathbb T_\xi\to \mathbb C$ are the $1$-sums (with respect to $r_1$, uniformly with respect to $\theta$) of the formal series solution \eqref{psi} along the directions defined by $r_1\in S^\pm$, respectively. The two invariant manifolds in \eqref{zerohopfmanifold} are related by the symmetry $\mathcal S_1$:
\begin{align}
 \Phi^+ (r_1,\theta)=\Phi^-(-r_1,-\theta+\pi/2)\quad \forall\, r_1\in S^+,\,\theta\in \mathbb T_\xi.\eqlab{Phipm}
\end{align}

\end{proposition}
\begin{proof}
 We first apply \thmref{main} to \eqref{prepared_form}. This gives invariant manifold solutions of the form
 \begin{align*}
  \widetilde \rho_1 = \widetilde \rho_1^\pm (r_1,\theta),\quad r_1\in S^\pm,
 \end{align*}
 using that $b\in \mathbb R\setminus [-1,0]$ is real.
 We then obtain the desired manifold \eqref{zerohopfmanifold} by blowing back down using \eqref{blowuprho1}. The property \eqref{Phipm} is a consequence of the symmetry and the uniqueness of $\Phi^\pm$ (as $1$-sums on sectorial domains with opening greater than $\pi$ cf. Watson's lemma \cite[Proposition 11]{balser2000a}).
\end{proof}

We emphasize that \propref{zerohopf} follows from \propref{prop25} upon blowing down using \eqref{entry1}.
% We emphasize that there are separate invariant manifolds given by
% \begin{align*}
% \rho_1 =   \Phi^\pm (-r_1,-\theta+\pi/2),\quad r_1\in  S^\pm,\,\theta\in \mathbb T_\xi.
% \end{align*}
% due to reversible symmetry \eqref{symmetryzerohopf}.
% % This is a result of the symmetry of \eqref{inner2} with respect to

\section{The reversible and resonant Hopf-Hopf bifurcation}\seclab{hopfhopf}

In this section, we consider the real-analytic unfolding of the reversible and resonant Hopf-Hopf singularity with the normal form \eqref{nfhopfhopf0}, repeated here for convenience:
\begin{equation}\eqlab{nfhopfhopf}
\begin{aligned}
\dot x &=-(1+\Omega)y+z+F,\\
\dot y &=(1+\Omega)x+w+G,\\
\dot z &= -(1+\Omega)w+\Gamma x+H, \\
\dot w &=(1+\Omega )z+\Gamma y+J,
\end{aligned}
\end{equation}
with
\begin{align*}
W=W(x,y,z,w,\epsilon)=\mathcal O(\vert (x,y,z,w,\epsilon)\vert^5),\quad W=F,G,H,J,
\end{align*}
and where
\begin{align}\eqlab{Gamma}
\begin{cases}
\Gamma(\rho^2,L,\epsilon) =  \epsilon - b \rho^2 +c L,\\
 \Omega(\rho^2 ,L,\epsilon)=\alpha \epsilon + \beta \rho^2 +\gamma L,
 \end{cases}
\end{align}
for
\begin{align*}
  \rho^2:=x^2+y^2,\,L:=wx-yz,\end{align*}
see \cite[Lemma 3.17, p. 215]{haragus2011a}. The system is assumed to be real-analytic on $B_\tau^4\subset \mathbb C^4$, $\epsilon\in (-\epsilon_0,\epsilon_0)$, $\tau>0$, $\epsilon_0>0$, both sufficiently small, and reversible with respect to the involution
\begin{align}
 \mathcal S\,:\,(x,y,z,w)\mapsto (x,-y,-z,w),\eqlab{symmetry1}
\end{align}
i.e.
\begin{align*}
 \begin{cases} W(x,-y,-z,w,\epsilon)=-W(x,y,z,w,\epsilon), \quad W=F,H\\
  Q(x,-y,-z,w,\epsilon)=Q(x,y,z,w,\epsilon),\quad Q=G,J.
 \end{cases}
\end{align*}
for all $\epsilon\in (-\epsilon_0,\epsilon_0)$. We are mainly interested in $b>0$ and $\epsilon\in [0,\epsilon_0)$.

We apply the following change of coordinates $(\rho,\psi,L,\theta)\mapsto (x,y,z,w)$ defined by:
\begin{align}\eqlab{symplecticpolar}
\begin{cases}
 x = \rho \cos \theta,\\
 y =\rho \sin \theta,\\
 z =\psi \cos \theta - L\rho^{-1} \sin \theta,\\
 w = \psi \sin \theta+ L\rho^{-1} \cos \theta.
 \end{cases}
\end{align}
(This is the (symplectic polar) coordinates used in \cite{glebsky1995a} for the Swift-Hohenberg equation.)
This brings \eqref{nfhopfhopf} into the following form
\begin{equation}\eqlab{nfhopfhopf2}
\begin{aligned}
\dot \rho &=\psi + F \cos \theta+G\sin \theta,\\
\dot \psi &= \Gamma \rho + \rho^{-3} L^2 + \rho^{-2} \bigg( (H \rho^2 + GL)\cos \theta+ (J\rho^2 - FL)\sin \theta\bigg),\\
\dot L &=\rho^{-1} \left(- (G \psi \rho-J \rho^2-FL) \cos \theta+ (F\psi \rho-H\rho^2 +GL)\sin \theta \right),\\
\dot \theta &=1+\Omega + \rho^{-2} L +\rho^{-1}\left(-F\sin \theta+G\cos \theta\right),
\end{aligned}
\end{equation}
where we now (by slight abuse of notation) have
\begin{align*}
W=W(\rho,\psi,L\rho^{-1},\theta,\epsilon)=\mathcal O(\vert (\rho,\psi,L\rho^{-1},\epsilon)\vert^5),\quad W=F,G,H,J,
\end{align*}
% The system
%  is now reversible with respect to %symmetric with respect to both involutions
% \begin{align}\eqlab{symmetry1}
% %  \mathcal S_0\,:\,(\rho,\theta,\psi,L)\mapsto(-\rho,\theta+\pi,-\psi,L),\quad
%  \mathcal S\,:\,(\rho,\theta,\psi,L)\mapsto(\rho,-\theta,-\psi,L)
% \end{align}
% We define
% \begin{align}\eqlab{symmetryS}
% \mathcal S = \{\operatorname{Id},\mathcal S_0,\mathcal S,\mathcal S\circ \mathcal S_0\},
% \end{align} as the group of symmetries.

% To motivate our definition of the inner problem of \eqref{nfhopfhopf2} as the unperturbed problem (i.e. $\epsilon=0$), we again first follow Lazutkin's approach. For this,
In the following, we first consider the truncated system obtained by setting $F=G=H=J=0$:
\begin{equation}\eqlab{trunchopfhopf}
\begin{aligned}
 \dot \rho &=\psi,\\
\dot \psi &= \Gamma \rho + \rho^{-3} L^2, \\
\dot L &=0,\\
\dot \epsilon &=0,
\end{aligned}
\end{equation}
and $\dot \theta=1+\Omega + \rho^{-2} L$, which decouples. We see that $L$ is conserved. Notice that we (in anticipation of a blowup) have added an equation for $\epsilon$.
For $\epsilon=L=0$, we have
\begin{align*}
 \dot \rho &=\psi, \\
 \dot \psi &= -b \rho^3,
\end{align*}
recall \eqref{Gamma},
which has a nilpotent singularity at the origin. This motivates the following blowup:
\begin{align} \eqlab{blowuphopfhopf}
 r\ge 0,\,(\breve \rho,\breve \psi,\breve L,\breve \epsilon)\in \mathbb S^3\mapsto \begin{cases}
  \rho =r \breve \rho,\\
  \psi =r^2 \breve \psi,\\
  L=r^3\breve L,\\
  \epsilon = r^2 \breve \epsilon,
 \end{cases}
\end{align}
see also \eqref{directionalformel} below.
We first consider the scaling chart $\breve \epsilon=1$ with chart-specific coordinates $(\rho_2,\psi_2,L_2,r_2)$ defined by
\begin{align}\eqlab{hopfhopf2}
 \begin{cases}
  \rho =r_2 \rho_2,\\
  \psi = r_2^2 \psi_2,\\
  L = r_2^3 L_2,\\
  \epsilon = r_2^2.
 \end{cases}
\end{align}
This gives the following equations
\begin{equation}\eqlab{rho2psi2L2eqns}
\begin{aligned}
 \rho_2' &=\psi_2,\\
 \psi_2' &=(1-b \rho_2^2)\rho_2 +\rho_2^{-3} L_2^2+cr_2 \rho_2 L_2,\\
 L_2' &=0,
\end{aligned}
\end{equation}
and $r_2'=0$,
after dividing the right hand side by $r_2$. Within $L_2=0$, we have
\begin{equation}\eqlab{rho2psi2}
\begin{aligned}
 \rho_2' &= \psi_2,\\
 \psi_2' &=(1-b\rho_2^2)\rho_2.
\end{aligned}
\end{equation}
This is the Duffing equation (see \cite[Section 2.2]{Guckenheimer97}) with a hyperbolic saddle at the origin. There are two centers at $(\rho_2,\psi_2)= (\pm \sqrt{b^{-1}},0)$ for any $b>0$.  The phaseportrait is illustrated in \figref{hopfhopf}. We focus on the homoclinic $\Gamma_0$ with $\rho_2>0$.
It has the following parametrization
\begin{align*}
 \begin{cases} \rho_2 (t)=\sqrt{\frac{2}{b}}\operatorname{sech}(t),\\%\frac{1}{2\sqrt{2\eta}} \frac{\sinh \frac{t}{2}}{\cosh^2\frac{t}{2}},\\
 \psi_2(t) =-\sqrt{\frac{2}{b}}\operatorname{sech}(t)\tanh(t), %\frac{1}{\sqrt{2\eta}} \frac{1}{\cosh \frac{t}{2}}.
 \end{cases}
\end{align*}
with poles in the complex plane at $t=i\frac{\pi}{2}+n\pi$, $n\in\mathbb Z$, of order $1$ and $2$, respectively.

\begin{remark}
 It is again easy to see that higher order terms of \eqref{nfhopfhopf} lead to regular perturbation terms of order $\mathcal O(r_2)$ of \eqref{rho2psi2L2eqns} (on compact domains with $\rho_2$ bounded uniformly away from $0$). The symmetry $\mathcal S$ takes the following form $(\rho_2,\psi_2,L_2,\theta)\mapsto (\rho_2,-\psi_2,L_2,-\theta)$ in the $(\rho_2,\psi_2,L_2,\theta)$-coordinates. In this way, one can obtain (as in \remref{homoclinic}) existence of two symmetric homoclinic orbits for all $0<r_2\ll 1$ (due to the intersection of the stable and unstable manifolds with $\psi_2=0$, $\theta=n\pi$, $n\in \mathbb Z$, being the fixed-point set of the symmetry). The interesting problem (from the perspective of exponentially small phenomena) is then (again) an asymptotic formula for the splitting of tangent spaces and whether there are other non-symmetric homoclinic orbits. The former question is addressed in \cite{gaivao2011a} for a class of Hamiltonian systems that generalize the Swift-Hohenberg equation. Here the authors argues for the existence of an exponentially small asymptotic formula for the splitting of the tangent vectors (expressed in terms of the symplectic form). We aim to study this phenomena within the context of the general normal form using the geometric approach of the present paper and \cite{bkt,new} in future work.
\end{remark}

\begin{figure}[h!]
\begin{center}
{\includegraphics[width=.65\textwidth]{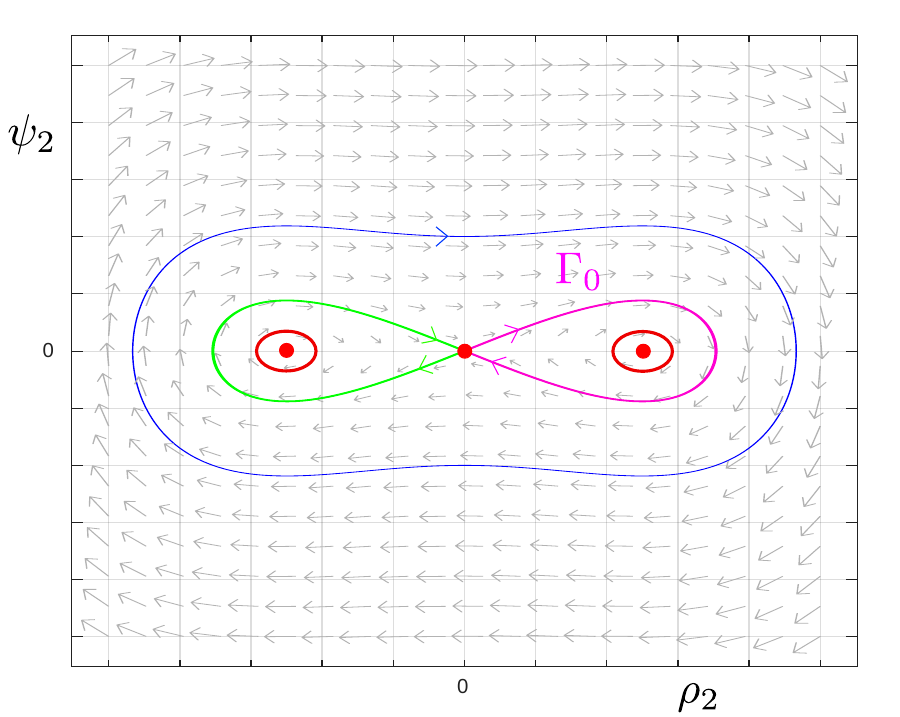}}
\end{center}
\caption{Phase portrait of \eqref{rho2psi2} for $b>0$.}
\figlab{hopfhopf}
% Remark: If $U$ smooth for $|u|<\nu$, the compact manifolds lie inside $|u|<\nu$ for large $n$
\end{figure}
To follow Lazutkin's approach one zooms in around the pole at $t=-i\frac{\pi}{2}$ (the pole at $t=i\frac{\pi}{2}$ is obtained by conjugation) by writing  $$t=t(s):=-i\frac{\pi}{2}+r_2 s,$$ and look for invariant manifolds parameterized by $s$ and $\theta$ as $r_2\to 0$. For $s\ne 0$, we have that
\begin{align*}
\begin{cases}
 \rho_2(t(s)) =i\sqrt{\frac{2}{b}}\operatorname{csch}(r_2 s)= i\sqrt{\frac{2}{b}} (r_2s)^{-1}(1+ O(r_2^2)),\\ %-\frac{i}{2\sqrt{2\eta}} \coth \left(\frac{r_2s}{2}\right)\operatorname{csch}\left(\frac{r_2s}{2}\right) = -\frac{2i}{\sqrt{2\eta} r_2^2 s^2} \left(1+\mathcal O(r_2)\right),\\
 \psi_2(t(s)) =-i\sqrt{\frac{2}{b}} \operatorname{csch}(r_2 s)\operatorname{coth}(r_2 s) = -i\sqrt{\frac{2}{b}} (r_2s)^{-2}(1+ O(r_2^2))
 %-\frac{i}{\sqrt{2\eta}} \operatorname{csch}\left(\frac{r_2 s}{2}\right)=-\frac{2i}{\sqrt{2\eta}r_2 s}\left(1+\mathcal O(r_2)\right),
\end{cases}
\end{align*}
as $r_2\to 0$ for $s\ne 0$, or in terms of $(\rho,\psi)$ by \eqref{hopfhopf2}:
\begin{align}\eqlab{t2}
\begin{cases}
 \rho(t(s)): = r_2 \rho_2(t(s)) \to i\sqrt{\frac{2}{b}} s^{-1},\\
 \psi(t(s)):=r_2^2 \psi_2(t(s)) \to -i\sqrt{\frac{2}{b}} s^{-2}.
\end{cases}
\end{align}
We then conclude that
\begin{align}
 \psi = i\sqrt{\frac{b}{2}}\rho^2,\quad \psi\in \mathbb C,\quad L=0,\quad \epsilon=0,\eqlab{this0}
\end{align}
defines an invariant set of \eqref{trunchopfhopf}. This can also easily be verified by direct insertion.
% We obtain
% $ \psi = -i\sqrt{\frac{b}{2}}\rho^2$
% upon applying the symmetry $\mathcal S_0$.
% with $L=\epsilon=0$.

This motivates the definition of the inner problem associated with $\Gamma_0$ as the unperturbed version of \eqref{nfhopfhopf2} (i.e. for $\epsilon=0$).
% \begin{align*}
% \begin{equation}\eqlab{Heqns}
% \begin{aligned}
%  \dot \rho &=\left(-\mu L +2\eta J\right) \psi+U,\\
%  \dot \theta &=1 -\mu J + V\rho^{-1},\\
%  \dot \psi &=-\rho+W\rho^{-2}L+X ,\\
%  \dot L &=Z,
%  \end{aligned}
% \end{equation}
% \begin{equation}\eqlab{nfhopfhopf2}
% \begin{aligned}
% \dot \rho &=\psi + F \cos \theta+G\sin \theta,\\
% \dot \theta &=1+\Omega + \rho^{-2} L +\rho^{-1}\left(-F\sin \theta+G\cos \theta\right),\\
% \dot \psi &= \Gamma \rho + \rho^{-3} L + \rho^{-2} \bigg( (H \rho^2 + GL)\cos \theta+ (J\rho^2 - FL)\sin \theta\bigg),\\
% \dot L &=\rho^{-1} \left(- (G \psi \rho-J \rho^2-FL) \cos \theta+ (F\psi \rho-H\rho^2 +GL)\sin \theta \right),\\
% \dot \epsilon &=0,
% \end{aligned}
% \end{equation}
% where
% \begin{align*}
% Q=Q(\rho,\theta,\psi,\rho^{-1}L) = \mathcal O(\vert (\rho,\psi,\rho^{-1} L)\vert^4),\quad Q=U,V,W,X,Z,
% \end{align*}
% are analytic functions on $B_\tau\times \mathbb T_\xi\times B_\tau\times B_\tau$, for $\tau>0$, $\xi>0$ both small enough. Here we have used that $R$ in \eqref{H} can be written as an analytic function of $(\rho,\theta,\psi,\rho^{-1} L)\in B_\tau\times \mathbb T_\xi\times B_\tau\times B_\tau$ in the new variables. The existence of $U,V,W,X,Z$ then follows from the chain rule.
% \end{align*}
In particular, we are concerned with invariant manifold solutions of the graph form
\begin{align*}
 \begin{cases}
 \psi= \psi(\rho,\theta),\\
 L=L(\rho,\theta),\end{cases} \quad \rho\in B_\tau,\,\theta\in \mathbb T_\xi,
\end{align*}
with
\begin{align*}
\begin{cases}
 \psi(\rho,\theta)=i \sqrt{\frac{b}{2}}\rho^2(1+\mathcal O(\rho)),\\ L(\rho,\theta)=\mathcal O(\rho^4).
 \end{cases}
\end{align*}
These manifolds are naturally parameterized in the $(r_1,\psi_1,L_1,\epsilon_1)$-coordinates of the $\breve \psi = 1$-chart associated with the blowup \eqref{blowuphopfhopf}:
\begin{align}\eqlab{entry1hopfhopf}
 \begin{cases}
  \rho = r_1\\
  \psi = r_1^2\psi_1,\\
  L = r_1^3 L_1,\\
  \epsilon =r_1^2\epsilon_1.
 \end{cases}
\end{align}
% Indeed, the set \eqref{this} takes the following form
%
%
%
% \begin{align}\eqlab{entry1}
%   \begin{cases}
%    x = r_1,\\
%    \rho = r_1 \rho_1,\\
%    \epsilon = r_1 \epsilon_1.
%   \end{cases}
% \end{align}
Indeed, here \eqref{this0} takes the following form
\begin{align}
 \psi_1 =i\sqrt{\frac{b}{2}},\quad r_1\in \mathbb C,\quad L_1=0,\quad \epsilon_1=0.\eqlab{this2}
\end{align}
The change of coordinates defined by \eqref{entry1hopfhopf} brings \eqref{nfhopfhopf2} with $\epsilon=0$ into the following form:
\begin{equation}\eqlab{innerhopfhopf}
\begin{aligned}
\dot r_1 &=r_1^2 \left(\psi_1+r_1^2U_1\right),\\
\dot \psi_1 &=r_1 (-b-2\psi_1^2+L_1^2)+r_1^2 Z_1,\\
\dot L_1 &=-3r_1 \psi_1 L_1+r_1^2 W_1,\\
\dot \theta &=1+r_1 L_1 +r_1^2 \beta+ r_1^3 V_1,
%  \dot x  &=x^2 \left(-1-a \rho_1^2 + x F_1\right),\\
%  \dot \rho_1 &=x\rho_1\left(1+b+a\rho_1^2+x F_1\right)+x^2 G_1 \cos \theta+x^2 H_1\sin \theta,\\
%  \dot \theta &=1+\rho_1^{-1} x^2 \left(-G_1\sin \theta+H_1\cos \theta\right),
\end{aligned}
\end{equation}
within $\epsilon_1=0$. Here the functions $Q_1=Q_1(\rho_1,r_1,L_1,\theta)$, $Q=U,V,Z,W$, are analytic on $B_\tau^3 \times \mathbb T_\xi $ with $\tau>0$, $\xi>0$ both small enough. They can obviously be expressed in terms of $F,G,H,J$ but the details of this will not be important here.
For \eqref{innerhopfhopf}, we look for invariant manifolds of the graph form
\begin{align}\eqlab{this}
\begin{cases}
\psi_1=\Phi^{\psi}(r_1,\theta),  \\ L_1=\Phi^L(r_1,\theta),\end{cases}
\end{align}
with
\begin{align*}
\begin{cases}
 \Phi^\psi(0,\theta) =i\sqrt{\frac{b}{2}},\\ \Phi^L(0,\theta)=0.\end{cases}
\end{align*}
The system \eqref{innerhopfhopf} is reversible with respect to
\begin{align}\eqlab{symmery1hopfhopf}
 \mathcal S_1\,:\,(r_1,\psi_1,L_1,\theta)\mapsto(r_1,-\psi_1,L_1,-\theta),
\end{align}
derived from \eqref{symmetry1},  \eqref{symplecticpolar} and \eqref{entry1hopfhopf}. Hence \eqref{this} is an invariant manifold solution if and only if
\begin{align*}
\begin{cases}
\psi_1=-\Phi^{\psi}(r_1,-\theta),  \\ L_1=\Phi^L(r_1,-\theta),\end{cases}
\end{align*}
is an invariant manifold solution. %Notice that these solutions are necessarily different since $\Phi^\psi(0,\theta)\ne 0$ for $b\ne 0$.

\subsection{Normal form}
In this section, we show that \eqref{innerhopfhopf} can be brought into the general form \eqref{mainvf}.
We proceed as in \secref{normalformzerohopf} and first prove existence of formal series solutions. For this we use the following result from \cite[Lemma 3.17, p. 205]{haragus2011a}: There exists a formal change of coordinates:
\begin{align}\eqlab{Tinftyhopfhopf}
 (\widetilde x,\widetilde y,\widetilde z,\widetilde w) = ( x, y, z, w) + T ( x, y, z, w),
\end{align}
with $T\in \mathbb R[[( x, y, z, w)]]$ starting with terms of order $5$, which brings \eqref{nfhopfhopf} with $\epsilon=0$ into the following formal normal form:
\begin{equation}\eqlab{nfhopfhopfformal}
\begin{aligned}
\dot{\widetilde x} &=-(1+\widetilde \Omega)\widetilde y+\widetilde z,\\
\dot{\widetilde y} &=(1+\widetilde \Omega)\widetilde x+\widetilde w,\\
\dot{\widetilde z} &= -(1+\widetilde \Omega)\widetilde w+\widetilde \Gamma \widetilde x, \\
\dot{\widetilde w} &=(1+\widetilde \Omega )\widetilde z+\widetilde \Gamma \widetilde y,
\end{aligned}
\end{equation}
where
\begin{align*}
 \widetilde Q \in \mathbb R[[\widetilde \rho^2,\widetilde L]],\quad \widetilde Q(0,0) = 0,\quad Q=\Omega,\Gamma,
\end{align*}
with
\begin{align*}
 \widetilde \rho^2 = \widetilde x^2 + \widetilde y^2,\quad \widetilde L = \widetilde w\widetilde x-\widetilde z\widetilde y.
\end{align*}
The formel change of coordinates defined by \eqref{Tinftyhopfhopf} is equivariant with respect to $\mathcal S$.
We have
\begin{align}\eqlab{Gamma1}
\widetilde \Gamma(\widetilde \rho^2,0)=- b \widetilde \rho^2 +\widetilde \rho^4 \widetilde \Gamma_1(\widetilde \rho^2),\quad \widetilde \Gamma_1\in \mathbb R[[\widetilde \rho^2]].
\end{align}
We consider the change of coordinates $(\widetilde \rho,\widetilde \psi,\widetilde L,\widetilde \theta)\mapsto (\widetilde x,\widetilde y,\widetilde z,\widetilde w)$ defined by \eqref{symplecticpolar} (with tildes added), which brings \eqref{nfhopfhopfformal} into the following system
\begin{equation}\nonumber
\begin{aligned}
\dot{\widetilde \rho} &=\widetilde \psi,\\
\dot{\widetilde \psi} &= \widetilde \Gamma (\widetilde \rho^2,\widetilde L)\widetilde \rho + \widetilde \rho^{-3} \widetilde L^2,\\
\dot{\widetilde L} &=0,
\end{aligned}
\end{equation}
and $\dot{\widetilde \theta} =1+\widetilde{\Omega} + \widetilde \rho^{-2} \widetilde L$, which decouples. This follows from a simple calculation.
In particular, within $\widetilde L=0$ we have
\begin{equation}\nonumber
\begin{aligned}
\dot{\widetilde \rho} &=\widetilde \psi,\\
\dot{\widetilde \psi} &= \widetilde \Gamma(\widetilde \rho^2,0)\widetilde \rho,
\end{aligned}
\end{equation}
with a nilpotent singularity at the origin (cf. \eqref{Gamma1}). We therefore apply the directional blowup:
\begin{align}\eqlab{directionalformel}
 (\widetilde r_1,\widetilde \psi_1)\mapsto \begin{cases}
                                            \widetilde\rho = \widetilde r_1,\\
                                            \widetilde \psi = \widetilde r_1^2 \widetilde \psi_1,
                                           \end{cases}
\end{align}
in line with \eqref{entry1hopfhopf}.
This gives
\begin{equation}\eqlab{formelhopfhopf}
\begin{aligned}
(\widetilde r_1^2)' &=2\widetilde r_1^2\widetilde \psi_1,\\
\widetilde \psi_1' &=-b-2\widetilde \psi_1^2 + \widetilde r_1^2 \widetilde \Gamma_1(\widetilde r_1^2),
\end{aligned}
\end{equation}
after division of the right hand side by $\widetilde r_1$ (desingularization). Here we have used \eqref{Gamma1}. Suppose first that $b<0$. Then $(\widetilde r_1^2,\widetilde \psi_1)=(0,\pm \sqrt{\frac{-b}{2}})$ are (formal) hyperbolic singularities and we have formal stable respectively unstable manifolds of the form:
\begin{align*}
 \widetilde \psi_1 = \pm \widetilde \Psi(\widetilde r_1^2),\quad \widetilde \Psi\in \mathbb R[[\widetilde r_1^2]], \quad \widetilde \Psi(0)=\sqrt{\frac{-b}{2}}.
\end{align*}
% For convenience, we illustrate the phaseportrait of \eqref{formelhopfhopf} (assuming that it is in convergent) for $b<0$ in \figref{formelhopfhopf}. The relevant manifolds are in magenta.
To see that the same is true for $b>0$, we replace $(\widetilde r_1,\widetilde \psi_1)$ by $i(\widetilde r_1,\widetilde \psi_1)$. This gives the following (real) system after multiplication of the right hand side by $i$:
\begin{align*}
 (\widetilde r_1^2)' &=-\widetilde r_1^2\widetilde \psi_1,\\
\widetilde \psi_1' &=-b+2\widetilde \psi_1^2 - \widetilde r_1^2 \widetilde \Gamma_1(-\widetilde r_1^2).
\end{align*}
Here
$(\widetilde r_1^2,\widetilde \psi_1)=(0,\pm \sqrt{\frac{b}{2}})$ are (formal) hyperbolic singularities and we therefore obtain the same conclusion as for $b<0$. We note that the two solutions (corresponding to $\pm$) are related by the symmetry $\mathcal S_1$.

\begin{lemma}
 Suppose that $b\in \mathbb R\setminus\{0\}$. Then there is a formal invariant manifold of \eqref{innerhopfhopf} of the form
 \begin{align*}
  \begin{cases}
   \psi_1 = \Phi^{\psi}(r_1,\theta),\\
   L_1 = \Phi^L(r_1,\theta),
  \end{cases}
 \end{align*}
with
\begin{align}\eqlab{Phihopfhopf}
\begin{cases}
\Phi^{\psi}(r_1,\theta) =\sum_{\alpha=0}^\infty \Phi^{\psi}_\alpha(\theta)r_1^\alpha \in \mathbb C\{\theta\}[[r_1]] ,\\
\Phi^{L}(r_1,\theta) =\sum_{\alpha=3}^\infty \Phi^{L}_\alpha(\theta)r_1^\alpha  \in r_1^3 \mathbb C\{\theta\}[[r_1]],
\end{cases}
\end{align}
and
\begin{align*}
 \Phi^\psi(0,\theta) = i\sqrt{\frac{b}{2}}.
\end{align*}
% The formel invariant manifold is symmetric with respect to $\mathcal S$:
\end{lemma}
\begin{proof}
We write the formal series solution
\begin{align}\eqlab{formalmanhopfhopf}\begin{cases} \widetilde \psi_1= \widetilde \Psi (\widetilde r_1^2),\\
               \widetilde L_1=0,
              \end{cases}
              \end{align}
              of \eqref{nfhopfhopfformal}
 in the $(r_1,\psi_1,L_1,\theta)$-coordinates. By combining \eqref{symplecticpolar}, \eqref{entry1hopfhopf}, and \eqref{Tinftyhopfhopf}, we obtain the following equations for \eqref{formalmanhopfhopf} in the $(r_1,\psi_1,L_1,\theta)$-coordinates:
\begin{align*}
\widetilde r_1\cos \widetilde \theta &= r_1\cos \theta+\mathcal O(r_1^5),\\
\widetilde r_1  \sin \widetilde \theta &= r_1 \sin \theta+\mathcal O(r_1^5),\\
\widetilde r_1^2\widetilde \Psi   \cos \widetilde \theta &= r_1^2\psi_1 \cos \theta-r_1^2 L_1\sin \theta+\mathcal O(r_1^5),\\
\widetilde r_1^2\widetilde \Psi \sin \widetilde \theta &= r_1^2\psi_1 \sin \theta+r_1^2 L_1 \cos \theta+\mathcal O(r_1^5),
\end{align*}
where all $\mathcal O(r_1^5)\in r_1^5 \mathbb R\{\psi_1,L_1,\theta\}[[r_1]]$. From the first two equations, we directly have that
\begin{align*}
 \widetilde r_1 = r_1\left(1 + \mathcal O(r_1^4)\right),
\end{align*}
with $\mathcal O(r_1^4)\in r_1^4 \mathbb R\{\psi_1,L_1,\theta\}[[r_1]]$.
Inserting this into the last two equations, we find that
\begin{align*}
 \widetilde \Psi(r_1^2 (1+\mathcal O(r_1^4)))\times ( \cos \theta +\mathcal O(r_1^4))  &= \psi_1 \cos \theta-L_1\sin \theta+\mathcal O(r_1^3),\\
\widetilde \Psi (r_1^2 (1+\mathcal O(r_1^4)))\times ( \sin \theta +\mathcal O(r_1^4)) &= \psi_1 \sin \theta+L_1\cos \theta+\mathcal O(r_1^3).
\end{align*}
But then using $\widetilde \Psi(0)= i\sqrt{\frac{b}{2}}$, we arrive at the following formal equations
\begin{align*}
 \begin{cases} \psi_1-i\sqrt{\frac{b}{2}} +\mathcal O(r_1^3)=0,\\
L_1+\mathcal O(r_1^3)=0.
\end{cases}
\end{align*}
This follows from a simple calculation.
We then complete the proof by solving these equations for $(\psi_1,L_1)$ (as formal series in $r_1$ with real-analytic $\theta$-dependent coefficient) using the formal version of the implicit function theorem. %upon using $\psi_1=i\sqrt{\frac{b}{2}}$, $L_1=0$ for $r_1=0$ by the formal version of the implicit function theorem.
% Subsequently, by multiplying the last two equations by $\widetilde r_1 \widetilde \Phi_\infty (\widetilde r_1^2)$ and equating the resulting right side with the right sides of the first two equations, we find that
% \begin{align*}
%  \left(\sqrt{1 + L_1^2\rho_1^{-2}}+\mathcal O(r_1^4)\right)  \widetilde \Phi_\infty \left(\cos \theta-L_1\rho_1^{-1}\sin \theta+\mathcal O(r_1^4)\right) =\rho_1 \cos \theta+\mathcal O(r_1^3),\\
%   \left(\sqrt{1 + L_1^2\rho_1^{-2}}+\mathcal O(r_1^4)\right)  \widetilde \Phi_\infty \left(\sin \theta+L_1\rho_1^{-1}\cos \theta+\mathcal O(r_1^4)\right) = \rho_1 \sin \theta+\mathcal O(r_1^3),
% \end{align*}
% with $\widetilde \Phi_\infty=\widetilde \Phi_\infty\left(r_1^2\left(1 + L_1^2\rho_1^{-2}+\mathcal O(r_1^4)\right)\right)$.
%

\end{proof}

We are now finally ready to bring \eqref{innerhopfhopf} into the general form \eqref{mainvf}. For this we define $\Phi^{\psi,N}$ and $\Phi^{L,N}$, $N\in \mathbb N$, $N\ge 3$, as the partial sums:
\begin{align*}
 \Phi^{q,N}(r_1,\theta):=\sum_{\alpha=0}^N \Phi_{\alpha}^q(\theta)r_1^\alpha,\quad q=\psi,L,
\end{align*}
recall \eqref{Phihopfhopf}.

\begin{lemma}\lemmalab{lemma33}
 Suppose that $b\ne 0$. Then for $M\in \mathbb N$ sufficiently large, $M\gg 1$, the blowup transformation $(r_1,\widetilde \psi_1,\widetilde L_1,\theta)\mapsto (r_1,\psi_1,L_1,\theta)$ defined by
 \begin{align}\eqlab{blowuprho1hopfhopf}
  \begin{cases}
   \psi_1 = \Phi^{\psi,2M}(r_1,\theta)+r_1^M \widetilde \psi_1,\\
   L_1 = \Phi^{L,2M}(r_1,\theta)+r_1^M \widetilde L_1,
  \end{cases}
 \end{align}
brings \eqref{innerhopfhopf} into the following \textnormal{prepared} analytic form:
\begin{equation}\eqlab{nfhopfhopffinal}
\begin{aligned}
r_1'&=r_1^2 i\sqrt{\frac{b}{2}}+r_1^3 F_0 + r_1^4 Q(r_1,\widetilde \psi_1,\widetilde L_1,\theta),\\
\widetilde \psi_1'&= i\sqrt{\frac{b}{2}}r_1 \lambda^1 \widetilde \psi_1 +r_1^2 Z(r_1,\widetilde \psi_1,\widetilde L_1,\theta)\\
\widetilde L_1'&= i\sqrt{\frac{b}{2}} r_1 \lambda^2 \widetilde L_1 +r_1^2 Z(r_1,\widetilde \psi_1,\widetilde L_1,\theta),\\
\theta' &=1,
\end{aligned}
\end{equation}
after division of the right hand side by a nonzero quantity for $(r_1,\widetilde \psi_1,\widetilde L_1,\theta)\in B_\tau^3 \times \mathbb T_\xi$. Here $\tau>0,\xi>0$, both small enough. The system is reversible with respect to $\mathcal S_1$ (see \eqref{symmery1hopfhopf}).
Finally, the following holds
\begin{align*}
 \lambda^1=-4-M<0,\quad \lambda^2 = -3-M<0.
\end{align*}

\end{lemma}
\begin{proof}
 Since $\psi_1=\Phi^{\psi,N}(r_1,\theta)$, $L_1=\Phi^{L,N}(r_1,\theta)$, defines a formal invariant manifold for $N\to \infty$, it follows by construction that
\begin{align*}
\widetilde Z^N(r_1,\theta):=&r_1^{-N-1}\bigg( r_1 (-b-2(\Psi^{\psi,N})^2+(\Psi^{L,N})^2)+r_1^2 Z_1 \\
 &-\frac{\partial \Phi^N}{\partial r_1} r_1^2 \left(\psi_1+r_1^2U_1\right)- \frac{\partial \Phi^N}{\partial \theta} \left(1+r_1 L_1 +r_1^2 \beta+ r_1^3 V_1\right)\bigg),\\
 \widetilde W^N(r_1,\theta):=&r_1^{-N-1}\bigg(-3r_1 \Psi^{\psi,N}\Psi^{L,N}+r_1^2 W_1 \\
 &-\frac{\partial \Phi^N}{\partial r_1} r_1^2 \left(\psi_1+r_1^2U_1\right)- \frac{\partial \Phi^N}{\partial \theta} \left(1+r_1 L_1 +r_1^2 \beta+ r_1^3 V_1\right)\bigg),
\end{align*}
are well-defined and analytic on $B_\tau \times \mathbb T_\xi$ for any $N\in \mathbb N$ for $\tau>0$, $\xi>0$, both small enough. Here $Q_1=Q_1(r_1,\Phi^{\psi,N},\Phi^{L,N},\theta)$, $Q_1=U,V,Z,W$, recall \eqref{W1defn}. (Notice, in particular that $(\widetilde Z^N,\widetilde W^N)=(0,0)$ if and only if $\psi_1=\Phi^{\psi,N}(r_1,\theta)$, $L_1=\Phi^{L,N}(r_1,\theta)$, is invariant.)
Then by applying \eqref{blowuprho1hopfhopf} (corresponding to $N=2M$) to \eqref{inner2}, we find the following equations
\begin{align*}
\dot r_1 &= r_1^2 \left(i\sqrt{\frac{b}{2}}+F^M(r_1,r_1^M \widetilde \psi_1,r_1^M \widetilde L_1,\theta)\right),\\
\dot{\widetilde \psi}_1 &= r_1 A^M(r_1,r_1^M \widetilde \psi_1,r_1^M \widetilde L_1,\theta) \widetilde \psi_1+r_1 B^M(r_1,r_1^M \widetilde \psi_1,r_1^M \widetilde L_1,\theta) \widetilde L_1 + r_1^{M+1} \widetilde Z^{2M}(r_1,\theta),\\
\dot{\widetilde L}_1 &= r_1 C^M(r_1,r_1^M \widetilde \psi_1,r_1^M \widetilde L_1,\theta) \widetilde \psi_1+r_1 D^M(r_1,r_1^M \widetilde \psi_1,r_1^M \widetilde L_1,\theta) \widetilde L_1 + r_1^{M+1} \widetilde W^{2M}(r_1,\theta),\\
\dot \theta &=1+r_1 \Theta^M(r_1, r_1^M \widetilde \psi_1,r_1^M\widetilde L_1,\theta),
% \theta' &= 1+ r_1^2 \Theta^M(r_1,r_1^M\rho_1,\theta),
\end{align*}
with $F^M(0,0,0,\theta)=0$, $\Theta^M(0,0,0,\theta)=0$, and
\begin{align*}
 \begin{cases}
  A^M(0,0,0,\theta)= -(4+M)i\sqrt{\frac{b}{2}},\\
B^M(0,0,0,\theta)=0,\\
C^M(0,0,0,\theta)=0,\\
  D^M(0,0,0,\theta)= -(3+M)i\sqrt{\frac{b}{2}},
 \end{cases}
\end{align*}
This follows from a simple calculation using the mean-value theorem. We now divide the right hand side by $$1+r_1\Theta^M(r_1,r_1^M\widetilde \psi_1,r_1^M \widetilde L_1,\theta),$$ which is nonzero on $B_\tau^3 \times \mathbb T_\xi$ for $\tau>0$, $\xi>0$, both small enough.
% and define
% \begin{align*}
%  \lambda_1 =  -(4+M),\quad \lambda_2=-(3+M).
% \end{align*}
% For $M\gg 1$, we have $\lambda<0$ and
The result then  follows from a simple expansion of the right hand side. (In fact, it is easy to see that $M=2$ suffices.)
\end{proof}

We see that
\begin{align}\eqlab{xhopfhopf}
x:=i\sqrt{\frac{b}{2}}r_1,\quad \mathbf y: = (\widetilde \psi_1,\widetilde L_1),
\end{align}
brings \eqref{nfhopfhopffinal} into the general form \eqref{mainvf}.
% In the following, we let $\upsilon S^\pm$, $c\in \mathbb C$, denote the sets
% \begin{align*}
%  cS^\pm : =\{x\in \mathbb C\,:\,c^{-1}x\in S^\pm\},
% \end{align*}
% recall \eqref{Spm}.
Recall that $\upsilon S^\pm$, $\upsilon\in \mathbb C$, denotes the sets
\begin{align*}
 \upsilon S^\pm : =\{x\in \mathbb C\,:\,\upsilon^{-1}x\in S^\pm\},
\end{align*}
with $S^\pm$ given by \eqref{Spm}.
\begin{proposition}\proplab{prop34}
 Suppose that $b\in \mathbb R\setminus\{0\}$ and
 define
 \begin{align}
  \upsilon:=i^{-1}\sqrt{\frac{2}{b}}.
 \end{align}
 Then for $\delta>0$, $\xi>0$, both sufficiently small, there exists two invariant manifold solutions of \eqref{innerhopfhopf} of the form:
 \begin{align}\eqlab{invhopfhopf}
  \begin{cases}
   \psi_1 = \Phi^{\psi,\pm}(r_1,\theta),\\
   L_1 = \Phi^{L,\pm}(r_1,\theta),
  \end{cases}\quad r_1\in \upsilon S^\pm,\,\theta\in \mathbb T_\xi,
 \end{align}
 where $\Phi^{\psi,\pm}(0,\theta)=i\sqrt{\frac{b}{2}}$, $\Phi^{L,\pm}(r_1,\theta)=\mathcal O(r_1^4)$, uniformly on the closure of $\upsilon S_\pm$, are
the $1$-sums of Gevrey-$1$ series in the directions defined by $r_1\in \upsilon S^\pm$.
\end{proposition}
\begin{proof}
 We first apply \thmref{main} to \eqref{nfhopfhopffinal}. This gives invariant manifolds of the form
\begin{align*}
  \begin{cases}
   \widetilde \psi_1 = \widetilde \Phi^{\psi,\pm}(r_1,\theta),\\
   \widetilde L_1 = \widetilde \Phi^{L,\pm}(r_1,\theta),
  \end{cases}\quad r_1\in \upsilon S^\pm,\,\theta\in \mathbb T_\xi,
 \end{align*}
 Here we have used \eqref{xhopfhopf} and that $b\ne 0$. We then obtain the desired manifold \eqref{invhopfhopf} by blowing back down using \eqref{blowuprho1hopfhopf}.
\end{proof}
By applying the symmetry $\mathcal S$, recall \eqref{symmery1hopfhopf}, we obtain a separate invariant manifolds solution:
\begin{align}\nonumber
  \begin{cases}
   \psi_1 = -\Phi^{\psi,\pm}(r_1,-\theta),\\
   L_1 = \Phi^{L,\pm}(r_1,-\theta),
  \end{cases}\quad r_1\in \upsilon S^\pm,\,\theta\in \mathbb T_\xi.
 \end{align}

 Finally, we emphasize that \propref{hopfhopf} follows from \propref{prop34} upon blowing down using \eqref{entry1hopfhopf}.

\section{A Banach-convolution-algebra-approach to Borel-Laplace}\seclab{borellaplace}
In this section, we review the Banach-convolution-algebra-approach to Borel-Laplace by Bonckaert and De Maesschalck \cite{bonckaert2008a} and extend it so that it can be applied to the PDE \eqref{main} (depending on $\theta\in \mathbb T_\xi$). We start by defining the appropriate Banach spaces and state the relevant properties  from \cite{bonckaert2008a} relating to the convolution algebra (see \lemmaref{convolution1}).% that is important for the Borel-Laplace approach.

\subsection{The Banach space $\mathcal G_\eta$}
% \note{need new symbol for $\theta$? $\varphi?$}
 We first consider analytic functions
 \begin{align}\eqlab{widehatW} \widehat W:\Omega \to \mathbb C,\end{align}  where $\Omega= \Omega(\kappa,\varphi,\nu)\subset \mathbb C$ is the star-shaped region defined by
\begin{align}\eqlab{Omega}
 \Omega = B_{\kappa}\cup S(\varphi,\nu),\quad 0<\nu<\pi.
\end{align}
see \figref{Omega}.
Here $B_\kappa\subset \mathbb C$, $\kappa>0$, is the open ball of radius $\kappa>0$ centered at the origin. Moreover,
$S(\varphi,\nu)\subset \mathbb C$ denotes the infinite sector of opening $\nu\in (0,\pi)$ centered along the direction defined by $\varphi\in \mathbb T$; notice that $S^\pm$ in \eqref{Spm} in contrast  is a \textit{local} sector of radius $\delta$ (in the $x$-space). The argument of \eqref{widehatW} is denoted by $w\in \mathbb C$ and we refer to the $w$-space as the Borel-plane. Functions in the Borel-plane  (like \eqref{widehatW})  are in this paper given a hat.

 \begin{figure}[h!]
\begin{center}
{\includegraphics[width=.5\textwidth]{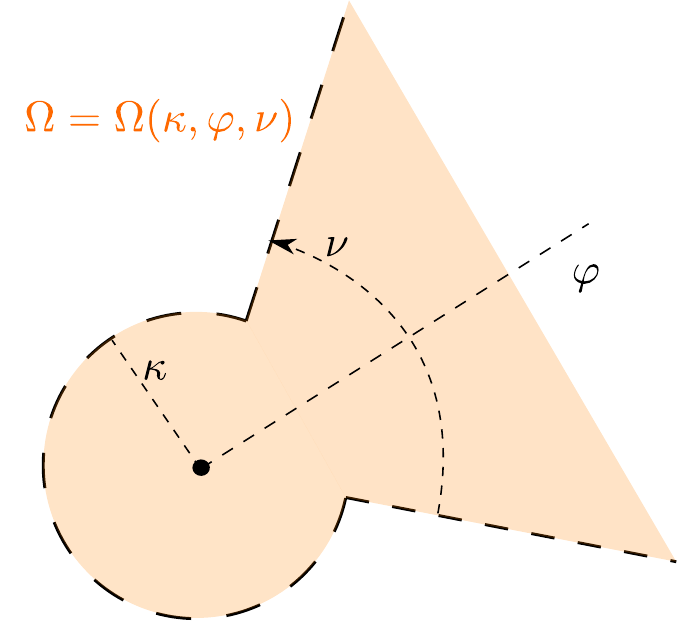}}
% \subfigure[]{\includegraphics[width=.45\textwidth]{psi0.pdf}}
\end{center}
\caption{Illustration of the set $\Omega\subset \mathbb C$. Here $\varphi\in \mathbb T=\mathbb R/(2\pi\mathbb Z)$ defines the direction of the sector, while $\nu>0$ the opening. It will suffice for our purposes to take $\varphi\in \{0,\pi\}$. Finally, $\kappa>0$ is the radius of the disc. }
\figlab{Omega}
% Remark: If $U$ smooth for $|u|<\nu$, the compact manifolds lie inside $|u|<\nu$ for large $n$
\end{figure}

It will suffice for our purposes to take
\begin{align}\eqlab{thetacond}
 \varphi \in \{0,\pi\};
\end{align}
$S(\varphi,\nu)$ will therefore be centered along the positive real axis (corresponding to $\varphi=0$) or the negative real axis (corresponding to $\varphi=\pi)$), similar to the local sectors $S^\pm$ in \eqref{Spm}.

Following \cite{bonckaert2008a}, we  define $\mathcal G_\eta$ for any $\eta>0$ as the space of analytic functions \eqref{widehatW} with
\begin{align*}
 \Vert \widehat W\Vert_{\eta}:=\sup_{w \in \Omega}\bigg\{\vert \widehat W(w)\vert  \e^{-\eta \vert w\vert}(1+ \eta^2 \vert w\vert^2)\bigg\}<\infty.
\end{align*}
% Notice that this norm is a slight modification of the norms used in \cite{}.
$\mathcal G_\eta$ with the norm $\Vert \cdot\Vert_\eta$ is a Banach space for all $\eta>0$. Since
\begin{align}\eqlab{prop}
\e^{-p}(1+p^2),
\end{align} is a decreasing function of $p\ge 0$, we have that $\mathcal G_{\eta'}\subset \mathcal G_{\eta}$ for all $\eta'>\eta$, see \cite[Proposition 2]{bonckaert2008a}.
This property of \eqref{prop} also implies that the constant function $1$ has unit norm:
\begin{align}\eqlab{unit}
 \Vert 1\Vert_\eta = 1,
\end{align}
for all $\eta>0$. The factor $\e^{-\eta\vert w\vert }$ in the norm $\Vert \cdot\Vert_\eta$ means that functions in $\mathcal G_\eta$ can be Laplace transformed (see \lemmaref{laplace} below). On the other hand, the factor $(1+ \eta^2 \vert w\vert^2)$ ensures that the convolution product:
\begin{align}\eqlab{convolution0}
 (\widehat W\star \widehat Z)(w) := w \int_0^1 \widehat W(ws) \widehat Z(w(1-s))ds,\quad \widehat W,\widehat Z:\Omega\to \mathbb C,
\end{align}
defines a bounded bilinear operator with a ``small'' operator norm (with respect to $\eta\to \infty$). %Indeed, \cite{} shows the following important result.

\begin{lemma}\lemmalab{convolution1}
\cite[Proposition 4]{bonckaert2008a} The following holds
\begin{align*}
 \Vert (\widehat W\star \widehat Z)\Vert_\eta \le 4\pi \eta^{-1} \Vert \widehat W\Vert_\eta \Vert \widehat Z\Vert_\eta,
\end{align*}
for all $\widehat W,\widehat Z\in \mathcal G_\eta$ and all $\eta>0$.
\end{lemma}
\begin{proof}
It will be useful for us to repeat the proof from \cite{bonckaert2008a} (we will use it in the proof of \lemmaref{wWstarZest} below): By the definition \eqref{convolution0}, we have the following estimate:
% The proof is elementary so we have decided to repeat here. We have
 \begin{align*}
 \vert (\widehat W\star \widehat Z)(w)\vert&\le
 \e^{\eta \vert w\vert} (1+\eta^2 \vert w\vert^2)^{-1} \Vert \widehat W\Vert_\eta \Vert \widehat Z\Vert_\eta   \int_0^{1} \frac{(1+\eta^2 \vert w\vert^2)}{(1+\eta^2 \vert w\vert^2 s^2)(1+\eta^2 \vert w\vert^2 (1-s)^2)} \vert w\vert ds\\
 &\le 8\e^{\eta \vert w\vert} (1+\eta^2 \vert w\vert^2)^{-1}  \Vert \widehat W\Vert_\eta \Vert \widehat Z\Vert_\eta\Vert   \int_0^{\frac12} \frac{1}{1+\eta^2 \vert w\vert^2 s^2}\vert w \vert ds,
%  &=8\e^{\eta \vert w\vert} (1+\eta^2 \vert w\vert^2)^{-1} \vert w\vert \Vert \widehat W\Vert_\eta \Vert \widehat Z\Vert_\eta \int_0^{\frac12} \frac{1}{1+\eta^2 \vert w\vert^2 s^2}\vert w \vert ds,
 \end{align*}
%  \begin{align*}
%  \vert (\widehat W\star \widehat Z)(w)\vert&\le
%  \e^{\eta \vert w\vert} (1+\eta^4 \vert w\vert^4)^{-1} \vert w\vert \sum_{\beta\in \mathbb Z} \Vert \widehat W_\beta\Vert_\eta \Vert \widehat Z_{\alpha-\beta}\Vert   \int_0^{1} \frac{\vert w\vert s \vert w\vert (1-s)(1+\eta^4 \vert w\vert^4)}{(1+\eta^4 \vert w\vert^4 s^4)(1+\eta^4 \vert w\vert^4 (1-s)^4} ds\\
%  &\le 2^{3}\e^{\eta \vert w\vert} (1+\eta^2 \vert w\vert^2)^{-1} \vert w\vert \sum_{\beta\in \mathbb Z} \Vert \widehat W_\beta\Vert_\eta \Vert \widehat Z_{\alpha-\beta}\Vert   \int_0^{\frac12} \frac{\vert w\vert s}{1+\eta^2 \vert w\vert^4 s^4}\vert w \vert ds\\
%  &=2^{3}\e^{\eta \vert w\vert} (1+\eta^2 \vert w\vert^2)^{-1} \vert w\vert \sum_{\beta\in \mathbb Z} \Vert \widehat W_\beta\Vert_\eta \Vert \widehat Z_{\alpha-\beta}\Vert   \int_0^{\frac12} \frac{\vert w\vert s}{1+\eta^2 \vert w\vert^4 s^4}\vert w \vert ds,
%  \end{align*}
 using the symmetry with respect to $s\mapsto 1-s$ and that $1+\eta^2 \vert w\vert^2 (1-s)^2\ge \frac{1}{4}(1+\eta^2 \vert w\vert^2)$ for all $0\le s\le \frac12$.
 We now apply the substitution $p=\eta\vert w\vert s$:
 \begin{align*}
  \int_0^{\frac12} \frac{1}{1+\eta^2 \vert w\vert^2 s^2}\vert w \vert ds\le  \eta^{-1} \int_0^{\infty} \frac{1}{1+s^2}ds  = \frac{\pi}{2}\eta^{-1},
  \end{align*}
 so that
 \begin{align*}
 \Vert  (\widehat W\star \widehat Z)\Vert_{\eta} \le 4 \pi \eta^{-1} \Vert \widehat W\Vert_\eta \Vert \widehat Z\Vert_\eta.
\end{align*}
\end{proof}

In the present paper, we will also need the following result.
\begin{lemma}\lemmalab{wWstarZest}
 Let $\widehat W,\widehat Z\in \mathcal G_\eta$ and denote the function $$w\mapsto w \widehat W(w), \quad w\in \Omega,$$ by
 \begin{align}\eqlab{wW}
 w.\widehat{ W}.
 \end{align}
 Then
 \begin{align*}
  w\mapsto w^{-1} ((w.\widehat W)\star Z)(w), \quad w\in \Omega,
 \end{align*}
 belongs to $\mathcal G_\eta$,
with the following uniform bound
%  and let $w.\widehat W$ denote the function $w\mapsto w \widehat W(w)$, $w\in \Omega$.
\begin{align*}
  \Vert w\mapsto w^{-1} ((w.\widehat W)\star Z)(w) \Vert_\eta \le 4 \pi \eta^{-1} \Vert \widehat W \Vert_\eta \Vert \widehat Z \Vert_\eta,\quad \forall\,\eta>0,
\end{align*}
% for some $C>0$ independent of $\eta$.
\end{lemma}
\begin{proof}
 The proof is similar to \lemmaref{convolution1}. Indeed we have
 \begin{align*}
 &\vert (w.\widehat W\star \widehat Z)(w)\vert\\
 &\le \e^{\eta \vert w\vert} (1+\eta^2 \vert w\vert^2)^{-1}\Vert \widehat W\Vert_\eta \Vert \widehat Z\Vert_\eta   \int_0^{1} \frac{ \vert w\vert s (1+\eta^2 \vert w\vert^2) }{(1+\eta^2 \vert w\vert^2 s^2)(1+\eta^2 \vert w\vert^2 (1-s)^2)}  \vert w\vert ds\\
 &\le \vert w\vert \left( \e^{\eta \vert w\vert} (1+\eta^2 \vert w\vert^2)^{-1}\Vert \widehat W\Vert_\eta \Vert \widehat Z\Vert_\eta   \int_0^{1} \frac{ 1+\eta^2 \vert w\vert^2 }{(1+\eta^2 \vert w\vert^2 s^2)(1+\eta^2 \vert w\vert^2 (1-s)^2)}  \vert w\vert ds\right),
 \end{align*}
 using $s\le 1$ in the final inequality.
  The last bracket is estimated above:
  \begin{align*}
   \vert w^{-1} (w.\widehat W\star \widehat Z)(w)\vert \e^{-\eta \vert w\vert} (1+\eta^2 \vert w\vert^2)\le 4 \pi \eta^{-1}  \Vert \widehat W\Vert_\eta \Vert \widehat Z\Vert_\eta,
  \end{align*}
for all $w\in \Omega$, $\eta>0$.
\end{proof}

\subsection{The Laplace transform}

Functions $\widehat W\in \mathcal G_\eta$ can be Laplace transformed:
\begin{align*}
 \mathcal L^{\varphi}[\widehat W](x)  = \int_0^{\infty\e^{i\varphi}} \widehat W(w) \e^{-w/x}dw,
\end{align*}
where the integration is along the ray defined by $w=r\e^{i\varphi}$, $r\ge 0$.
In particular, we have the following result:
\begin{lemma}\lemmalab{laplace}
\cite[Proposition 3]{bonckaert2008a} The Laplace transform $\mathcal L^{\varphi}$ defines (by analytic continuation) a linear continuous operator from $\mathcal G_\eta$ to the set of analytic bounded
functions on a sector $S(\varphi,\pi + \nu)\cap B_\kappa$ (equipped with the sup-norm) where
\begin{align}
 \kappa<\frac12 \eta^{-1} \sin \frac{\nu}{2}.\eqlab{deltacond}
\end{align}
In particular, when \eqref{deltacond} holds then we have the following bound:
\begin{align}\eqlab{Laplacebound}
  \vert \mathcal L^{\varphi}[\widehat W](x) \vert\le \eta^{-1} \Vert \widehat W \Vert_\eta\quad \forall\,\widehat W\in \mathcal G_\eta,
\end{align}
for all $x\in S(\varphi,\pi + \nu)\cap B_\kappa$.
\end{lemma}
\begin{proof}
 We refer to \cite[Proposition 3]{bonckaert2008a} for the proof. For comparison, we emphasize that we denote their $(\mu,\epsilon/2,\nu)$ by $(\eta,\nu,\delta)$ in the present manuscript. The bound \eqref{Laplacebound} follows from the last estimate of the proof of \cite[Proposition 3]{bonckaert2008a} using \eqref{deltacond}.
\end{proof}
We recall that
\begin{align}\eqlab{laplacewn}
 \mathcal L^{\varphi}\left[w\mapsto w^{\beta-1}\right](x)=(\beta-1)! x^\beta\quad \forall\,\beta\in \mathbb N,\,\varphi\in \mathbb T.
\end{align}

\begin{remark}\remlab{borellaplace}
There is a different perspective of Borel-Laplace (resurgence \cite{sauzin2015}). Indeed, one could start from a Gevrey-1 series:
\begin{align}\eqlab{series}
 W(x) = \sum_{\beta\in \mathbb N} W_\beta x^\beta,\quad \vert W_\beta\vert\le A \tau^\beta (\beta-1)!
\end{align}
Then the Borel transform is given by
\begin{align*}
 \widehat W(w):=\sum_{\beta\in \mathbb N} W_{\beta} \frac{w^{\beta-1}}{(\beta-1)!},
\end{align*}
which is now convergent on the disc $\vert w\vert<\tau^{-1}$. Then the series \eqref{series} is said to be $1$-summable in the direction $\varphi$ if (a) the analytic function $\widehat W$ can be endlessly continued to $\widehat W^\varphi$ along the ray defined by $w=r\e^{i\varphi}$, $r\ge 0$, and (b) it is of at most exponential growth of order $1$ along this ray, i.e.
\begin{align*}
 \e^{-\eta \vert w\vert} \vert \widehat W^\varphi(w)\vert =\mathcal O(1) \quad \mbox{for}\quad w\to \infty\e^{i\varphi},
\end{align*}
for some $\eta>0$ large enough, see \cite{balser2000a,de2020a}.
Therefore, when the series is $1$-summable, then the extension $\widehat W^\varphi$ of  $\widehat W$ can be Laplace-transformed $\mathcal L^{\varphi}[\widehat W^\varphi]$ to obtain an analytic function $W^\varphi$ of $x$ in some local sector (like $S^\pm$ in \eqref{Spm}), see \lemmaref{laplace}.

In this paper, our viewpoint is slightly different. Instead of trying to establish the Gevrey-property of the asymptotic series first and then perform the continuation of the Borel transform, we achieve this more indirectly by following \cite{bonckaert2008a}. The idea is to set up an appropriate nonlinear equation in $\mathcal G_\eta$ for the Borel transform $\widehat W=\widehat W^\varphi$ of our unknown function. Importantly, this equation in the Borel plane is perturbative by \lemmaref{convolution1} for $\eta\gg 1$ and it can therefore (easily!) be solved by a fixed-point argument. Finally, the desired solution is obtained through the Laplace transform (which is well-defined by \lemmaref{laplace}). In this way, we also obtain the Gevrey-$1$ asymptotic series in a more indirect way: We can write $\widehat W$ as a convergent series $\widehat W(w) = \sum_{\beta\in \mathbb N} \widehat W_\beta w^{\beta-1}$, $\vert W_\beta\vert\le A \tau^\beta$, $A>0,\tau>0$, near the origin (this is the motivation for including the open ball $B_\kappa$ in $\Omega$, see \eqref{Omega}). We then obtain the Gevrey-$1$ asymptotic series by applying the formal Laplace transformation (recall \eqref{laplacewn}):
\begin{align*}
 W(x) \sim_1 \sum_{\alpha \in \mathbb N} \widehat W_\beta(\beta-1)! x^\beta,\quad x\in S(\varphi,\pi+\nu)\cap B_\kappa.
\end{align*}
% for $x\in .$
% along a local sector centered around $\varphi$ and with opening greather than $\pi$.

To transform the equation \eqref{main} into an equation in the Borel plane for the Borel-transform $\widehat{\mathbf y}$ of $\mathbf y$, we will use the following basic properties of the Laplace transform:
\begin{align}
 \mathcal L^{\varphi}[\widehat W\star \widehat Z] =\mathcal L^{\varphi}[\widehat W]\mathcal L^{\varphi}[\widehat Z],\eqlab{prop1}
\end{align}
and
\begin{align}
 x^2 \frac{d}{dx} \mathcal L^{\varphi}[\widehat W](x) = \mathcal L^{\varphi}[w.\widehat W](x).\eqlab{prop2}
\end{align}
together with \lemmaref{final}. See \secref{equations} for further details.
% $\widehat W$ can be continued
% Now, $\widehat W\in \mathcal G_\eta$ has a convergent power series expansion at the origin:
% \begin{align*}
%  \widehat W (w) = \sum_{\beta \in \mathbb N_0} \widehat W_n w^n,\quad \vert \widehat W_n\vert \le A B^n,
% \end{align*}
% with $0<B^{-1}<\delta$. Since $\mathcal L^{\varphi}(w^n)(x) = n! x^{n+1}$, we obtain an asymptotic Gevrey-1 series of  $\mathcal L^{\varphi}(\widehat W)(x)$.
%
% Gevrey-1, $1$-summable series, resurgence.
\end{remark}

In the following section, we consider Fourier series with coefficients in $\mathcal G_\eta$.

\subsection{The Banach space $\mathcal G_\eta\{\e^{i\theta}\}$}
  Next, we consider analytic functions:
  \begin{align}\eqlab{mathcalWtheta}
   \widehat W\,:\,\Omega \times \mathbb T_\xi \to \mathbb C,\quad \mathbb T_\xi := \mathbb C/(2\pi \mathbb Z),
  \end{align}
% Here $\mathbb T_\xi$ denotes $\mathbb T=\mathbb R/(2\pi\mathbb Z)$ extended into the complex plane such that $\theta\in \mathbb T_\xi$ if and only if $\operatorname{Re}(\theta)\in \mathbb R/(2\pi \mathbb Z)$ and $0\le \vert \operatorname{Im}\theta \vert<\xi$.
which we write as Fourier series:
% We write $\widehat W$ as a Fourier series:
\begin{align*}
 \widehat W (w,\theta ) = \sum_{\alpha\in \mathbb Z}\widehat W_\alpha(w) \e^{i\alpha \theta},
\end{align*}
and suppose that $\widehat W_\alpha\in \mathcal G_\eta$ for all $\alpha\in \mathbb Z$. In particular, we define $\mathcal G_\eta\{\e^{i\theta}\}$, $\eta>0$, as the space of analytic functions \eqref{mathcalWtheta} with
%
% Notice that the convolution is well-defined on $\mathcal G$.
% We now equip $\mathcal G$ with the $\eta$-dependent Banach norm
\begin{align}\eqlab{Wert}
 \Wert \widehat W\Wert_{\eta}:= \sup_{\theta \in \mathbb T_\xi} \left\{\sum_{\alpha\in \mathbb Z} \Vert \widehat W_\alpha\Vert_{\eta} \e^{-\alpha \operatorname{Im}\theta}\right\}<\infty.
\end{align}
$\mathcal G_\eta\{\e^{i\theta}\}$ with the norm $\Wert \cdot\Wert_\eta$ is also a Banach space for all $\eta>0$, with $\mathcal G_{\eta'}\{\e^{i\theta}\}\subset \mathcal G_\eta\{\e^{i\theta}\}$ for all $\eta'>\eta$.
\begin{remark}\remlab{normeqv}
Notice that
\begin{align*}
   \frac12 \sum_{\alpha\in \mathbb Z} \Vert \widehat W_\alpha\Vert_{\eta} \e^{\vert \alpha\vert \xi} \le \Wert \widehat W\Wert_{\eta}\le \sum_{\alpha\in \mathbb Z} \Vert \widehat W_\alpha\Vert_{\eta} \e^{\vert \alpha\vert \xi}.
\end{align*}
The upper bound is obvious. For the lower bound, we set $\operatorname{Im}\theta=\pm \xi$:
\begin{align*}
\sum_{-\alpha\in \mathbb N} \Vert \widehat W_\alpha\Vert_{\eta} \e^{-\alpha \xi} \le \sum_{\alpha\in \mathbb Z} \Vert \widehat W_\alpha\Vert_{\eta} \e^{-\alpha \xi}\le \Wert \widehat W\Wert_{\eta},\\
\sum_{\alpha\in \mathbb N_0} \Vert \widehat W_\alpha\Vert_{\eta} \e^{\alpha \xi} \le \sum_{\alpha\in \mathbb Z} \Vert \widehat W_\alpha\Vert_{\eta} \e^{\alpha \xi}\le \Wert \widehat W\Wert_{\eta},
\end{align*}
proving that
\begin{align*}
 \sum_{\alpha\in \mathbb Z} \Vert \widehat W_\alpha\Vert_{\eta} \e^{\vert \alpha\vert \xi}\le 2\Wert \widehat W\Wert_{\eta}
\end{align*}
It follows that $\widehat W\in \mathcal G_\eta\{\e^{i\theta}\}$ implies that
\begin{align*}
 \Vert \widehat W_\alpha \Vert_{\eta}\e^{\vert \alpha\vert \xi}\to 0 \quad \mbox{for $\alpha\to \infty$}.
\end{align*}
Although not essential, we prefer to work with the norm $\Wert \widehat W\Wert_\eta$ (rather than e.g. the equivalent norm $\sum_{\alpha\in \mathbb Z} \Vert \widehat W_\alpha\Vert_{\eta} \e^{\vert \alpha\vert \xi}$) as it performs well under multiplication (see \lemmaref{convolution2}).
\end{remark}
We Laplace transform the functions \eqref{mathcalWtheta}  term-wise, i.e.
\begin{align}\eqlab{laplacedefn}
 \mathcal L^{\varphi}[\widehat W][(x,\theta): = \sum_{\alpha\in\mathbb Z} \mathcal L^{\varphi} [\widehat W_\alpha](x) \e^{i\alpha \theta}.
\end{align}
We have the following:
\begin{lemma}\lemmalab{laplace1}
$\mathcal L^{\varphi}$ is a bounded linear operator from $\mathcal G_\eta\{\e^{i\theta}\}$ to
%  If $\widehat W\in \mathcal G_\eta\{\e^{i\theta}\}$ then this function is
the space of analytic defined on $(x,\theta)\in (S(\varphi,\pi+\nu)\cap B_\kappa)\times \mathbb T_\xi$ for $\kappa>0$ small enough, see \eqref{deltacond}. In particular, whenever \eqref{deltacond} holds then we have the following estimate:
\begin{align*}
  \vert \mathcal L^{\varphi}[\widehat W](x,\theta)\vert\le \eta^{-1} \Wert W\Wert_\eta,
\end{align*}
for all $(x,\theta)\in (S(\varphi,\pi+\nu)\cap B_\kappa)\times \mathbb T_\xi$.
\end{lemma}
\begin{proof}
By \lemmaref{laplace}, it suffices to show the bound. We find that
\begin{align*}
 \vert \mathcal L^{\varphi}[\widehat W](x,\theta)\vert &\le \eta^{-1} \sum_{\alpha\in\mathbb Z}  \Vert \widehat W_\alpha\Vert_\eta \e^{-\alpha \operatorname{Im}(\theta)}\le \eta^{-1} \Wert W\Wert_\eta,
\end{align*}
by \eqref{Laplacebound} and \eqref{laplacedefn}.
\end{proof}

Now, the product of Fourier series is given by ``convolution of the coefficients'' (Cauchy's product rule):
\begin{align}\eqlab{cauchy}
 \sum_{\alpha \in \mathbb Z}f_\alpha\e^{i\alpha \theta}\sum_{\beta \in \mathbb Z}g_\beta\e^{i\beta\theta} = \sum_{\alpha\in \mathbb Z}\left(\sum_{\beta\in \mathbb Z} f_\beta g_{\alpha-\beta}\right)\e^{i\alpha \theta},
\end{align}
and we therefore define the convolution on $\mathcal G_\eta\{\e^{i\theta}\}$ at the level of the coefficients in the following way:
\begin{align}\eqlab{convolution2}
 (\widehat W\star \widehat Z)(w,\theta):=\sum_{\alpha\in \mathbb Z} \left(\sum_{\beta \in \mathbb Z}  \widehat W_\beta \star \widehat Z_{\alpha-\beta}\right) \e^{i\alpha \theta}\quad \forall\,\widehat W,\,\widehat Z\in \mathcal G_\eta\{\e^{i\theta}\}.
\end{align}
In this way,
\begin{align}\eqlab{laplacefs}
\mathcal L^{\varphi}[\widehat W\star \widehat Z]= \mathcal L^{\varphi}[\widehat W]\mathcal L^{\varphi}[\widehat Z]\quad \forall\,\widehat W,\widehat Z\in \mathcal G_\eta\{\e^{i\theta}\}.
\end{align}

% We have:
\begin{lemma}\lemmalab{convolution2}
 The convolution defines a bounded bilinear operator $\mathcal G_\eta\{\e^{i\theta}\}\times \mathcal G_\eta\{\e^{i\theta}\}\to \mathcal G_\eta\{\e^{i\theta}\}$:
 \begin{align*}
 \Wert \widehat W\star \widehat Z\Wert_\eta \le 4 \pi \eta^{-1}\Wert \widehat W\Wert_\eta \Wert \widehat Z\Wert_\eta,
 \end{align*}
 for all $\widehat W,\widehat Z\in \mathcal G_\eta\{\e^{i\theta}\}$ and any $\eta>0$.
\end{lemma}
\begin{proof}
 By the definition \eqref{convolution2}, we have $\widehat W\star \widehat Z=\sum_{\alpha\in \mathbb Z}(\widehat W\star \widehat Z)_\alpha \e^{i\theta}$ with
\begin{align*}
 (\widehat W\star \widehat Z)_\alpha(w)=\sum_{\beta\in \mathbb Z} w \int_0^1 \widehat W_{\beta}(ws) \widehat Z_{\alpha-\beta}(w(1-s)) ds.
\end{align*}
From \lemmaref{convolution1}, we have
% \begin{align*}
%  \vert (\widehat W\star \widehat Z)_\alpha(w)\vert&\le \e^{\eta \vert w\vert} (1+\eta^4 \vert w\vert^4)^{-1} \vert w\vert  2\pi \eta^{-2}\sum_{\beta\in \mathbb Z}\Vert \widehat W_\beta\Vert_\eta \Vert \widehat Z_{\beta-\alpha}\Vert_\eta.
% %  \e^{\eta \vert w\vert} (1+\eta^4 \vert w\vert^4)^{-1} \vert w\vert \sum_{\beta\in \mathbb Z} \Vert \widehat W_\beta\Vert_\eta \Vert \widehat Z_{\alpha-\beta}\Vert   \int_0^{1} \frac{\vert w\vert s \vert w\vert (1-s)(1+\eta^4 \vert w\vert^4)}{(1+\eta^4 \vert w\vert^4 s^4)(1+\eta^4 \vert w\vert^4 (1-s)^4} ds\\
% %  &\le 2^{3}\e^{\eta \vert w\vert} (1+\eta^2 \vert w\vert^2)^{-1} \vert w\vert \sum_{\beta\in \mathbb Z} \Vert \widehat W_\beta\Vert_\eta \Vert \widehat Z_{\alpha-\beta}\Vert   \int_0^{\frac12} \frac{\vert w\vert s}{1+\eta^2 \vert w\vert^4 s^4}\vert w \vert ds\\
% %  &=2^{3}\e^{\eta \vert w\vert} (1+\eta^2 \vert w\vert^2)^{-1} \vert w\vert \sum_{\beta\in \mathbb Z} \Vert \widehat W_\beta\Vert_\eta \Vert \widehat Z_{\alpha-\beta}\Vert   \int_0^{\frac12} \frac{\vert w\vert s}{1+\eta^2 \vert w\vert^4 s^4}\vert w \vert ds,
%  \end{align*}
% %  using the symmetry with respect to $s\mapsto 1-s$.
%  We now apply the subsitution $p=\eta\vert w\vert s$:
%  \begin{align*}
%   \int_0^{\frac12} \frac{\vert w\vert s}{1+\eta^2 \vert w\vert^4 s^4}\vert w \vert ds\le  \eta^{-2} \int_0^{\infty} \frac{s}{1+s^4}ds  = \frac{\pi}{4}\eta^{-2},
%   \end{align*}
%  so that
 \begin{align*}
 \Vert  (\widehat W\star \widehat Z)_\alpha\Vert_{\eta} \le 4 \pi \eta^{-1} \sum_{\beta\in \mathbb Z} \Vert \widehat W_\beta\Vert_\eta \Vert \widehat Z_{\alpha-\beta}\Vert.
\end{align*}
This leads to the following estimate (upon using \eqref{cauchy}) for any $\theta\in \mathbb T_\xi$
\begin{align*}
 \sum_{\alpha\in \mathbb Z} \Vert (\widehat W\star \widehat Z)_\alpha \Vert_\eta \e^{-\alpha \operatorname{Im}\theta} &\le 4 \pi \eta^{-1} \sum_{\alpha\in \mathbb Z}\sum_{\beta\in \mathbb Z} \Vert  W_\beta\Vert_\eta e^{-\beta \operatorname{Im}\theta} \Vert \widehat Z_{\alpha-\beta}\Vert e^{-(\alpha-\beta) \operatorname{Im}\theta} \\
 &= 4 \pi \eta^{-1} \sum_{\alpha'\in \mathbb Z}\Vert \widehat W_{\alpha'}\Vert_\eta e^{-\alpha' \operatorname{Im}\theta} \sum_{\beta'\in \mathbb Z}\Vert \widehat Z_{\beta'} \Vert_\eta e^{-\beta' \operatorname{Im}\theta}\\
 &\le4 \pi \eta^{-1} \Wert \widehat W\Wert_\eta \Wert \widehat Z\Wert_\eta.
\end{align*}
Here we have taken the sup over $\theta\in \mathbb T_{\xi}$ in the last inequality, recall \eqref{Wert}. The result then follows.

\end{proof}

\subsection{The Banach spaces $\mathcal G^n_\eta$ and $\mathcal G^n_\eta\{\e^{i\theta}\}$}
Finally, we define $\mathcal G^n_{\eta}$ ($\mathcal G^n_\eta\{\e^{i\theta}\}$) as the space of analytic functions $\widehat{\mathbf{W}}=(\widehat W^1,\ldots,\widehat W^n)\,:\Omega \to \mathbb C^n$ ( $\widehat{\mathbf{W}}=(\widehat W^1,\ldots,\widehat W^n)\,:\Omega\times \mathbb T_\xi \to \mathbb C^n$) with each component $\widehat{W}^j$, $j\in \{1,\ldots,n\}$, belonging to $\mathcal G_\eta$ ($\mathcal G_\eta\{\e^{i\theta}\}$, respectively). We then equip these spaces with the following Banach norms
\begin{align}
\Vert \widehat{\mathbf W}\Vert_\eta:=\sup_{j\in \{1,\ldots,n\}} \Vert \widehat W^j\Vert_\eta,\eqlab{Vertn}
\end{align}
and
\begin{align}
\Wert \widehat{\mathbf W}\Wert_\eta:=\sup_{j\in \{1,\ldots,n\}} \Wert \widehat W^j\Wert_\eta,\eqlab{Wertn}
\end{align}
 respectively. We use bold to denote symbols in $\mathbb C^n$ and therefore it should be clear from the context whether the norms are understood on $\mathbb C$ (as on the right hand sides of \eqref{Vertn} and \eqref{Wertn}) or on $\mathbb C^n$ (as on the left hand sides of \eqref{Vertn} and \eqref{Wertn}).
% \end{align*}
\subsection{Further results}
We will need a few additional results regarding $\mathcal G^n_\eta$, $n\ge 1$, from \cite{bonckaert2008a}, that we list as follows.
\begin{enumerate}
 \item \label{item10} Firstly, suppose that $Q(x) = \sum_{\beta\in \mathbb N} Q_\beta x^\beta$ is a convergent series, $\vert Q_\beta\vert \le A \tau^\beta$, $A,\tau>0$. Then the Borel-transform of $Q$:
 \begin{align*}
  \widehat Q(w): = \sum_{\beta\in \mathbb N} Q_\beta \frac{w^{\beta-1}}{(\beta-1)!},
 \end{align*}
is an entire function with at most exponential growth of order $1$ for $w\to \infty$:
\begin{align*}
 \vert \widehat Q(w)\vert \le A\e^{\tau w}\quad\forall\,w\in \mathbb C.
\end{align*}
In particular, $\Vert \widehat Q\Vert_\eta \le A$ for $\eta>0$ large enough, see \cite[pp. 313--314]{bonckaert2008a}.
\item \label{this} Next, we consider a convergent series $Q(x,\mathbf z)=\sum_{\beta\in \mathbb N^n} Q_\beta(x){ \mathbf z}^\beta$, ${\mathbf z}=(z^1,\ldots,z^n)\in B_\kappa^n\subset \mathbb C^n$, with $Q_\beta(x)=\sum_{\gamma=1}^\infty Q_{\beta,\gamma} x^\gamma$, $\vert Q_{\beta,\gamma}\vert\le A \tau^{\vert \beta\vert+\gamma}$, $\vert \beta\vert = \beta^1+\cdots +\beta^n$, $A>0,\tau>0$.  Here ${ \mathbf z}^\beta = (z^1)^{\beta^1}\cdots (z^n)^{\beta^n}$ as usual. We then define
\begin{align*}
 \widehat Q(\widehat{\mathbf z}) :=\sum_{\beta\in \mathbb N^n} \widehat Q_\beta\star \widehat{\mathbf{z}}^{\star \beta},
\end{align*}
where $\widehat Q_\beta(w) = \sum_{\gamma=1}^\infty Q_{\beta,\gamma} \frac{w^{\gamma-1}}{(\gamma-1)!}$ denotes the Borel transform of $Q_\beta$. By item (\ref{item10}), we have that $\Vert \widehat Q_\beta\Vert_\eta\le A \tau^{\vert \beta\vert}$.
Now, suppose that $\widehat{\mathbf  z}\in \mathcal G^n_\eta$ and that $\Vert \widehat{\mathbf  z}\Vert_\eta\le C_1$, for any fixed $C_1>0$. Then there is a $\eta_0(C_1)>0$ large enough such that
\begin{align*}
\Vert \widehat Q(\widehat{\mathbf{z}})\Vert_\eta\le 2^n A\tau,
%  \widehat F(\widehat{\mathbf{z}})\in \mathcal G_\eta,
\end{align*}
for all $\eta>\eta_0$,
see \cite[Proposition 5]{bonckaert2008a}. Moreover,
\begin{align*}
 \mathcal L[\widehat Q(\widehat{\mathbf z})](x) = Q(x,\mathcal L[\widehat{\mathbf z}](x)),
\end{align*}
for all $x\in S(\varphi,\pi + \nu)\cap B_\kappa$.
\item Finally, $\widehat{\mathbf z}\mapsto \widehat Q(\widehat{\mathbf z})$, $\Vert \widehat{\mathbf z}\Vert_\eta\le C_1$, $\eta>\eta_0$, is $C^1$, see \cite[Eq. (18)]{bonckaert2008a}, with
\begin{align*}
 D\widehat Q(\widehat{\mathbf z})(\widehat{\mathbf w}): = \sum_{i=1}^n \widehat{\left(\frac{\partial Q}{\partial z^j} \right)} \star \widehat w^j,
\end{align*}
where $\widehat{\mathbf q}=(\widehat q^1,\ldots,\widehat q^n)$, $\mathbf q=\mathbf z,\mathbf w$, with the following  bound
\begin{align*}
\Vert D\widehat Q(\widehat{\mathbf z})\Vert_\eta  \le \mathcal O(\eta^{-1}),
\end{align*}
for all $\Vert \widehat{\mathbf z}\Vert_\eta\le C_1$.

% for all $\eta>\eta_0$.
% for some $C_2>0$.
%which is well-defined .

\end{enumerate}

Consider $Q$ as in item (\ref{this}). In the present  paper, we then define $\widehat Q(\widehat{\mathbf z})$ analogously for $\widehat{\mathbf z}\in \mathcal G_\eta^n\{\e^{i\theta}\}$:
\begin{align}
 \widehat Q(\widehat{\mathbf z}) :=\sum_{\beta\in \mathbb N^n} \widehat Q_\beta(w)\star \widehat{\mathbf{z}}^{\star \beta}.\eqlab{widehatFalpha}
\end{align}
We similarly have that $\widehat Q(\widehat{\mathbf z}) \in \mathcal G_\eta\{\e^{i\theta}\}$ for all $\Wert \widehat{\mathbf z}\Wert_\eta\le C_1$ provided that $\eta>\eta_0$.
Indeed, by \lemmaref{convolution2}, we have that
\begin{align*}
 \Wert \widehat Q(\widehat{\mathbf z})\Wert_\eta&\le  \sum_{\beta\in \mathbb N^n} \Vert \widehat Q_\beta\Vert_\eta \left(4\pi \eta^{-1} \Wert \widehat{\mathbf z}\Wert\right)^{\vert \beta\vert}\\
 &\le  A\sum_{\beta\in \mathbb N^n} \left(4\pi \eta^{-1} \tau \Wert \widehat{\mathbf z}\Wert\right)^{\vert \beta\vert}\\
 &=A \left(\sum_{\vert \beta\vert=1}^\infty \left(4\pi \eta^{-1} \tau \Wert \widehat{\mathbf z}\Wert\right)^{\vert \beta\vert}\right)^n\\
 &\le  A (1-4\pi \eta^{-1}\tau C_1)^{-n},
\end{align*}
using that $\Wert \widehat Q_\beta\Wert_\eta = \Vert \widehat Q_\beta\Vert_\eta\le A\tau ^{\vert \beta\vert}$ since $\widehat Q_\beta$ is independent of $\theta$. For any $\eta>8\pi \tau C_1$, we therefore have that
\begin{align*}
\Wert \widehat Q(\widehat{\mathbf z})\Wert_\eta\le 2^{n} A.
\end{align*}
We also have that
\begin{align*}
 \mathcal L^{\varphi}( \widehat Q(\widehat{\mathbf z}))(x,\theta) = Q(x,\mathcal L^{\varphi}(\widehat{\mathbf z})(x,\theta)),
\end{align*}
for all $x\in S(\varphi,\pi + \nu)\cap B_\kappa$, $\theta\in \mathbb T_\xi$.
This follows from \lemmaref{laplace1} and \eqref{laplacefs}.

We now finally turn to the Borel transform
of analytic Fourier series
\begin{align*}
 Q(x,\mathbf z,\theta) = \sum_{\alpha\in \mathbb Z} Q_{\alpha}(x,\mathbf z) \e^{i\alpha \theta},
\end{align*}
with
\begin{align}\eqlab{cond1}
 Q_\alpha(x,\mathbf  z) = \sum_{\beta \in \mathbb N^n} Q_{\alpha,\beta} (x) \mathbf z^\beta,\quad
 Q_{\alpha,\beta} (x) = \sum_{\gamma=1}^n Q_{\alpha,\beta,\gamma} x^\gamma.
\end{align}
In further details, we suppose that $Q$ is analytic on $(x,\mathbf z,\theta) \in B_\kappa \times B_\kappa^n \times \mathbb T_\zeta$ with $\zeta>\xi$, $\kappa>0$.
We therefore suppose that
\begin{align}\eqlab{cond2}
 \vert Q_{\alpha,\beta,\gamma}\vert \le A \tau^{\vert\beta\vert +\gamma} \e^{-\zeta \vert \alpha\vert}\quad \forall\,\alpha \in \mathbb Z,\,\beta\in \mathbb N^n,\,\gamma\in \mathbb N,
\end{align}
with $A>0$, $\tau>0$ large enough, after possibly decreasing $\zeta>0$ and $\kappa>0$.
We then define the Borel transform of $Q$ as follows:
\begin{align}
 \widehat Q(\widehat{\mathbf z})(w,\theta):=\sum_{\beta\in \mathbb N^n} \left(\sum_{\alpha\in \mathbb Z} \widehat Q_{\alpha,\beta}(w) \e^{i\alpha \theta}\right) \star \widehat{\mathbf z}^{\star \beta},\eqlab{widehatF}
\end{align}
where $\widehat Q_{\alpha,\beta}$ is the Borel transform of $Q_{\alpha,\beta}$:
\begin{align}
 \widehat Q_{\alpha,\beta}(w):= \sum_{\gamma=1}^\infty Q_{\alpha,\beta,\gamma} \frac{w^{\gamma-1}}{(\gamma-1)!}.\eqlab{Qab}
\end{align}
% For completeness, we give a proof of the following essential result.
\begin{lemma}\lemmalab{final}
 Consider \eqref{widehatF} and \eqref{Qab}, satisfying \eqref{cond1} and \eqref{cond2}. Then
 \begin{align}
  \widehat{Q}\,:\, B_{C_1}^{\mathcal G} \to \mathcal G_\eta \{\e^{i\theta}\},
 \end{align}
where $B_{C_1}^{\mathcal G}$ is the open ball of radius $C_1>0$ in $\mathcal G^n_\eta\{\e^{i\theta}\}$, is well-defined and $C^1$ for all $\eta>\eta_0(C_1)>0$ with the following uniform bounds:
\begin{align}\eqlab{Qbound}
\begin{cases}
\sup_{\widehat{\mathbf z}\in B_{C_1}^{\mathcal G}} \Wert \widehat Q(\widehat{\mathbf z})\Wert_\eta \le \frac{2^{n+1} A}{1-\e^{-(\zeta-\xi)}},\\
\sup_{\widehat{\mathbf z}\in B_{C_1}^{\mathcal G}}\Vert D\widehat Q(\widehat{\mathbf z})\Vert =\mathcal O(\eta^{-1}).
\end{cases}
\end{align}
% operator norm:
% \begin{align*}
% \Wert D\widehat{Q}(\widehat{\mathbf z})\Wert = \mathcal O(\eta^{-1}).
% \end{align*}
Finally,
\begin{align*}
\mathcal L^{\varphi}(\widehat{Q}(\widehat{\mathbf z}))(x,\theta)= Q(x,\mathcal L^{\varphi}(\widehat{\mathbf z})(x,\theta),\theta),
\end{align*}
for all $(x,\theta)\in (S(\varphi,\theta+\nu)\cap B_\kappa)\times \mathbb T_\xi$ with $\kappa>0$ small enough.
% for all $\eta>\eta_0(C)>0$.
\end{lemma}
\begin{proof}
First, we notice that
\begin{align*}
 \Wert \sum_{\alpha\in \mathbb Z} \widehat Q_{\alpha,\beta}(w) \e^{i\alpha \theta}\Wert_\eta \le A \tau^{\vert\beta\vert} \sum_{\alpha \in\mathbb Z} \e^{-\vert \alpha\vert(\zeta-\xi)}\le\frac{ 2A \tau^{\vert \beta\vert} }{1-\e^{-(\zeta-\xi)}},
\end{align*}
recall item (\ref{item10}) above.
But then by \lemmaref{convolution2}, we find that
\begin{align*}
  \Wert \widehat Q(\widehat{\mathbf z})\Wert_\eta&\le\sum_{\beta\in \mathbb N^n} \frac{ 2A \tau^{\vert \beta\vert} }{1-\e^{-(\zeta-\xi)}} \left(4\pi \eta^{-1} \Wert \widehat{\mathbf z}\Wert_\eta\right)^{\vert \beta\vert}\\
  &\le \frac{2^{n+1} A}{1-\e^{-(\zeta-\xi)}},
\end{align*}
for all $\Wert \widehat{\mathbf z}\Wert_\eta\le C_1$ provided that $\eta>8\pi \tau C_1$. The fact that $\widehat Q$ is $C^1$ proceeds completely analagously, see also the text around  \cite[Eq. (18)]{bonckaert2008a}. %Moreover, \eqref{mathcalL} is identical to the proof of \cite[]{}
\end{proof}

\section{Equations in the Borel plane}\seclab{equations}
In this section, we now finally turn to solving \eqref{main}, repeated here for convenience
\begin{align}\eqlab{inveqn}
x^2 \mathbf y'_x (1+xF_0+x^2 F_1)+\mathbf y'_\theta+x \mathbf A \mathbf y = x^2 \mathbf G_1,
\end{align}
with $()'_z=\frac{\partial}{\partial z}$, $z=x,\theta$.  For this we will solve an associated equation (\eqref{boreleqn} below) for the Borel transform $\widehat{\mathbf y}\in \mathcal G_\eta^n\{\e^{i\theta}\}$ of $\mathbf y$. In turn, we obtain $\mathbf y$ upon application of the Laplace transform. For the following statement, we define
$1 \star \widehat{\mathbf Q}$
component-wise, i.e.
\begin{align*}
 1 \star \widehat{\mathbf Q} = (1 \star \widehat Q^1,\ldots,1\star \widehat Q^n),
\end{align*}
for any $\widehat{\mathbf Q}=(\widehat Q^1,\ldots,\widehat Q^n)\in \mathcal G_\eta^n\{\e^{i\theta}\}$.
\begin{lemma}\lemmalab{boreleqn}
Suppose that $\widehat{\mathbf y}\in \mathcal G_\eta^n\{\e^{i\theta}\}$ solves
\begin{align}\eqlab{boreleqn}
 w. \widehat{\mathbf y}+F_0 \star (w. \widehat{\mathbf y}) + \widehat{\mathbf y}'_\theta+1 \star \mathbf A \widehat{\mathbf y} = 1\star\left(- (w.\widehat{\mathbf y}) \star \widehat{(x. F_1)}+ \widehat{(x. \mathbf G_1)}\right);
\end{align}
recall the notations introduced in \eqref{wW} and  \eqref{widehatF}.
% where
%
 Then
\begin{align*}
\mathbf y(x,\theta) =\mathcal L^{\varphi} [\widehat{\mathbf y}](x,\theta)
\end{align*}
solves \eqref{inveqn} on $(x,\theta)\in (S(\varphi,\pi+\nu)\cap B_\kappa)\times \mathbb T_\xi$ with $\kappa>0$ satisfying \eqref{deltacond} and $0<\xi<\zeta$.
\end{lemma}
\begin{proof}
 The result follows from \eqref{prop1}, \eqref{prop2}, \lemmaref{final} and the fact that the Borel transform of $x$ is the constant function $1$, recall \eqref{laplacewn}. %The central observation is the following:
\end{proof}

In the following, we will solve \eqref{boreleqn} using the contraction mapping theorem in $\mathcal G_\eta^n\{\e^{\i\theta}\}$ using \lemmaref{final} with $\eta>0$ large enough (in the same fashion as \cite{bonckaert2008a}).
% %
% The idea is to Borel transform this equation to obtain an equation for $\widehat{\mathbf y}$ in the Borel plane. The aim is then to solve this equation using the contraction mapping theorem in $\mathcal G_\eta^n\{\e^{i\theta}\}$ taking $\eta>0$ large enough, recall \lemmaref{final}. In this way, we obtain a solution $\mathbf y = \mathcal L^{\varphi}(\widehat{\mathbf y})$ of \eqref{main} by taking the Laplace transformation.

% % Recall that we look for $\mathbf y = \mathbf y(x,\theta)$ of \eqref{mainvf}, which we repeat here convinience:
% \begin{align*}
%  \dot x &=  x^2 (1+xF_0+ x^2 F_1(x,\mathbf y,\theta)),\\
%  \dot{\mathbf y} & =  x\left( -\mathbf A {\mathbf y} +x \mathbf G_1(x,{\mathbf y},\theta)\right),\\
%  \dot \theta &=1,
% \end{align*}
% with $b\ne 0$, $F_0\in \mathbb C$, and
% \begin{align*}
%  \mathbf A=  \operatorname{diag}(\lambda^1,\ldots,\lambda^n),\quad \forall\,j\in \{1,\ldots,n\}\,:\,\lambda^j>0.
% \end{align*}
% Recall also that $$F_1:B_\tau(0)\times B_{\tau}^n(0)\times \mathbb T_\zeta\rightarrow \mathbb C,\quad \mathbf G_1:B_\tau(0)\times B_\tau^n(0)\times \mathbb T_\zeta\to \mathbb C^n,$$ are assumed to be analytic functions. The invariant manifolds of
%
%
%  The invariant manifolds satisfy the PDE:

  \subsection{Auxiliary equation}
%   We now consider the auxiliary equation
%   \begin{align*}
%    -bx^2 \mathbf y'_x -b  xF_0  x^2 \mathbf y'_x+\mathbf y'_\theta-b \mathbf A\mathbf y = b \mathbf H,
%   \end{align*}
% for $\mathbf y(x,\theta)$ for a given analytic function $\mathbf H:(S(\varphi,\pi+\nu)\cap B_\delta)\times \mathbb T_\xi \to \mathbb C^n$. In fact, we will suppose that
Before we turn to solving \eqref{boreleqn}, we first consider an auxiliary equation
  \begin{align}\eqlab{aux}
   w .\widehat{\mathbf y}+ F_0 \star (w. \widehat{\mathbf y}) + \widehat{\mathbf y}'_\theta+1 \star \textbf{A} \widehat{\mathbf y} = 1\star \widehat{\mathbf H},
  \end{align}
  with $\widehat{\mathbf H}=w.\widehat{\mathbf W}+\widehat{\mathbf Z}$ where
  \begin{align}\eqlab{condH}
  \widehat{\mathbf Q}\in \mathcal G^n_\eta\{\e^{i\theta}\},\quad \widehat{\mathbf Q}=\widehat{\mathbf W},\widehat{\mathbf Z}.
  \end{align}
  We consider this as an equation for $\widehat{\mathbf y}\in \mathcal G^n_\eta\{\e^{i\theta}\}$.

  To solve \eqref{aux} for $\widehat{\mathbf y}$, we insert $$\widehat{\mathbf Q} = \sum_{\alpha\in \mathbb Z} \widehat{\mathbf Q}_\alpha\e^{i\alpha\theta}, \quad \widehat{\mathbf Q}=\widehat{\mathbf y},\widehat{\mathbf H},\widehat{\mathbf W},\widehat{\mathbf Z}.$$
%
%   and insert
%   \begin{align*}
%    \widehat{\mathbf y} = \sum_{\alpha\in \mathbb Z} \widehat{\mathbf y}_\alpha \e^{i\alpha \theta}.
%   \end{align*}
This gives
\begin{align}
 w. \widehat{\mathbf y}_0 + F_0 \star (w. \widehat{\mathbf y}_0) + 1\star \mathbf  A \widehat{\mathbf y}_0 = 1\star  \widehat{\mathbf H}_0,\quad \widehat{\mathbf  H}_0 =w.\widehat{\mathbf W}_{0}+\widehat{\mathbf Z}_{0},\eqlab{alpha0eqn}
\end{align}
for $\alpha=0$ and
\begin{align}
 w.\widehat{\mathbf y}_\alpha +F_0 \star (w. \widehat{\mathbf y}_\alpha) +i\alpha \widehat{\mathbf y}_\alpha + 1\star \mathbf A \widehat{\mathbf y}_\alpha = 1\star \widehat{\mathbf H}_\alpha,\quad \widehat{\mathbf H}_\alpha=w.\widehat{\mathbf W}_{\alpha}+\mathbf Z_{\alpha},\eqlab{alphaneq0eqn}
\end{align}
for $\alpha\ne 0$.

\subsubsection{The case $\alpha=0$}
We first focus on $\alpha=0$ and write $\widehat{\mathbf y}_0 = (\widehat{y}_0^1,\ldots,\widehat{ y}_0^n)$. Then by differentiating \eqref{alpha0eqn} with respect to $w$, and using that $(1\star Z)'(w) = Z(w)$ (cf. \eqref{convolution0}), we find that
\begin{align*}
 w \frac{d}{dw}\widehat{y}_0^j = -(\lambda^j+1) \widehat{y}_0^j - F_0 w \widehat{y}_0^j+H_0^j,\quad j\in \{1,\ldots,n\}.
\end{align*}
Recall that $F_0$ is a constant.
We consider the associated vector-field:
\begin{align*}
  \frac{d}{dt}\widehat{y}_0^j &= -(\lambda^j+1) \widehat{y}_0^j - F_0 w \widehat{y}_0^j+H_0^j,\quad j\in \{1,\ldots,n\}\\
  \frac{d}{dt} w &= w.
  \end{align*}
  Here $(\widehat{y}_0^j,w)=(0,0)$ is a hyperbolic singularity, with eigenvalues $-(\lambda^j+1),1$ of the linearization.
 Indeed, given that $\lambda^j>0$, we have an unstable manifold as a graph over $w$. In fact, a simple calculation shows that it takes the following explicit graph form
\begin{align}\eqlab{y0j}
 \widehat{y}_0^j(w) = \int_0^1 H_0^j(ws) \e^{-F_0w(1-s)} s^{\lambda^j} ds,\quad H_0^j = w.W_{0}^j+Z_{0}^j,\quad j\in \{1,\ldots,n\}.
\end{align}
We recall (cf. \eqref{condH}) that $\widehat Q_{0}^j\in \mathcal G_\eta$, $\widehat Q=\widehat W,\widehat Z$, for all $j\in\{1,\ldots,n\}$.

\begin{lemma}\lemmalab{haty0est}
Let $\widehat{\mathbf y}_0=(\widehat y_0^1,\ldots,\widehat y_0^n)$ be given by \eqref{y0j}. Then there is a constant $C>0$ such that
 \begin{align}\eqlab{haty0est2}
  \Vert \widehat{\mathbf y}_0 \Vert_\eta\le C\left(\eta^{-1} \Vert \mathbf W_0\Vert_\eta + \Vert \mathbf Z_0 \Vert_\eta\right).
 \end{align}
for all $\eta>2\vert F\vert$.
\end{lemma}
\begin{proof}
 We estimate
\begin{align*}
 \vert \widehat{y}_0^j(w)\vert \e^{-\eta \vert w\vert} (1+\eta^2 \vert w\vert^2 ) \le \Vert W_0^j \Vert_\eta I_1+ \Vert Z_0^j \Vert_\eta I_2,
\end{align*}
with
\begin{align*}
 I_1 &=\int_0^1 \e^{-(\eta-\vert F_0\vert) \vert w\vert(1-s)}\frac{1+\eta^2 \vert w\vert^2}{1+\eta^2 \vert w\vert^2 s^2}\vert w\vert ds,\\
I_2 &=\int_0^1 \e^{-(\eta-\vert F_0\vert) \vert w\vert(1-s)}\frac{1+\eta^2 \vert w\vert^2}{1+\eta^2 \vert w\vert^2 s^2} ds.
\end{align*}
Here we have used that $\lambda^j>0$ for all $i\in\{1,\ldots,n\}$. We now split the integrals into $s\in [0,\frac12]$ and $s\in [\frac12,1]$. We first consider $I_1$. For $s\in [0,\frac12]$ we have
\begin{align*}
 \int_0^{\frac12} \e^{-(\eta-\vert F_0\vert)\vert w\vert (1-s)}\frac{1+\eta^2 \vert w\vert^2}{1+\eta^2 \vert w\vert^2 s^2}\vert w\vert ds&\le \frac12 \e^{-(\eta-\vert F_0\vert) \vert w\vert\frac12 }(1+\eta^2 \vert w\vert^2)\vert w\vert\\
 &\le \eta^{-1} \sup_{p\ge 0} \left\{p \e^{-p/2} (1+4p^2)\right\},
\end{align*}
for $\eta>2\vert F_0\vert$. Clearly, $$\sup_{p\ge 0} \left\{p \e^{-p/2} (1+4p^2)\right\}<\infty.$$ Using CAS, we find that $$\sup_{p\ge 0} \left\{p \e^{-p/2} (1+4p^2)\right\}\approx 43.32,$$ but the precise number will not be important to us. Next, for $s\in [\frac12,1]$ we find
\begin{align*}
 \int_{\frac12}^1 \e^{-(\eta-\vert F_0\vert)\vert w\vert (1-s)}\frac{1+\eta^2 \vert w\vert^2}{1+\eta^2 \vert w\vert^2 s^2}\vert w\vert ds&\le 4 \e^{-(\eta-\vert F_0\vert)\vert w\vert}  \int_{\frac12}^1 \e^{(\eta-\vert F_0\vert) \vert w\vert s} \vert w\vert ds\le \frac{8}{\eta},
\end{align*}
for any $\eta>2\vert F_0\vert$.
We therefore conclude that $I_1\le C\eta^{-1}$ for all $\eta>2\vert F_0\vert$ for some $C>0$ large enough. By estimating $I_2$ in the same way, we find that $I_2\le C$ for all $\eta>2\vert F_0\vert$.
%
% This directly leads to
% \begin{align*}
%  I_1\le C\eta^{-1},\quad I_2\le C,
% \end{align*}
% with $C>0$ large enough,
% for any $\eta>2\vert F\vert$.
This concludes the proof upon using that
\begin{align*}
 \Vert \widehat{\mathbf y}_0\Vert_\eta &:= \sup_{j\in \{1,\ldots,n\}} \Vert \widehat y_0^j\Vert_\eta\\
 &\le C\sup_{j\in \{1,\ldots,n\}}\left\{\eta^{-1}  \Vert W_0^j \Vert_\eta+ \Vert Z_0^j \Vert_\eta \right\}\\
 &=C \left(\eta^{-1} \Vert \mathbf W_0 \Vert_\eta+ \Vert \mathbf Z_0 \Vert_\eta \right).
\end{align*}

\end{proof}

\subsubsection{The case $\alpha\ne 0$}
Next for $\alpha\ne 0$, we consider the operator $\mathbf T_\alpha\,:\,\mathcal G^n_\eta\to \mathcal G^n_\eta$ defined by
\begin{align}\eqlab{fixpointyp}
 \mathbf T_\alpha[\widehat{\mathbf y}_\alpha](w) = \widehat{\mathbf y}_\alpha(w) +\frac{1}{w + i\alpha} \big(F_0 (1\star  w.\widehat{\mathbf y}_\alpha)(w) - (1\star \mathbf A \widehat{\mathbf y}_\alpha)(w)\big).
\end{align}
Then it is elementary to write \eqref{alphaneq0eqn} as
\begin{align}\eqlab{alphaneq0eqn2}
 \mathbf T_\alpha[\widehat{\mathbf y}_\alpha](w) = \frac{1}{w+i\alpha} (1\star \widehat{\mathbf H}_\alpha)(w).
\end{align}

\begin{lemma}\lemmalab{init_bound}
Suppose that $\nu\in (0,\pi)$, and that \eqref{thetacond} holds. Then for $\delta>0$ small enough, there is a constant $c_1>0$ such that
 \begin{align*}
  \vert w+i\alpha\vert \ge c_1(\vert w\vert+\vert \alpha\vert)\ge c_1>0,\quad \forall\,w\in \Omega,\,\alpha\in \mathbb Z\setminus\{0\}.
 \end{align*}

\end{lemma}
\begin{proof}
 Consider first $w\in S(\varphi,\nu)$ with $\varphi\in\{0,\pi\}$. Then
 \begin{align*}
  \vert w+i\alpha\vert^2 & \ge  \left(\vert w\vert^2 +\alpha^2 - 2\vert \alpha \vert \vert w\vert \sin (\nu/2)\right),
 \end{align*}
 upon writing $w=\vert w\vert\e^{\operatorname{arg}(w)}$ and minimizing the left hand side over ${\operatorname{arg}(w)}$. We now use Young's inequality:
 \begin{align}\eqlab{cs}
  -2 \vert \alpha \vert \vert w\vert  \ge -\vert \alpha  \vert^2 - \vert w\vert^2,
 \end{align}
so that
\begin{align*}
 \vert w+i\alpha\vert^2\ge  (1-\sin(\nu/2))(\vert w\vert^2 +\alpha^2)\ge\frac12( 1-\sin(\nu/2))(\vert w\vert +\vert \alpha\vert)^2.
\end{align*}
The final equality also follows from \eqref{cs}.
% The result then follows from application of Young's inequality again:
% \begin{align*}
%  \vert w'\vert^2 +\alpha^2\ge \frac12 (\vert w'\vert+\vert \alpha\vert)^2.
% \end{align*}
% We then obtain the desired inequality for $0<\nu\ll 1$ by applying Cauchy-Schwarz:
% \begin{align}
%  \vert w'\vert^2 +\alpha^2 \ge \frac12 (\vert w\vert +\vert \alpha \vert)^2.\eqlab{cs}
% \end{align}
Next for $w\in B_{\delta}(0)$, we similarly find that
\begin{align*}
  \vert  w+i\alpha\vert^2 & \ge  \left(\vert \alpha \vert-\vert w\vert \right)^2\ge\left(\vert \alpha \vert-\delta \right)^2\ge \frac14 \vert \alpha \vert^2\ge \frac18 (\vert \alpha \vert^2+\vert w\vert^2)\ge \frac{1}{16}(\vert w\vert+\vert\alpha\vert)^2,
\end{align*}
for any $0\le\vert w\vert<\delta<\frac12 \le \frac12 \vert \alpha \vert$. In conclusion, there is a constant $c_1>0$ for $\delta>0$ small enough such that
\begin{align*}
 \vert w+i \alpha\vert\ge  c_1 (\vert w\vert+\vert \alpha\vert)\ge c_1>0,
\end{align*}
for all $w\in \Omega$, $\alpha\in \mathbb Z\setminus\{0\}$. %We put $c_1:=c\min(\vert b\vert,1)$.

\end{proof}

\begin{lemma}\lemmalab{Talphainv}
 There is a constant $\eta_0>0$ independent of $\alpha\in \mathbb Z\setminus\{0\}$, such that linear operator $\mathbf T_\alpha\,:\,\mathcal G^n_\eta\to \mathcal G^n_\eta$ is an isomorphism for all $\eta>\eta_0$ large enough. In particular, the operator norm $\Vert \mathbf T_\alpha^{-1} \Vert$ for the inverse is uniformly bounded by $2$:
 \begin{align*}
 \Vert \mathbf T_\alpha^{-1} \Vert<  2\quad \forall\,\alpha\in \mathbb Z\setminus\{0\},
\end{align*}
for all $\eta>\eta_0$.
\end{lemma}
\begin{proof}
% We can now show that $\mathbf T_\alpha$ is well-defined.
By \lemmaref{wWstarZest} and \eqref{unit}, we have
\begin{align*}
 \vert (1\star w.\widehat{y}_\alpha^j)(w)\vert \e^{-\eta\vert w\vert}\left(1+\eta^2\vert w\vert^2\right) \le 4\pi \vert w\vert \eta^{-1} \Vert \widehat y_\alpha^j\Vert_\eta\quad \forall\,w\in \Omega.
\end{align*}
Moreover, by \lemmaref{init_bound}, we have that
\begin{align}\eqlab{init_bound_est}
 \frac{\vert w\vert}{\vert w+i\alpha \vert}\le c_1^{-1},\quad \frac{1}{\vert w+i\alpha \vert}\le c_1^{-1}\quad \forall\, \alpha\in \mathbb Z\setminus\{0\}, \,w\in \Omega.
\end{align}
This together with
\lemmaref{convolution1} then gives
\begin{align*}
 \Vert \mathbf T_\alpha(\widehat{\mathbf y}_\alpha)(w)-\widehat{\mathbf y}_\alpha\Vert_\eta\le C\eta^{-1} \Vert \widehat{\mathbf y}\Vert_\eta\le \frac12 \Vert \widehat{\mathbf y}\Vert_\eta,
\end{align*}
for all $\eta>2C$. Here $C>0$ is a sufficiently large constant independent of $\alpha\in \mathbb Z\setminus\{0\}$ and $\eta>0$. This shows that $\mathbf T_\alpha$ has a bounded inverse $\mathbf T_\alpha^{-1}$ with operator norm
\begin{align*}
 \Vert \mathbf T_\alpha^{-1} \Vert\le \sum_{\beta=0}^\infty (C\eta^{-1})^\beta< 2,
\end{align*}
for all $\eta>2C$.
\end{proof}
We then write the solution \eqref{alphaneq0eqn} as
\begin{align}\eqlab{yalpha}
\widehat{\mathbf y}_\alpha = \mathbf T_\alpha^{-1} \left[w\mapsto \frac{1}{w+i\alpha} (1\star \widehat{\mathbf H}_\alpha)(w)\right]\quad \forall\,\alpha\in \mathbb Z\setminus\{0\},
\end{align}
recall also \eqref{alphaneq0eqn2}.

\begin{lemma}
Consider $\widehat{\mathbf y}_\alpha$, $\alpha\in \mathbb Z\setminus\{0\}$, given by \eqref{yalpha} and let $\delta>0$ be small enough (recall \lemmaref{init_bound}). Then there exists a constant $C>0$ such  that
 \begin{align}
  \Vert \widehat{\mathbf y}_\alpha\Vert_\eta\le C\eta^{-1} \left( \Vert \widehat{\mathbf W}_\alpha\Vert_\eta+\Vert \widehat{\mathbf Z}_\alpha\Vert_\eta\right),\eqlab{yalphaest}
 \end{align}
for all $\alpha\in \mathbb Z\setminus\{0\}$ and all $\eta>0$ large enough.
\end{lemma}
\begin{proof}
The proof is similar to the proof of \lemmaref{Talphainv}. First by \lemmaref{convolution1} and \lemmaref{wWstarZest}, we have that
\begin{align*}
 \vert (1\star \widehat{H}_\alpha^j)(w)\vert \e^{-\eta\vert w\vert}(1+\eta^2 \vert w\vert^2)\le 4\pi  \eta^{-1}\left(\vert w\vert\Vert \widehat{W}_\alpha^j\Vert_\eta+\Vert \widehat Z_\alpha^j\Vert_\eta\right) \quad \forall\,w\in \Omega.
\end{align*}
We therefore conclude that
\begin{align*}
 \Vert \widehat{\mathbf y}_\alpha \Vert_\eta\le \Vert \mathbf T_\alpha^{-1}\Vert 4\pi c_1 ^{-1}\eta^{-1} \sup_{j\in \{1,\ldots,n\}} \left( \Vert \widehat{W}_\alpha^j\Vert_\eta+\Vert \widehat Z_\alpha^j\Vert_\eta\right)\le 8\pi c_1^{-1} \eta^{-1} \left( \Vert \widehat{\mathbf W}_\alpha\Vert_\eta+\Vert \widehat{\mathbf  Z}_\alpha\Vert_\eta\right),
\end{align*}
upon using \lemmaref{init_bound}, recall \eqref{init_bound_est}.

\end{proof}

We now summarize our findings on the auxiliary problem \eqref{aux}:
\begin{proposition}\proplab{aux}
The following holds for all $\eta>\eta_0$ with $\eta_0>0$ large enough:
The auxiliary equation \eqref{aux} has a unique solution
$$\widehat{\mathbf y}=:\mathbf T^{-1}[\widehat{\mathbf W},\widehat{\mathbf Z}] \in \mathcal G^n_\eta\{\e^{i\theta}\},$$ for any $\widehat{\mathbf W},\widehat{\mathbf Z}\in \mathcal G^n_\eta\{\e^{i\theta}\}$. In particular, the solution operator $\mathbf T^{-1}\,:\,\mathcal G^n_\eta\{\e^{i\theta}\} \times \mathcal G^n_\eta\{\e^{i\theta}\}\to \mathcal G^n_\eta\{\e^{i\theta}\}$ is linear and bounded, i.e. there is a (smallest) constant $\Wert \mathbf T^{-1}\Wert>0$ such that
 \begin{align*}
  \Wert \mathbf T^{-1}[\widehat{\mathbf W},\widehat{\mathbf Z}]\Wert_\eta \le \Wert \mathbf T^{-1}\Wert \left(\eta^{-1} \Wert \mathbf W\Wert_\eta+\Wert \mathbf Z\Wert_\eta\right)\quad \forall\,\widehat{\mathbf W},\widehat{\mathbf Z}\in \mathcal G^n_\eta\{\e^{i\theta}\}.
 \end{align*}

\end{proposition}
\begin{proof}
 By \eqref{haty0est2} and \eqref{yalphaest}, we have
 \begin{align*}
  \Wert \widehat{\mathbf y}\Wert_\eta &:= \sup_{\theta\in \mathbb T_\xi}\left\{ \sum_{\alpha\in \mathbb Z}\Vert \widehat{\mathbf y}_\alpha\Vert_\eta \e^{-\alpha \operatorname{Im}(\theta)}\right\}\\
  &\le C\sup_{\theta\in \mathbb T_\xi}\left\{ \sum_{\alpha\in \mathbb Z}\left(\eta^{-1} \Vert  {\mathbf W}_\alpha\Vert_\eta +\Vert {\mathbf Z}_\alpha\Vert_\eta\right) \e^{-\alpha \operatorname{Im}(\theta)}\right\}\\
  &=C \left(\eta^{-1} \Wert \mathbf W\Wert_\eta+\Wert \mathbf Z\Wert_\eta\right),
 \end{align*}
 for all $\eta>0$ large enough,
as claimed.
\end{proof}

\subsection{Completing the proof of \thmref{main}}
We are now ready to solve \eqref{boreleqn}, repeated here for convenience:
\begin{align}\eqlab{boreleqn2}
  w.\widehat{\mathbf y}+ F_0 \star (w. \widehat{\mathbf y}) + \widehat{\mathbf y}'_\theta+1 \star \mathbf A \widehat{\mathbf y} = 1\star\left( -(w.\widehat{\mathbf y}) \star \widehat{(x. F_1)}+\widehat{(x. \mathbf G_1)}\right),
\end{align}
for $\widehat{\mathbf y}\in \mathcal G^n_\eta\{\e^{i\theta}\}$. For this, we use the bounded solution operator $\mathbf T^{-1}$ of the auxiliary problem, recall \propref{aux}, to write \eqref{boreleqn2} in the fixed-point form:
\begin{align}
 \widehat{\mathbf y} = \mathbf T^{-1} \left[-w^{-1}.((w.\widehat{\mathbf y})\star \widehat{(x. F_1)}),\widehat{(x.\mathbf G_1)}\right].\eqlab{fixpoint}
\end{align}
We consider the open ball $B_{C_1}^{\mathcal G}\subset \mathcal G^n_\eta\{\e^{i\theta}\}$ of radius $C_1>0$ large enough centered at the origin.
Let $\widehat{\mathbf y}\mapsto \Theta(\widehat{\mathbf y})\in \mathcal G^n_\eta\{\e^{i\theta}\}$, be the mapping defined by the right hand side of \eqref{fixpoint}.

It follows from \lemmaref{final} that
$\widehat{\mathbf y}\mapsto \widehat{(x Q_1)}\in \mathcal G^n_\eta\{\e^{i\theta}\}$, $\widehat{\mathbf y}\in B_{C_1}^{\mathcal G}$, for ${Q}=F,{\mathbf G}$, are $C^1$ for all $\eta>\eta_0(C_1)$. In particular, \eqref{Qbound} holds for some $A>0,\zeta>\xi$ (independent of $C_1$). In turn, upon also using \lemmaref{wWstarZest} and \propref{aux}, we conclude that $\Theta\,:\,B_{C_1}^{\mathcal G}\to B_{C_1}^{\mathcal G}$ is a contraction for all $\eta\gg 1$. Therefore there is a unique solution of \eqref{fixpoint} in $B_{C_1}^{\mathcal G}$ for all such $\eta>0$. In turn, we have that the Laplace transform $\mathbf y^{\varphi} :=\mathcal L^{\varphi}[\widehat{\mathbf y}]$ solves \eqref{inveqn}, recall \lemmaref{boreleqn}. This completes the proof of \thmref{main}, setting $\mathbf y^+:=\mathbf y^0$ (corresponding to $\varphi=0$) and $\mathbf y^-:=\mathbf y^\pi$ (corresponding to $\varphi=\pi$).

\section{Study of the difference}\seclab{diff}
In this section, we prove \thmref{main2}. For this, we consider $\Delta{\mathbf y}(x,\theta):=\mathbf y^+(x,\theta)-\mathbf y^-(x,\theta)$ which is well-defined for
$$x\in \Delta S:=S^+\cap S^-,\quad \theta\in \mathbb T_\xi,$$ with asymptotic series $\Delta{\mathbf y}\sim_1 0$ for $x\to 0$ in $\Delta S$, uniformly with respect to $\theta\in \mathbb T_\xi$. Since $\mathbf y^+(x,\theta)$ and $\mathbf y^-(x,\theta)$ are both solutions of \eqref{main}, we obtain the following equation for the difference
\begin{align}\eqlab{Deltayyeqn}
 x^2 (1+xF_0+x^2 \widetilde F_1(x,\theta)) \Delta{\mathbf y}'_x + \Delta{\mathbf y}'_\theta + x \mathbf A \Delta{\mathbf y} =x^2 \widetilde{ \mathbf G}_1(x,\theta) \Delta{\mathbf y},
\end{align}
by using the mean-value theorem. Here
\begin{align}\eqlab{Qfourier}
%  \begin{cases}
Q(x,\theta) = \sum_{\alpha\in \mathbb Z} Q_{\alpha}(x)\e^{i\alpha\theta},
% \\ \widetilde{\mathbf G}_1(x,\theta) = \sum_{\alpha\in \mathbb Z} \widetilde{\mathbf G}_{1,\alpha}(x)\e^{i\alpha\theta},
%    \Delta {\mathbf y}(x,\theta) = \sum_{\alpha\in \mathbb Z} \Delta {\mathbf y}_{\alpha}(x)\e^{i\alpha\theta},
%     \end{cases}
\end{align}
where
\begin{align}\eqlab{Qnorm0}
 \Vert Q\Vert:=\sup_{(x,\theta) \in \Delta S\times \mathbb T_\xi} \sum_{\alpha\in \mathbb Z} \vert Q_\alpha(x)\vert \e^{-\alpha\operatorname{Im}\theta}<\infty,
%  \Vert \widetilde{\mathbf G}_{1,\alpha}(x,\dot) \Vert&:=\sup_{\theta\in \mathbb T_\xi} \sum_{\alpha\in \mathbb Z} \vert \widetilde{\mathbf G}_1(x)\vert \e^{-\alpha\operatorname{Im}\theta}<\infty,
\end{align}
with $\vert \cdot\vert$ denoting the norm in $\mathbb C^n$, $\mathbb C$, $\mathbb C^{n\times n}$ for $Q=\Delta \mathbf y,\widetilde F_1,\widetilde{\mathbf G}_1$ (along with their partial derivatives with respect to $x$),  respectively. This follows from \lemmaref{laplace1} and \thmref{main}.

In this section, we find it convenient to divide \eqref{Deltayyeqn} by the factor $1+xF_0+x^2 \widetilde F_1(x,\theta)$. This brings the equation into the following form
\begin{align}\eqlab{Deltayyeqn2}
 x^2 \Delta{\mathbf y}'_x =-  (1-xF_0)\Delta{\mathbf y}'_\theta - x \mathbf A \Delta{\mathbf y} +x^2 \left(\widetilde F_1(x,\theta) \Delta{\mathbf y}'_\theta+ \widetilde{ \mathbf G}_1(x,\theta) \Delta{\mathbf y}\right),
\end{align}
for some new (!) functions $\widetilde F_1\,:\,\Delta S\times \mathbb  T_\xi\to \mathbb C$, $\widetilde{\mathbf G}\,:\,\Delta S\times \mathbb T_\xi\to \mathbb C^{n\times n}$, that are also bounded in the norm \eqref{Qnorm}. Henceforth we drop the tildes.

To study \eqref{Deltayyeqn2}, we write $\Delta \mathbf y,F_1,{\mathbf G}_1$ as Fourier series:
\begin{align*}
%  \begin{cases}
Q(x,\theta) = \sum_{\alpha\in \mathbb Z} Q_{\alpha}(x)\e^{i\alpha\theta},\quad Q=\Delta \mathbf y,F_1,{\mathbf G}_1.
% \\ \widetilde{\mathbf G}_1(x,\theta) = \sum_{\alpha\in \mathbb Z} \widetilde{\mathbf G}_{1,\alpha}(x)\e^{i\alpha\theta},
%    \Delta {\mathbf y}(x,\theta) = \sum_{\alpha\in \mathbb Z} \Delta {\mathbf y}_{\alpha}(x)\e^{i\alpha\theta},
%     \end{cases}
\end{align*}
This leads to the following equations
\begin{align*}
x^2  \Delta{\mathbf y}'_\alpha &= \left(-i\alpha \left(1-xF_0\right) -x \mathbf A\right)
\Delta{\mathbf y}_\alpha \\
&+x^2\left(\sum_{\beta \in \mathbb Z} F_{1,\alpha-\beta}  i\beta \Delta{\mathbf y}_{\beta} +\sum_{\beta\in \mathbb Z} \mathbf G_{1,\alpha-\beta} \Delta{\mathbf y}_{\beta}\right),
\end{align*}
for all $\alpha\in \mathbb Z$.

Notice that $\Delta S$ is the union of two sectors of the complex plane centered along the directions defined by $\pm \frac{\pi}{2}$, each with opening $\chi\in (0,\pi)$. We therefore write $x= i p$ and consider $p$ real: $p\in [-p_0,p_0]$, $p_0>0$. This leads us to study the infinite dimensional system:
\begin{align*}
\frac{d}{dt} p &=-p^2,\\
 \frac{d}{dt} \Delta{\mathbf y}_\alpha &= \left(\alpha \left(1-i pF_0\right) +p \mathbf A\right)
\Delta{\mathbf y}_\alpha \\
&-i p^2 \left(\sum_{\beta \in \mathbb Z} F_{1,\alpha-\beta}  i\beta \Delta{\mathbf y}_{\beta} + \sum_{\beta\in \mathbb Z} \mathbf G_{1,\alpha-\beta} \Delta{\mathbf y}_{\beta}\right),\quad \alpha\in \mathbb Z.
\end{align*}
In fact, we restrict attention to $p\in (0,p_0]$, $0<p_0<\delta$, and are interested in bounded solutions for $t\ge 0$.
\subsection{Banach spaces} \seclab{banachdiff} For analytic functions
\begin{align}\eqlab{Wfourier}
W\,:\,\mathbb T_\xi \to \mathbb C,\quad  W(\theta) = \sum_{\alpha\in \mathbb Z} W_{\alpha}\e^{i\alpha\theta},
\end{align}
we define
\begin{align}\eqlab{Qnorm}
\begin{cases}
\Vert W\Vert_0 : =
 \sup_{\theta\in \mathbb T_\xi} \left\{\sum_{\alpha\in \mathbb Z} \vert W\vert \e^{-\alpha\operatorname{Im}\theta}\right\}, \\ \Vert W\Vert_{\frac12}:= \sup_{\theta\in \mathbb T_\xi} \left\{\sum_{\alpha\in \mathbb Z} \vert \alpha W\vert \e^{-\alpha\operatorname{Im}\theta}\right\},
 \end{cases}
 %  \Vert \widetilde{\mathbf G}_{1,\alpha}(x,\dot) \Vert&:=\sup_{\theta\in \mathbb T_\xi} \sum_{\alpha\in \mathbb Z} \vert \widetilde{\mathbf G}_1(x)\vert \e^{-\alpha\operatorname{Im}\theta}<\infty,
\end{align}
% where the $\sup$ is taken over ${(t,\theta) \in [0,\infty) \times \mathbb T_\xi}$.
and consider the Banach spaces $\mathcal F_0$ and $\mathcal F_1$ of Fourier series \eqref{Wfourier} with
\begin{align*}
\Vert W\Vert_0<\infty,\quad
\Vert W\Vert_1:=\max \{\Vert \cdot\Vert_0,\Vert \cdot \Vert_{\frac12}\}<\infty,
\end{align*}
% \end{align*}
respectively.
% for $Q=F_1,{\mathbf G}_1,\Delta \mathbf y$,
%  We equip $\mathcal F$ with the Banach norm
% \begin{align*}
%  \Vert \cdot\Vert:=\max \{\Vert \cdot\Vert_0,\Vert \cdot \Vert_1\}.
% \end{align*}
Similarly, we define $\mathcal F_k^n$, $k\in \{0,1\}$, as the Banach spaces of functions $\mathbf W=(W^1,\ldots,W^n)\,:\, \mathbb T_\xi \to \mathbb C^n$ where each component $W^j$, $j\in \{1,\ldots,n\}$, belongs to $\mathcal F_k$  and define $\Vert \mathbf W\Vert_k:=\max_{j\in \{1,\ldots,n\}} \Vert W^j\Vert_k$.

Finally, $\mathcal F_{1,\pm}^n\subset \mathcal F_1^n$ denotes the subset of $\mathcal F_1^n$ consisting of (one-sided) Fourier series
\begin{align}\eqlab{Wpm}
\mathbf W_- := \sum_{-\alpha \in  \mathbb N_0} \mathbf W_\alpha \e^{i\alpha\theta},\quad \mathbf W_+ := \sum_{\alpha \in  \mathbb N} \mathbf W_\alpha \e^{i\alpha\theta},
\end{align} respectively.
In particular, given a series $\mathbf W := \sum_{\alpha \in  \mathbb Z} \mathbf W_\alpha \e^{i\alpha\theta}\in \mathcal F_1^n$, then we define $$\widehat W_- := \sum_{-\alpha \in  \mathbb N_0} \mathbf W_\alpha \e^{i\alpha\theta}\in \mathcal F_{1,-}^n,\quad \mathbf W_+ := \sum_{\alpha \in  \mathbb N} \mathbf W_\alpha \e^{i\alpha\theta}\in \mathcal F_{1,+}^n.$$ We let $\operatorname{L}(\mathcal F_{1,-}^n,\mathcal F_{1,+}^n)$ denote the set of linear bounded operators from $\mathcal F_{1,-}^n$ to $\mathcal F_{1,+}^n$.

\subsection{Auxiliary equation}

 We now consider the auxiliary system:
 \begin{equation}\eqlab{auxdiff}
\begin{aligned}
\frac{d}{dt} p &=-p^2,\\
 \frac{d}{dt} \Delta{\mathbf y}_\alpha &= \left(\alpha \left(1-ip F_0\right) +p \mathbf A\right)
\Delta{\mathbf y}_\alpha + p^2 \mathbf H_\alpha,
\end{aligned}
\end{equation}
with $p\in (0,p_0]$, $0<p_0<\delta$ and where $\mathbf H = \sum_{\alpha\in \mathbb Z} \mathbf H_\alpha \e^{i\alpha \theta}\in C_b([0,\infty);\mathcal F_0^n)$, $\Vert \mathbf H\Vert_0<\infty$.

By the usual variation of constant formulation (see e.g. \cite[Lemma 5.2]{meiss2007a}), we find that bounded solutions of \eqref{auxdiff}
are given by
\begin{align}\eqlab{Deltayalphaaux}
 \footnotesize{\begin{cases}
\Delta \mathbf y_\alpha = \left(\frac{p(0)}{p(t)}\right)^{\mathbf A} \e^{\alpha \left(t- i F_0 \log\frac{p(0)}{p(t)}\right)}\Delta \mathbf y_\alpha(0) +\int_0^t  \left(\frac{p(s)}{p(t)}\right)^{\mathbf A} \e^{\alpha \left((t-s)- i F_0 \log\frac{p(s)}{p(t)}\right) } p(s)^2 \mathbf H_\alpha(s)ds & \alpha<0,\\
\Delta \mathbf y_0 =  \int_{\infty}^t  \left(\frac{p(s)}{p(t)}\right)^{\mathbf A} p(s)^2 \mathbf H_0(s) ds &\alpha=0,\\
  \Delta \mathbf y_\alpha =\int_{\infty}^t  \left(\frac{p(s)}{p(t)}\right)^{\mathbf A} \e^{\alpha \left((t-s)- i F_0 \log\frac{p(s)}{p(t)}\right)} p(s)^2 \mathbf H_\alpha(s) ds &\alpha>0,
 \end{cases}}
\end{align}
with $p'(t) = -p(t)^2$, $p(0)\in [0,p_0]$. In these expressions, %we have for simplicity suppressed the dependency of $s\in [0,\infty)$ on $\mathbf H_\alpha$ (writing $\mathbf H_\alpha$ instead of $\mathbf H_\alpha(s)$). Moreover,
we have defined
\begin{align*}
 q^{\mathbf A} := \operatorname{diag}(q^{\lambda^1},\dots,q^{\lambda^n}),\quad q\ge 0,
\end{align*}
and used that $\lambda^j>0$, recall \eqref{Adiag}.  This follows from simple calculations.

\begin{remark}The following lemma shows that $\Delta \mathbf y=\sum_{\alpha\in \mathbb Z} \Delta \mathbf y_\alpha \e^{i\alpha \theta}$ given by \eqref{Deltayalphaaux}  belongs to $C_b([0,\infty);\mathcal F_1^n)$. To emphasize this, we therefore write $\Delta \mathbf y(t,\theta)$ as $$\Delta \mathbf y(t)(\theta),$$ in the following. In this way, we can then write $$\Delta \mathbf y(t)\in \mathcal F_1^n,$$ for all $t\ge 0$ without confusion.
\end{remark}

%
% In particular, we have the following.
\begin{lemma}\lemmalab{mathbbT}
 Consider $\Delta \mathbf y = \sum_{\alpha\in \mathbb Z} \Delta \mathbf y_\alpha \e^{i\alpha \theta}$ given by \eqref{Deltayalphaaux} with $\mathbf H=\sum_{\alpha\in \mathbb Z} \mathbf H_\alpha \e^{i\alpha\theta}\in C_b([0,\infty);\mathcal F_0^n)$ and
  \begin{align*}
   \Delta \mathbf y_-(0)  = \sum_{-\alpha\in \mathbb N} \Delta \mathbf y_\alpha(0)\e^{i\alpha\theta}\in \mathcal F_{1,-}^n.
  \end{align*}
Then for any $0<p_0\ll 1$, we have \begin{align*}
\Delta \mathbf y\in C_b([0,\infty);\mathcal F_1^n),
 \end{align*}
in particular
\begin{align*}
 \Vert \Delta \mathbf y(t)\Vert_1 \le  \Vert \Delta \mathbf y_-(0)\Vert_1 + Cp_0^2 \Vert \mathbf H(t) \Vert_0\quad \forall\,t\ge 0,
 \end{align*}
with $C>0$ large enough.

\end{lemma}
\begin{proof}
 We first estimate $\Delta \mathbf y_0$ in \eqref{Deltayalphaaux}. Using $p'(t)=-p(t)^2$, we obtain
 \begin{align*}
  \vert \Delta \mathbf y_0^j(t)\vert \le (\lambda^j+1)^{-1} p_0 \sup_{t\ge 0}\vert H_0^j(t) \vert<\infty.
 \end{align*}
For $\alpha\in \mathbb N$, we find for any $s\in [t,\infty)$:
\begin{align*}
 &\log \left\vert \left(\frac{p(s)}{p(t)}\right)^{\lambda^j} \e^{\alpha \left((t-s)- i F_0 \log\frac{p(s)}{p(t)}\right)}\right\vert\\
 &=\alpha\left(p(t)^{-1}-p(s)^{-1}\right)\times \\
 &\left(1+\alpha^{-1}(\lambda^j+\alpha\operatorname{Re}(-i F_0))p(t) \frac{1}{p(t)/p(s)-1} \log (p(t)/p(s)) \right)\\
 &\le
\frac12 \alpha\left(p(t)^{-1}-p(s)^{-1}\right), \end{align*}
 for $p_0>0$ small enough.
 Here we have used (a):
  \begin{align*}
   t-s = \frac{1}{p(t)}-\frac{1}{p(s)}\le 0,
  \end{align*}
  which follows from $p'=-p^2$, (b): $0\le p(t)\le p_0$ for all $t\in [0,\infty)$,
  along with (c):
  \begin{align*}
   \frac{1}{z-1}\log z \le 1\quad \forall\,z\ge 1.
  \end{align*}
 This leads to
\begin{align*}
\vert \Delta y_\alpha^j(t)\vert &\le p_0^2 \int_{t}^\infty\e^{\frac12 \alpha(t-s)} ds \sup_{t\ge 0}\vert H_\alpha(t)\vert \\
&=\frac{2 p_0^2}{\vert \alpha\vert} \sup_{t\ge 0}\vert H_\alpha(t)\vert,
\end{align*}
and hence $\vert \Delta y_\alpha^j(t)\vert\le {2 p_0^2} \sup_{t\ge 0}\vert H_\alpha(t)\vert<\infty$
for all $j\in \{1,\ldots,n\}$.
For $-\alpha\in\mathbb N$, we similarly find that
\begin{align*}
  \vert \Delta y_\alpha^j(t)\vert&\le \e^{\frac12 \alpha t} \vert \Delta y_\alpha^j(0)\vert+\int_0^t \e^{\frac12 \alpha (t-s)} ds  p_0^2 \sup_{t\ge 0}\vert H_\alpha(t)\vert\\
  &\le  \vert \Delta y_\alpha^j(0)\vert-\frac{2}{\alpha}p_0^2 \sup_{t\ge 0}\vert H_\alpha(t)\vert,
  \end{align*}
  and therefore also $\vert \Delta y_\alpha^j(t)\vert \le  \vert \Delta y_\alpha^j(0)\vert+2p_0^2 \sup_{t\ge 0}\vert H_\alpha(t)\vert<\infty$.
The statements now easily follow. %completes the proof.
% \int_{\infty}^t  \left(\frac{p(s)}{p(t)}\right)^{\mathbf A} \e^{-\alpha \left((t-s)-b^{-1} i F_0 \log\frac{p(s)}{p(t)}\right)} p(s)^2 \mathbf H_\alpha ds
\end{proof}

The previous result shows that the linear operator $$\mathbf T(\Delta \mathbf y_-(0),p(0))\,:\, \mathbf H\mapsto \Delta \mathbf y=\mathbf T(\Delta \mathbf y_-(0),p(0))\left[\mathbf H\right],$$ defined by the right hand side of \eqref{Deltayalphaaux}, is bounded from $C_b([0,\infty);\mathcal F_0^n)$ to $C_b([0,\infty);\mathcal F_1^n)$.
Here the dependency with respect to $$ \Delta \mathbf y_-(0)= \sum_{-\alpha\in \mathbb N} \Delta \mathbf y_\alpha(0)\e^{i\alpha\theta}\in \mathcal F_{1,-}^n,$$ is affine whereas the dependency on $p(0)\in [0,p_0]$ is $C^\infty$. To show the latter, we just use that
\begin{align}
t = \frac{1}{p(t)}-\frac{1}{p(0)},\quad t\ge 0,\eqlab{teqn}
 \end{align}
and $$
\frac{p(0)}{p(t)}= \frac{p(0)^2}{1+p(0) t},\quad t\ge 0.$$

This leads to the following fixed-point formulation for \eqref{Deltayyeqn2} for $p\in [0,p_0]$, $0<p_0\ll 1$: $\Delta \mathbf y \in C_b([0,\infty);\mathcal F_1^n)$ if and only if
\begin{align*}
 \Delta \mathbf y = -i  \mathbf T(\Delta \mathbf y_-(0),p(0))\left[  F_1 \Delta{\mathbf y}'_\theta+{ \mathbf G}_1 \Delta{\mathbf y} \right].
\end{align*}
In the following result, we ask the reader to recall the notation in \secref{banachdiff} for the one-sided Fourier series, see \eqref{Wpm}.

\begin{proposition}\proplab{Wcs}
There exists a center-stable manifold $W^{cs}$ of \eqref{Deltayyeqn2} of the graph form
 \begin{align}
  \Delta \mathbf y_+ = p \widetilde{\mathbf M}^{cs}(p)\Delta \mathbf y_-,\quad p\in [0,p_0],\eqlab{centerstable}
 \end{align}
 with $\widetilde{\mathbf M}^{cs}\,:\,[0,p_0]\to \operatorname{L}(\mathcal F_{1,-}^n,\mathcal F_{1,+}^n)$ being $C^\infty$-smooth.
% with $\Delta \mathbf y_\pm\in \mathcal F_{1,\pm }^n$.
% Here
\end{proposition}
\begin{proof}
We first write
\begin{align*}
 F_1 \Delta{\mathbf y}'_\theta=\sum_{\alpha\in \mathbb Z} \left(\sum_{\beta\in \mathbb Z}F_{1,\alpha-\beta} i\beta\Delta{\mathbf y}_{\beta} \right)\e^{i\alpha\theta},\quad { \mathbf G}_1 \Delta{\mathbf y}  = \sum_{\alpha\in \mathbb Z} \left(\sum_{\beta\in \mathbb Z} \mathbf G_{1,\alpha-\beta} \Delta \mathbf y_{\beta}\right)\e^{i\alpha\theta}.
\end{align*}
Then by \eqref{Qnorm0} with $x=i p\in \Delta S$, we have
\begin{align*}
\Vert F_1 \Delta{\mathbf y}'_\theta\Vert_0\le
\sup_{\theta\in \mathbb T_\xi} \sum_{\alpha\in \mathbb Z} \left(\vert F_{1,\alpha-\beta} \vert \e^{-(\alpha-\beta)\operatorname{Im}(\theta)} \vert \beta\vert \vert \Delta{\mathbf y}_{\beta} \vert \e^{-\beta\operatorname{Im}(\theta)} \right)\le \Vert F_{1}\Vert_0 \Vert \Delta \mathbf y\Vert_1,
\end{align*}
and similarly
\begin{align*}
%  \begin{align*}
\Vert \mathbf G_1 \Delta{\mathbf y}\Vert_0\le \Vert  \mathbf  G_1\Vert_0 \Vert \Delta \mathbf y\Vert_0\le \Vert \mathbf  G_1\Vert_0 \Vert \Delta \mathbf y\Vert_1.
\end{align*}
It then follows from \lemmaref{mathbbT} that
% We then obtain the desired result from the fact that
\begin{align*}
\Delta \mathbf y \mapsto -i\mathbf T(\Delta \mathbf y_-(0),p(0))\left[  F_1 \Delta{\mathbf y}'_\theta+{ \mathbf G}_1 \Delta{\mathbf y} \right],
\end{align*}
has a unique fixpoint $\Delta \mathbf y^*(\Delta \mathbf y_-(0),p(0))\in C_b([0,\infty);\mathcal F_1^n)$ for all $p(0)\in (0,p_0]$, $p_0>0$ sufficiently small. We believe that this is clear enough. The desired manifold is then given by
\begin{align*}
\Delta \mathbf y_+ = \Delta \mathbf y^*_+(\Delta \mathbf y_-,p)(t=0),
\end{align*}
for any $\Delta \mathbf y_-\in \mathcal F_{1,-}^n$, $p\in [0,p_0]$.
The result then follows from the linearity of the problem.
\end{proof}

\subsection{Normal form on $W^{cs}$}
%To complete the proof of \thmref{main2},
In the following, we restrict to the center-stable manifold $W^{cs}$.
This gives the following system
\begin{equation}\eqlab{Wcs0}
\begin{aligned}
%  \begin{align*}
\frac{d}{dt} p &=-p^2,\\
 \frac{d}{dt} {\Delta{\mathbf y}}_\alpha &= \left(\alpha \left(1-i pF_0\right) +p \mathbf A\right)
{\Delta{\mathbf y}}_\alpha \\
% &-i p^2 \left(\sum_{\gamma \in \mathbb N_0} F_{1,\alpha-\gamma}  i\gamma p^2 [\widetilde{\mathbf M}^{cs}(p){\Delta \mathbf y}_-]_\gamma+\mathbf G_{1,\alpha-\gamma} p^2 [\widetilde{\mathbf M}^{cs}(p){\Delta \mathbf y}_-]_\gamma\right)\\
&-i p^2 \left(\sum_{-\beta \in \mathbb N}  \widetilde F_{1,\alpha-\beta}  i\beta {\Delta \mathbf y}_\beta+\sum_{-\beta \in \mathbb N}\widetilde{\mathbf G}_{1,\alpha-\beta} {\Delta \mathbf y}_\beta\right),\quad -\alpha\in \mathbb N,
% \end{align*}
\end{aligned}
\end{equation}
% and define $\widetilde{\Delta \mathbf y}_\alpha(t)$ for all $-\alpha\in \mathbb N$ by
% \begin{align*}
%  \Delta \mathbf  y_\alpha(t) = %\left(\frac{p(0)}{p(t)}\right)^{\mathbf A}
%  \e^{-\left(t-i F_0 \log\frac{p(0)}{p(t)}\right)}\widetilde{\Delta \mathbf y}_\alpha(t).
% \end{align*}
% This leads to the following system on $W^{cs}$:
% \begin{equation}\eqlab{Wcs}
% \begin{aligned}
% %  \begin{align*}
% \frac{d}{dt} p &=-p^2,\\
%  \frac{d}{dt} \widetilde{\Delta{\mathbf y}}_\alpha &= \left((\alpha+1) \left(1-i pF_0\right) +p \mathbf A\right)
% \widetilde{\Delta{\mathbf y}}_\alpha \\
% % &-i p^2 \left(\sum_{\gamma \in \mathbb N_0} F_{1,\alpha-\gamma}  i\gamma p^2 [\widetilde{\mathbf M}^{cs}(p)\widetilde{\Delta \mathbf y}_-]_\gamma+\mathbf G_{1,\alpha-\gamma} p^2 [\widetilde{\mathbf M}^{cs}(p)\widetilde{\Delta \mathbf y}_-]_\gamma\right)\\
% &-i p^2 \left(\sum_{-\gamma \in \mathbb N} \widetilde F_{1,\alpha-\gamma}  i\gamma \widetilde{\Delta \mathbf y}_\gamma+\sum_{-\gamma \in \mathbb N}\widetilde{\mathbf G}_{1,\alpha-\gamma} \widetilde{\Delta \mathbf y}_\gamma\right),\quad -\alpha\in \mathbb N,
% % \end{align*}
% \end{aligned}
% \end{equation}
for some new smooth functions
\begin{align*}
\widetilde Q_1=\sum_{\gamma \in \mathbb Z} \widetilde Q_{1,\gamma} \e^{i\gamma \theta},\quad Q=F,\mathbf G,
\end{align*} that are bounded (along with their partial derivatives with respect to $p$) in the norm \eqref{Qnorm0}.

We drop the tildes henceforth.
% where $[\cdot]_\alpha$ denotes the $\alpha$th Fourier coefficient. We see that the linear part of the equation for $\widetilde{\Delta \mathbf y}_{-1}$ vanishes. We therefore have a two-dimensional center-space, spanned by $(p,\widetilde{\Delta \mathbf y}_{-1})$.
Moreover, we let $\mathcal F_{1,-\star}^n\subset \mathcal F_{1,-}^n$ denote the closed subspace consisting of the series:
\begin{align*}
\mathbf W_{-\star}  = \sum_{-\alpha\in \mathbb N\setminus{\{1\}}}\mathbf W_\alpha \e^{i\alpha \theta}=\sum_{\alpha=-2}^{-\infty} \mathbf W_\alpha \e^{i\alpha \theta}.
\end{align*}
\begin{proposition}
Fix any $k\in \mathbb N$. Then we have the following for $p_0=p_0(k)>0$ small enough: On $W^{cs}$ there exist $C^k$-smooth functions
$\widetilde{\mathbf M}^{ss}:[0,p_0]\to \operatorname{L}(\mathcal F_{1,-\star}^n,\mathbb C^n)$, $\widetilde{\mathbf M}^c\,:\,[0,p_0]\to \operatorname{L}(\mathbb C^n,\mathcal F_{1,-\star}^n)$ such that the change of coordinates $(p,\widetilde{\Delta \mathbf y}_{-1},\widetilde{\Delta \mathbf y}_{-\star})\mapsto (p,\Delta \mathbf y_{-1},\Delta \mathbf y_{-\star}) $ defined by
\begin{align}\eqlab{ccfinaldiff}
 \begin{cases}
 {\Delta \mathbf y}_{-1} &= \widetilde{\Delta \mathbf y}_{-1} + p \widetilde{\mathbf M}^{ss}(p) \widetilde{\Delta \mathbf y}_{-\star},\\
 {\Delta \mathbf y}_{-\star} &=p \widetilde{\mathbf M}^{c}(p) \widetilde{\Delta \mathbf y}_{-1}+ \widetilde{\Delta \mathbf y}_{-\star},
 \end{cases}
\end{align}
brings \eqref{Wcs} into the following (diagonalized) normal form:
\begin{equation}\eqlab{nffibercoord}
\begin{aligned}
%  \begin{align*}
\frac{d}{dt} p &=-p^2,\\
\frac{d}{dt} \widetilde{\Delta{\mathbf y}}_{-1} &= \left(-(1-ip F_0) +p\mathbf A+p^2 \widetilde{\mathbf  H}_{-1}(p)\right) \widetilde{
\Delta{\mathbf y}}_{-1},\\
\frac{d}{dt} \widetilde{\Delta{\mathbf y}}_{\alpha } &= \left(\alpha (1-ip F_0) +p\mathbf A\right) \widetilde{
\Delta{\mathbf y}}_{\alpha}+ p^2 \widetilde{\mathbf  H}_{\alpha}(p)\widetilde{
\Delta{\mathbf y}}_{-\star},\quad -\alpha \in \mathbb N\setminus\{1\},
% % -i p \left(\widetilde F_{1,0}  i(-1)+\widetilde{\mathbf G}_{1,0} \right)+p^2 \widetilde H_{1,0}(p\right)\widetilde{
% % \Delta{\mathbf y}}_{-1}\\
% % &-ip^4 \left(\sum_{-\gamma \in \mathbb N} \widetilde F_{1,\alpha-\gamma}  i\gamma [\widetilde{\mathbf M}^c(p)]_\gamma +\sum_{-\gamma \in \mathbb N}\widetilde{\mathbf G}_{1,\alpha-\gamma}  [\widetilde{\mathbf M}^c(p)]_\gamma\right)\widetilde{
% % \Delta{\mathbf y}}_{-1}
%  \frac{d}{dt} \widetilde{\Delta{\mathbf y}}_\alpha &= \left((\alpha+1) \left(1-i pF_0\right) +p \mathbf A\right)
% \widetilde{\Delta{\mathbf y}}_\alpha \\
% % &-i p^2 \left(\sum_{\gamma \in \mathbb N_0} F_{1,\alpha-\gamma}  i\gamma p^2 [\widetilde{\mathbf M}^{cs}(p)\widetilde{\Delta \mathbf y}_-]_\gamma+\mathbf G_{1,\alpha-\gamma} p^2 [\widetilde{\mathbf M}^{cs}(p)\widetilde{\Delta \mathbf y}_-]_\gamma\right)\\
% &-i p^2 \left(\sum_{-\gamma \in \mathbb N} \widetilde F_{1,\alpha-\gamma}  i\gamma \widetilde{\Delta \mathbf y}_\gamma+\sum_{-\gamma \in \mathbb N}\widetilde{\mathbf G}_{1,\alpha-\gamma} \widetilde{\Delta \mathbf y}_\gamma\right),\quad -\alpha\in \mathbb N,
% \end{align*}
\end{aligned}
\end{equation}
Here $\widetilde{\mathbf  H}_{-1}:[0,p_0]\to \mathbb C^{n\times n}$ and $\widetilde{\mathbf  H}_{\alpha}:[0,p_0]\to \operatorname{L}(\mathcal F_{1,-\star}^n,\mathbb C^{n})$,  $\Vert \widetilde{\mathbf  H}_{\alpha}(p)\Vert =\mathcal O(\e^{-\vert \alpha\vert \xi})$, $-\alpha \in \mathbb N\setminus\{1\}$, are $C^k$-smooth functions.

\end{proposition}
\begin{proof}
%  The proof
% \end{proof}
The coordinates $(p,\widetilde{\Delta \mathbf y}_{-1},\widetilde{\Delta \mathbf y}_{-\star})$ are ``fiber coordinates'' associated with the stable foliation of a center manifold $W^c$ inside $W^{cs}$, see e.g. \cite{chow1991a,chow1988a}. To obtain $W^c$, we
first introduce the (intermediate) coordinates
${\Delta{\mathbf z}}_\alpha(t)$ for all $-\alpha\in \mathbb N$ defined by
\begin{align*}
 \Delta \mathbf  y_\alpha(t) = %\left(\frac{p(0)}{p(t)}\right)^{\mathbf A}
 \e^{-\left(t-i F_0 \log\frac{p(0)}{p(t)}\right)}{\Delta{\mathbf z}}_\alpha(t).
\end{align*}
This leads to the following system on $W^{cs}$:
\begin{equation}\eqlab{Wcs}
\begin{aligned}
%  \begin{align*}
\frac{d}{dt} p &=-p^2,\\
 \frac{d}{dt} {\Delta{\mathbf z}}_\alpha &= \left((\alpha+1) \left(1-i pF_0\right) +p \mathbf A\right)
{\Delta{\mathbf z}}_\alpha \\
% &-i p^2 \left(\sum_{\gamma \in \mathbb N_0} F_{1,\alpha-\gamma}  i\gamma p^2 [\widetilde{\mathbf M}^{cs}(p)\widetilde{\Delta \mathbf y}_-]_\gamma+\mathbf G_{1,\alpha-\gamma} p^2 [\widetilde{\mathbf M}^{cs}(p)\widetilde{\Delta \mathbf y}_-]_\gamma\right)\\
&-i p^2 \left(\sum_{-\beta \in \mathbb N}  F_{1,\alpha-\beta}  i\beta {\Delta{\mathbf z}}_\beta+\sum_{-\beta \in \mathbb N}{\mathbf G}_{1,\alpha-\beta} {\Delta{\mathbf z}}_\beta\right),\quad -\alpha\in \mathbb N.
% \end{align*}
\end{aligned}
\end{equation}
In these coordinates,  the linear part of the equation for ${\Delta{\mathbf z}}_{-1}$ now vanishes (we see this from setting $\alpha=-1$ in \eqref{Wcs}) and we have a two-dimensional center space spanned by $(p,{\Delta{\mathbf z}}_{-1})$.
We therefore proceed in the usual way, by multiplying the right hand side by $\chi(\frac{p}{p_0})$, $0<p_0<\delta$, with $\chi$ denoting a $C^\infty$-smooth cut-off function where $\operatorname{supp} \chi\in [0,2)$ and $\chi(q)=1$ for all $q\in [0,1]$, and look for solutions bounded in backward time, i.e. we work in $C_b((-\infty,0];\mathcal F_{1,-})$. This leads to $W^c$ as a  graph
\begin{align*}
%  \{\widetilde{ \Delta \mathbf y}_{-\alpha}\}_{\alpha=2}^{\infty} =
{\Delta{\mathbf z}}_{-\star} =
p \widetilde{\mathbf M}^{c}(p) {\Delta{\mathbf z}}_{-1},
\end{align*}
with $\widetilde{\mathbf M}^c\,:\,[0,p_0]\to\operatorname{L}(\mathbb C^n, \mathcal F_{1,-\star}^n)$ being $C^k$-smooth for $p_0=p_0(k)>0$ small enough.
In particular, the proof is identical to the construction of $W^{cs}$ in \propref{Wcs}, and further details are therefore left out for simplicity. See also \cite{chow1991a,chow1988a}.

% On $W^c$, we have the following reduced problem:
%  \begin{align*}
% \frac{d}{dt} p &=-p^2,\\
% \frac{d}{dt} \widetilde{\Delta{\mathbf y}}_{-1} &= p \left(\mathbf A+p \widetilde{\mathbf  H}_{-1}(p)\right) \widetilde{
% \Delta{\mathbf y}}_{-1},
% % -i p \left(\widetilde F_{1,0}  i(-1)+\widetilde{\mathbf G}_{1,0} \right)+p^2 \widetilde H_{1,0}(p\right)\widetilde{
% % \Delta{\mathbf y}}_{-1}\\
% % &-ip^4 \left(\sum_{-\gamma \in \mathbb N} \widetilde F_{1,\alpha-\gamma}  i\gamma [\widetilde{\mathbf M}^c(p)]_\gamma +\sum_{-\gamma \in \mathbb N}\widetilde{\mathbf G}_{1,\alpha-\gamma}  [\widetilde{\mathbf M}^c(p)]_\gamma\right)\widetilde{
% % \Delta{\mathbf y}}_{-1}
%  \end{align*}
% with $\widetilde H_{-1}:[0,p_0]\to \mathbb C$ being $C^k$-smooth.
% % Finally, we use the stable fiber projection on $W^{cs}$ to write the system in the following normal form:

Finally, we turn to the strong stable manifold $W^{ss}\subset W^{cs}$. This will give the stable foliation. To obtain this, we define the new (intermediate) coordinates ${\Delta \mathbf z}_\alpha$ for all $-\alpha\in \mathbb N$ defined by
\begin{align*}
%  \begin{align*}
 \Delta \mathbf  y_\alpha(t) = %\left(\frac{p(0)}{p(t)}\right)^{\mathbf A}
 \e^{-\frac32 \left(t-i F_0 \log\frac{p(0)}{p(t)}\right)}{\Delta \mathbf z}_\alpha(t).
% \end{align*}
\end{align*}
This gives
\begin{equation}\eqlab{Wss}
\begin{aligned}
%  \begin{align*}
\frac{d}{dt} p &=-p^2,\\
 \frac{d}{dt} {\Delta{\mathbf z}}_\alpha &= \left((\alpha+\tfrac32) \left(1-i pF_0\right) +p \mathbf A\right)
{\Delta{\mathbf z}}_\alpha \\
% &-i p^2 \left(\sum_{\gamma \in \mathbb N_0} F_{1,\alpha-\gamma}  i\gamma p^2 [\widetilde{\mathbf M}^{cs}(p)\widetilde{\Delta \mathbf y}_-]_\gamma+\mathbf G_{1,\alpha-\gamma} p^2 [\widetilde{\mathbf M}^{cs}(p)\widetilde{\Delta \mathbf y}_-]_\gamma\right)\\
&-i p^2 \left(\sum_{-\beta \in \mathbb N} F_{1,\alpha-\beta}  i\beta {\Delta \mathbf z}_\beta+\sum_{-\beta \in \mathbb N}{\mathbf G}_{1,\alpha-\beta} {\Delta \mathbf z}_\beta\right),\quad -\alpha\in \mathbb N,
% \end{align*}
\end{aligned}
\end{equation}
and we look for solutions bounded  in forward time, i.e. we work in $C_b([0,\infty);\mathcal F_{1,-}^n)$. This gives an invariant manifold of the graph form
\begin{align*}
 {\Delta{\mathbf z}}_{-1} = p \widetilde{\mathbf M}^{ss}(p){\Delta \mathbf z}_{-\star},
\end{align*}
with $\widetilde{\mathbf M}^{ss}:[0,p_0]\to \operatorname{L}(\mathcal F_{1,-\star}^n,\mathbb C^n)$ being $C^k$-smooth. The details are again similar to \propref{Wcs} and further details are therefore left out.  The statement then follows (upon using the linearity with respect to $\Delta \mathbf y_{-}$)  from a simple calculation.
\end{proof}
% This leads to the following:

% \begin{proof}
%  s
% \end{proof}
\subsection{Completing the proof of \thmref{main2}}
We now study the normal form \eqref{nffibercoord}. Let $\Phi(\cdot,\cdot)$ denote the state transition matrix associated with
\begin{align*}
  \mathbf{Q}'(p)  = -\mathbf H_{-1}(p) \mathbf Q(p),\quad p\in [0,p_0].
\end{align*}
$\Phi$ is locally well-defined.
Then a basic calculation shows that
\begin{align*}
 \widetilde{\Delta \mathbf y}_{-1}(t) = \e^{-\left( t-iF_0\log \frac{p(0)}{p(t)}\right)} \left(\frac{p(0)}{p(t)}\right)^{\mathbf A} \Phi(p(t),p(0)) \widetilde{\Delta \mathbf y}_{-1}(0).
\end{align*}
We can also easily estimate
\begin{align*}
 \Vert \widetilde{\Delta \mathbf y}_{-\star}(t)\Vert_0  = \e^{-\left( t-iF_0\log \frac{p(0)}{p(t)}\right)} \left(\frac{p(0)}{p(t)}\right)^{\mathbf A} \mathcal O(\e^{-\frac12 t})\Vert \widetilde{\Delta \mathbf y}_{-\star}(0)\Vert_0,
\end{align*}
for $p_0>0$ small enough and $t\ge 0$. (It is clearly possible to improve the $\mathcal O(\e^{-\frac12 t})$-remainder, but this will not be important here.)
In the following, we put $p(0)=p_0$.
% for any $c>0$ small enough.
Then by \eqref{teqn}, \eqref{centerstable} and \eqref{ccfinaldiff} we conclude that
\begin{equation}\eqlab{Deltayfinal}
\begin{aligned}
 \Delta \mathbf y_{-1}(t) &= \e^{-\frac{1}{p(t)}-iF_0\log p(t)}p(t)^{-\mathbf A} \left(\mathbf C_{-1}+\mathcal O(p(t))\right),\\
 \Delta \mathbf y_{-*}(t) &=\e^{-\frac{1}{p(t)}-iF_0\log p(t)}p(t)^{-\mathbf A} \mathcal O(p(t))\in \mathcal F_{1,-*}^n,\\
  \Delta \mathbf y_{+}(t) &=\e^{-\frac{1}{p(t)}-iF_0\log p(t)}p(t)^{-\mathbf A} \mathcal O(p(t))\in \mathcal F_{1,+}^n,
\end{aligned}
\end{equation}
for $t\to \infty$ ($p(t)\to 0^+$)
% upon using \eqref{teqn}
with
\begin{align*}
 \mathbf C_{-1}:=\e^{\frac{1}{p_0}+iF_0\log p_0}p_0^{\mathbf A} \Phi(0,p_0) \widetilde{\Delta \mathbf y}_{-1}(0)\in \mathbb C^n.
\end{align*}
Here we have used the group property of the state transition matrix to write:
\begin{align*}
 \Phi(p(t),p_0) = \Phi(p(t),0)\Phi(0,p_0) = \left(\mathbf{Id}+\mathcal O(p(t)\right) \Phi(0,p_0).
\end{align*}
%
% We define
% \begin{align*}
%  C_+ : =
% \end{align*}
This completes the proof of \thmref{main2}. Notice in particular that the remainder $\mathcal O(p(t))$ in \eqref{Deltayfinal} is $C^k$-smooth with respect to $p(t)\in [0,p_0(k)]$ for any $k$. But as $\mathbf y^\pm(ip,\theta)$ are unique and $C^\infty$ for $p\in (0,\delta)$, then one can easily deduce that they are in fact $C^\infty$ as claimed. For the estimate of $\mathbf R_\alpha$, we recall \remref{normeqv}.
\subsection*{Acknowledgement}
The author was funded by Danish Research Council (DFF) grant 4283-00014B.
\newpage
\bibliography{refs}

\begin{thebibliography}{10}

\bibitem{baldom2013a}
I.~Baldom\'a, O.~Castej\'on, and T.~M. Seara.
\newblock {Exponentially small heteroclinic breakdown in the generic Hopf-Zero
  singularity}.
\newblock {\em Journal of Dynamics and Differential Equations}, 25(2):335--392,
  2013.

\bibitem{baldoma2018a}
I.~Baldom\'a, O.~Castej\'on, and T.~M. Seara.
\newblock {Breakdown of a 2D heteroclinic connection in the Hopf-Zero
  singularity (II): The generic case}.
\newblock {\em Journal of Nonlinear Science}, 28(4):1489--1549, 2018.

\bibitem{MR4455359}
I.~Baldom\'a, M.~Giralt, and M.~Guardia.
\newblock Breakdown of homoclinic orbits to {$L_3$} in the {RPC}3{BP}({I}).
  {C}omplex singularities and the inner equation.
\newblock {\em Adv. Math.}, 408:Paper No. 108562, 64, 2022.

\bibitem{MR4621957}
I.~Baldom\'a, M.~Giralt, and M.~Guardia.
\newblock Breakdown of homoclinic orbits to {$L_3$} in the {RPC}3{BP}({II}).
  {A}n asymptotic formula.
\newblock {\em Adv. Math.}, 430:Paper No. 109218, 72, 2023.

\bibitem{MR4940205}
I.~Baldom\'a, M.~Guardia, and D.~E. Pelinovsky.
\newblock On a countable sequence of homoclinic orbits arising near a
  saddle-center point.
\newblock {\em Comm. Math. Phys.}, 406(9):215, 65, 2025.

\bibitem{baldoma2012a}
I.~Baldoma and P.~Martin.
\newblock The inner equation for generalized standard maps.
\newblock {\em Siam Journal on Applied Dynamical Systems}, 11(3):1062--1097,
  2012.

\bibitem{MR4743478}
I.~Baldom\'a, T.~M. Seara, and R.~Moreno.
\newblock Splitting of separatrices for rapid degenerate perturbations of the
  classical pendulum.
\newblock {\em SIAM J. Appl. Dyn. Syst.}, 23(2):1159--1198, 2024.

\bibitem{balser2000a}
W.~Balser.
\newblock {\em Formal Power Series and Linear Systems of Meromorphic Ordinary
  Differential Equations}.
\newblock Springer, 2000.

\bibitem{bonckaert2008a}
P.~Bonckaert and P.~De~Maesschalck.
\newblock Gevrey normal forms of vector fields with one zero eigenvalue.
\newblock {\em Journal of Mathematical Analysis and Applications},
  344(1):301--321, 2008.

\bibitem{chow1991a}
S.~N. Chow, X.~B. Lin, and K.~N. Lu.
\newblock Smooth invariant foliations in infinite dimensional spaces.
\newblock {\em Journal of Differential Equations}, 94(2):266--291, 1991.

\bibitem{chow1988a}
S.~N. Chow and K.~Lu.
\newblock {Invariant manifolds for flows in Banach spaces}.
\newblock {\em Journal of Differential Equations}, 74(2):285--317, 1988.

\bibitem{costin2009}
O.~Costin.
\newblock {\em Asymptotics and {B}orel summability}, volume 141 of {\em Chapman
  \& Hall/CRC Monographs and Surveys in Pure and Applied Mathematics}.
\newblock CRC Press, Boca Raton, FL, 2009.

\bibitem{de2021a}
P.~De~Maesschalck, F.~Dumortier, and R.~Roussarie.
\newblock {\em Canard Cycles: From Birth to Transition}, volume~73.
\newblock Springer Science and Business Media Deutschland GmbH, 2021.

\bibitem{de2020a}
P.~De~Maesschalck and K.~Kenens.
\newblock Gevrey asymptotic properties of slow manifolds.
\newblock {\em Nonlinearity}, 33(1):341--387, 2020.

\bibitem{DMS2016}
P.~De~Maesschalck and S.~Schecter.
\newblock The entry-exit function and geometric singular perturbation theory.
\newblock {\em J. Differ. Equations}, 260(8):6697--6715, 2016.

\bibitem{gaivao2011a}
J.~P. Gaiv{\~a}o and V.~Gelfreich.
\newblock Splitting of separatrices for the hamiltonian-hopf bifurcation with
  the swift-hohenberg equation as an example.
\newblock {\em Nonlinearity}, 24(3):677--698, 2011.

\bibitem{gelfreich2001a}
V.~Gelfreich and D.~Sauzin.
\newblock {Borel summation and splitting of separatrices for the H\'enon map}.
\newblock {\em Annales De L'institut Fourier}, 51(2):513--567, 2001.

\bibitem{gelfreich1999a}
V.~G. Gelfreich.
\newblock A proof of the exponentially small transversality of the separatrices
  for the standard map.
\newblock {\em Communications in Mathematical Physics}, 201(1):155--216, 1999.

\bibitem{glebsky1995a}
L.~Y. Glebsky and L.~M. Lerman.
\newblock {On small stationary localized solutions for the generalized 1-D
  Swift-Hohenberg equation}.
\newblock {\em Chaos}, 5(2):424--431, 1995.

\bibitem{MR4892796}
O.~M.~L. Gomide, M.~Guardia, T.~M. Seara, and C.~Zeng.
\newblock On small breathers of nonlinear {K}lein-{G}ordon equations via
  exponentially small homoclinic splitting.
\newblock {\em Invent. Math.}, 240(2):661--777, 2025.

\bibitem{Guckenheimer97}
J.~Guckenheimer and P.~Holmes.
\newblock {\em Nonlinear Oscillations, Dynamical Systems and Bifurcations of
  Vector Fields}.
\newblock Springer Verlag, 5th edition, 1997.

\bibitem{haragus2011a}
M.~Haragus and G.~Iooss.
\newblock {\em Local Bifurcations, Center Manifolds, and Normal Forms in
  Infinite-Dimensional Dynamical Systems}.
\newblock EDP Sciences, 2011.

\bibitem{hayes2016a}
M.~G. Hayes, T.~J. Kaper, P.~Szmolyan, and M.~Wechselberger.
\newblock Geometric desingularization of degenerate singularities in the
  presence of fast rotation: A new proof of known results for slow passage
  through hopf bifurcations.
\newblock {\em Indagationes Mathematicae}, 27(5):1184--1203, 2016.

\bibitem{jones_1995}
C.~K. R.~T. Jones.
\newblock {\em Geometric Singular Perturbation Theory, Lecture Notes in
  Mathematics, Dynamical Systems (Montecatini Terme)}.
\newblock Springer, Berlin, 1995.

\bibitem{Gucwa2009783}
I.~Kosiuk and P.~Szmolyan.
\newblock Geometric singular perturbation analysis of an autocatalator model.
\newblock {\em Discrete and Continuous Dynamical Systems - Series S},
  2(4):783--806, 2009.

\bibitem{kristiansen2024a}
K.~U. Kristiansen.
\newblock Blowup analysis of a hysteresis model based upon singular
  perturbations.
\newblock {\em Journal of Nonlinear Science}, 34(1):6, 2024.

\bibitem{kristiansen2025a}
K.~U. Kristiansen.
\newblock Improved gevrey-1 estimates of formal series expansions of center
  manifolds.
\newblock {\em Studies in Applied Mathematics}, 154(6):e70063, 2025.

\bibitem{bkt}
K.~U. Kristiansen.
\newblock {A geometric approach to exponentially small splitting: The generic
  zero-Hopf bifurcation of co-dimension two}.
\newblock {\em arXiv-preprint::2603.12103}, 2026.

\bibitem{new}
K.~U. Kristiansen.
\newblock {A geometric approach to exponentially small splitting: Zero-Hopf
  bifurcations of arbitrary co-dimension}.
\newblock {\em arXiv-preprint::2603.12115}, 2026.

\bibitem{MR4855745}
K.~U. Kristiansen and P.~Szmolyan.
\newblock Analytic weak-stable manifolds in unfoldings of saddle-nodes.
\newblock {\em Nonlinearity}, 38(2):025019, 70, 2025.

\bibitem{krupa_extending_2001}
M.~Krupa and P.~Szmolyan.
\newblock Extending geometric singular perturbation theory to nonhyperbolic
  points - fold and canard points in two dimensions.
\newblock {\em {SIAM} Journal on Mathematical Analysis}, 33(2):286--314, 2001.

\bibitem{kuehn2015a}
C.~Kuehn and P.~Szmolyan.
\newblock Multiscale geometry of the olsen model and non-classical relaxation
  oscillations.
\newblock {\em Journal of Nonlinear Science}, 25(3):583--629, 2015.

\bibitem{lazutkin2005a}
V.~F. Lazutkin.
\newblock {Splitting of separatrices for the Chirikov standard map}.
\newblock {\em Journal of Mathematical Sciences}, 128(2):2687--2705, 2005.

\bibitem{meiss2007a}
J.~D. Meiss.
\newblock {\em Differential dynamical systems}, volume~14.
\newblock Society for Industrial and Applied Mathematics, 2007.

\bibitem{MR4445442}
F.~Merle, P.~Rapha\"{e}l, I.~Rodnianski, and J.~Szeftel.
\newblock {On the implosion of a compressible fluid {I}: {S}mooth self-similar
  inviscid profiles}.
\newblock {\em Ann. of Math. (2)}, 196(2):567--778, 2022.

\bibitem{neishtadt1987a}
A.~I. Neishtadt.
\newblock Persistence of stability loss for dynamical bifurcations .1.
\newblock {\em Differential Equations}, 23(12):1385--1391, 1987.

\bibitem{neishtadt1988a}
A.~I. Neishtadt.
\newblock Persistence of stability loss for dynamical bifurcations .2.
\newblock {\em Differential Equations}, 24(2):171--176, 1988.

\bibitem{sauzin2015}
D.~Sauzin.
\newblock Nonlinear analysis with resurgent functions.
\newblock {\em Ann. Sci. \'{E}c. Norm. Sup\'{e}r. (4)}, 48(3):667--702, 2015.

\bibitem{yang1997a}
T.~S. Yang and T.~R. Akylas.
\newblock On asymmetric gravity-capillary solitary waves.
\newblock {\em Journal of Fluid Mechanics}, 330:215--232, 1997.

\end{thebibliography}
\bibliographystyle{plain}
\newpage
\appendix
\section{Formal series invariant manifolds}\applab{formalinvman}
We consider the formal system
\begin{align*}
 \dot x &= \lambda_1 x+Q(x,y),\\
 \dot y &=\lambda_2 y +P(x,y),
\end{align*}
with $\lambda_1\lambda_2<0$ and $Q,P\in \mathbb R[[x,y]]$ containing no constant nor linear terms. The origin is therefore a formal saddle.
In this appendix, we will prove the following:
\begin{lemma}\lemmalab{formalinvman}
 There exists formal stable and unstable manifolds of the form:
%  formal stable and unstable manifolds
 \begin{align*}
  y = m^s(x)\in x^2 \mathbb R[[x]],\quad x = m^u(y)\in y^2\mathbb R[[y]],
 \end{align*}
 respectively.
\end{lemma}
It clearly suffices to focus on the formal stable manifold only. Notice that the invariance is understood in the formal sense, i.e. $y=m^s(x)$ satisfies
\begin{align}
 (\lambda_1 x+Q(x,y)) \frac{dy}{dx} = \lambda_2 y + P(x,y),\quad y(0)=0, \,y'(0)=0.\eqlab{inv1}
\end{align}
as an equation for a formal series $y\in x^2 \mathbb R[[x]]$. By dividing through by $\lambda_1$, we see that it is without loss of generality to take
\begin{align}\eqlab{lambda12}\lambda_1=1, \quad \lambda_2=-\lambda<0.\end{align} To study \eqref{inv1}, we then apply the directional blowup defined by
\begin{align}
 y = x y_1.
\end{align}
Then $Q(x,y)=x^2 Q_1(x,y_1)$, $P(x,y)=x^2P_1(x,y_1)$, where $W_1\in \mathbb R[[x,y]]$, $W=Q,P$. This brings \eqref{inv1} into the following prepared normal form
\begin{align*}
 x \frac{dy_1}{dx} = -(\lambda +1)y_1 +x F(x,y_1),\quad y_1\in  x\mathbb R[[x]],
\end{align*}
with $F\in \mathbb R[[x,y_1]]$. This follows from a simple calculation. Consider first the auxiliary equation:
\begin{align}
 x \frac{dy_1}{dx} = -(\lambda+1) y_1 +x G(x),\quad y_1\in  x\mathbb R[[x]],\eqlab{y1eqn}
\end{align}
with $G=\sum_{\alpha=0}^\infty G_\alpha x^\alpha\in \mathbb R[[x]]$ given. Then a basic calculation shows that $y_1$ satisfies
\begin{align*}
 y_1 = \sum_{\alpha =1}^\infty \frac{G_{\alpha-1}}{\alpha+\lambda+1}x^\alpha = \int_0^1 t^{\lambda+1} x G(xt) dt,
\end{align*}
with the integration understood term-wise.
This leads to the following fixed-point formulation of \eqref{y1eqn}:
\begin{align}\eqlab{fixedy1}
 y_1 = \mathcal T[y_1](x):=x\int_0^1 t^{\lambda+1}  F(xt, y_1(xt)) dt.
\end{align}
Here $F(xt, y_1(xt))\in \mathbb R\{t\}[[x]]$ for any $y_1\in \mathbb R[[x]]$ and the integration is again understood term-wise. Due to the multiplication of $x$, the right hand side of \eqref{fixedy1} therefore defines a formal series in $x\mathbb R[[x]]$ for any $y\in x\mathbb R[[x]]$:
\begin{align}\eqlab{Tprop}
 \mathcal T[y_1]\in x\mathbb R[[x]]\quad \forall\,y_1\in \mathbb R[[x]].
\end{align}
We now follow \cite[Section 2.1]{de2020a} and equip $\mathbb R[[x]]$ with the metric
\begin{align*}
\operatorname{d}(F,G) = 2^{-K},\quad K =  \min\{\alpha \in \mathbb N_0 \,:\,F_\alpha-B_\alpha=0\},
\end{align*}
where $F=\sum_{\alpha=0}^\infty F_\alpha x^\alpha$, $G=\sum_{\alpha=0}^\infty G_\alpha x^\alpha$. $\mathbb R[[x]]$ is then a complete metric space. We have the following obvious properties of $\operatorname{d}$:
\begin{align}\eqlab{dprop}
 \operatorname{d}(F,G)=\operatorname{d}(F-G,0),\quad \operatorname{d}(FG,0)\le \operatorname{d}(F,0)\operatorname{d}(G,0).
\end{align}
Due to \eqref{Tprop}, we have
\begin{align*}
\operatorname{d}(\mathcal T[y_1],0)\le \frac12\quad \forall\,y_1\in \mathbb R[[x]],
\end{align*}
and similarly by using \eqref{dprop}:
\begin{align*}
\operatorname{d}(\mathcal T[y_1],\mathcal T[y_2])\le \frac12 \operatorname{d}(y_1,y_2)\quad \forall\,y_1,y_2\in \mathbb R[[xx]].
\end{align*}
It follows that $\mathcal T$ is a contraction on the closed set $x\mathbb R[[x]]$ where $\operatorname{d}(y_1,0)\le \frac12$.

\end{document}